\documentclass[twoside,a4paper]{amsart}

\usepackage{mathabx}

\usepackage{pdfsync}

\usepackage{xcolor}
\definecolor{webgreen}{rgb}{0,.5,0}
\definecolor{webbrown}{rgb}{.6,0,0}
\definecolor{RoyalBlue}{cmyk}{1, 0.50, 0, 0}
\usepackage[colorlinks=true, breaklinks=true, urlcolor=webbrown, linkcolor=RoyalBlue, citecolor=webgreen,backref=page]{hyperref}

\usepackage{courier} 
\normalfont
\usepackage{epsfig, graphicx, subfigure}
\usepackage{verbatim, setspace}
\usepackage{amsmath, amssymb}

\def\cal{\mathcal}
\let\Re=\undefined
\DeclareMathOperator{\Re}{Re}
\let\Im=\undefined
\DeclareMathOperator{\Im}{Im}

\newcommand{\dd}{\stackrel{\rm def}{=}}

\newcommand{\R}     {\mathbb{R}}

\renewcommand{\arg}{\mathrm{arg}}

\def\cal{\mathcal}
\let\Re=\undefined
\DeclareMathOperator{\Re}{Re}
\let\Im=\undefined
\DeclareMathOperator{\Im}{Im}

\def\ge{\geqslant}
\def\le{\leqslant}

\newtheorem{theorem}{Theorem}[section]

\newtheorem{corollary}[theorem]{Corollary}
\newtheorem{lemma}[theorem]{Lemma}

\theoremstyle{remark}

\numberwithin{equation}{section}

\begin{document}

\title[ Spatial asymptotics of Green's function for elliptic operators and applications \ldots]{Spatial asymptotics of Green's function for elliptic operators and applications: a.c. spectral type, wave operators for wave equation}
\author{Sergey A. Denisov}

\thanks{
The work done in the last section
 of the paper was supported by a grant of the Russian Science Foundation (project RScF-14-21-00025)
 and  research
conducted in the rest of the paper was supported by the grants NSF-DMS-1464479,  NSF DMS-1764245, and Van Vleck Professorship Research Award. Author gratefully acknowledges hospitality of IHES where part of this work was done.
}

\address{
\begin{flushleft}
Sergey Denisov: denissov@wisc.edu\\\vspace{0.1cm}
University of Wisconsin--Madison\\  Department of Mathematics\\
480 Lincoln Dr., Madison, WI, 53706,
USA\vspace{0.1cm}\\and\\\vspace{0.1cm}
Keldysh Institute of Applied Mathematics, Russian Academy of Sciences\\
Miusskaya pl. 4, 125047 Moscow, RUSSIA\\
\end{flushleft}
}\maketitle

\begin{abstract}
In the three-dimensional case, we consider  Schr\"odinger operator and an elliptic operator in divergence form. For slowly-decaying oscillating potentials, we establish spatial asymptotics of the Green's function. The main term in this asymptotics involves  $L^2(\mathbb{S}^2)$-valued analytic function whose behavior is studied away from the spectrum. This analysis is used to prove that absolutely continuous spectrum of both operators fills $\mathbb{R}^+$. We also apply our technique to 
establish   existence of wave operators for wave equation under optimal conditions for decay of potential.
\end{abstract} \vspace{1cm}



\maketitle

\setcounter{tocdepth}{3}

\tableofcontents

\section{Introduction}

In this paper, we study two operators that are central for Spectral and Scattering Theory of wave propagation.  The first one is Schr\"odinger operator
\begin{equation}\label{n11}
H=-\Delta+V, \quad x\in \mathbb{R}^3 
\end{equation}
and the second one is  elliptic operator written in ``divergence form''
\begin{equation}\label{n12}
D=-{\rm div} (1+V)\nabla, \quad x\in \mathbb{R}^3.
\end{equation}
We will study the Schr\"odinger operator in the first part of the paper, operator \eqref{n12} will be considered in the second part.  The potential $V$ is always assumed to be real-valued and decaying at infinity at a certain rate to be specified below. One motivation for this work comes from the following problem suggested by B. Simon \cite{rarry}:\smallskip

{\it Let $V$ be a function on $\mathbb{R}^\nu$ which obeys 
\begin{equation}\label{l22}
\int_{\mathbb{R}^\nu} |x|^{-\nu+1}V^2(x)dx<\infty.
\end{equation}
Prove that $-\Delta+V$ has an a.c. spectrum of infinite multiplicity on $[0,\infty)$ if $\nu\ge 2$.
}
\smallskip

For $\nu=1$, a lot is known, e.g., the characterization of $V\in L^2(\mathbb{R}^+)$ in terms of spectral data was obtained in  \cite{ks-a2}.
The Simon's multidimensional $L^2$ conjecture generated a lot of activity and many results were obtained.
We recommend two recent surveys \cite{surv1}, \cite{surv2} and \cite{o1,o2,o3} for more information and the list of references.  

The goal of this paper is to go far beyond understanding the a.c. spectral type. When the spectral parameter is taken off the spectrum, we will study the asymptotics of the Green's function and  establish  existence of the a.c. spectrum and wave operators as a consequence.  In a sense, this paper builds on ideas introduced in \cite{den1} where less precise estimates were proved using  perturbation theory and more restrictive class of potentials was treated. \smallskip

To illustrate the kind of results obtained in this paper, we list a few of them below. First, we need the following notation: given function $f$ defined on $\{x\in \mathbb{R}^3:|x|>N\}$ and $p\geq 1$, we introduce
\[
\|f\|_{\ell^p[N,\infty),L^\infty}\dd \left(\sum_{n=N}^\infty \Bigl(\sup_{n<|x|<n+1} |f(x)|\Bigr)^p\right)^{1/p}.
\]

\begin{theorem}\label{il_1}
Consider $V$ that  satisfies the following conditions:
\begin{equation}\label{main-assump_il}
V={\rm div}\, Q, \quad Q\in C^1(\mathbb{R}^3),\quad   \|V\|_{\ell^2(\mathbb{Z}^+),L^\infty}+\|Q\|_{\ell^2(\mathbb{Z}^+),L^\infty}<\infty\,.
\end{equation}
Then, $\sigma_{ac}(-\Delta+V)=\R^+.$
\end{theorem}
Comparing it to other recent  results in the field (see, e.g., \cite{surv1} and \cite{surv2}), this theorem is, perhaps, the strongest in terms of unconditional point-wise decay imposed on $V$. This rate of decay also turns out to be optimal on $\ell^p$ scale. The statement of the theorem is contained in a stronger result, theorem~\ref{t5}, which is proved in the first part of the paper along with auxiliary lemmas. The method is based on analysis of the spatial asymptotics of the Green's function $G(x,y,z)$ when $z$ is a regular point of $H$, $y\in \mathbb{R}^3$ is fixed and $x$ tends to infinity in arbitrary direction.  Under rather mild (and, again, essentially optimal) assumptions on $V$, we prove the formula 
\begin{equation}\label{noas}
G(x,y,z)=G^0(x,y,z)(A_\infty(\sigma,y,k)+o(1)), \,|x|\to\infty, x/|x|\to \sigma\in \mathbb{S}^2,
\end{equation}
where $G^0$ is Green's function of $H_0=-\Delta$ and  $A_\infty(\sigma,y,k)$ is $L^2(\mathbb{S}^2)$-valued  function analytic in $k=\sqrt z$.
 We  obtain the uniform estimates for $A_\infty$ and study its boundary behavior in $k\in \mathbb{C}^+$ near the real line by identifying the proper harmonic majorant. The standard  properties of the vector-valued functions in the Hardy class $H^2(\mathbb{D})$ imply the entropy bound for the spectral measures and theorem~\ref{il_1} follows immediately as a corollary.
 
  In the second part of the paper, we study operator \eqref{n12} and  the wave equation
\begin{equation}\label{acu-acu}
u_{tt}={\rm div}(1+V)\nabla u, \quad u(x,0)=f_0(x),\quad u_t(x,0)=f_1(x),
\end{equation}
which corresponds to, e.g., the propagation of acoustic waves in the medium described by potential $V$.  Formally, the group $e^{it\sqrt{D}}$ defines the solutions to \eqref{acu-acu} if $D$ is given by \eqref{n12}. The operator \eqref{n12} is non-negative under very mild assumptions on $V$ so $\sqrt D$ is well-defined by the Spectral Theorem and the evolution $e^{it\sqrt{D}}$  preserves the $L^2(\mathbb{R}^3)$ norm.  Our central contribution is the following theorem.

\begin{theorem}\label{main-th_il}
Suppose $V$ satisfies conditions:
\begin{equation}\label{2-assu}
\|V\|_\infty<1,\, V={\rm div} \, Q, \, Q\in C^2(\mathbb{R}), \, \max_{j=0,1,2}\|D^jQ\|_{\ell^2(\mathbb{Z}^+),L^\infty}<\infty\,.
\end{equation}
Then, the following wave operators exist
\begin{equation}\label{wowo}
W^{\pm}(\sqrt{D},\sqrt{H_0})\dd {s-\lim}_{t\to\pm \infty}e^{it\sqrt{D}}e^{-it\sqrt{H_0}}
\end{equation}
and the limit is understood in the strong sense.
\end{theorem}
Existence of wave operators, in a standard way, implies that $D$, restricted to the ranges of $W^\pm$ is unitarily equivalent to $H_0=-\Delta$ and that guarantees the infinite multiplicity of the a.c. spectrum of~$D$. 

Another application of our technique has to do with, perhaps, the most natural and basic question about the long time behavior of solution to equation \eqref{acu-acu}: given some $f_0$ and $f_1$, does the solution propagate ballistically like in the unperturbed case? In view of possible eigenvalues embedded into the continuous spectrum, the answer to the general question should be negative (indeed, if $\Psi$ is an eigenfunction for eigenvalue $E=1$, we observe that function $\cos t \cdot \Psi(x)$ solves the problem with $f_0=\Psi, f_1=0$ but does not propagate at all). However, we have the following theorem.

\begin{theorem} \label{il5} Suppose $V$ satisfies conditions of theorem \ref{main-th_il}, $f$ is compactly supported, nonnegative, and is not  zero identically, then a nontrivial part of $f$ propagates ballistically. More precisely, we can write $f=h_1+h_2$, where $h_1\perp h_2, h_1\neq 0$ and
\[
\lim_{t\to +\infty} \|e^{-it\sqrt D}h_1-e^{-it\sqrt H_0}(W^{+})^{-1}h_1\|_2=0\,.
\]
\end{theorem}
To clarify the statement, $h_1$ is chosen as the orthogonal projection of $f$ to the range of $W^+$ and $h_2$ is $h_1$'s orthogonal complement in $L^2(\mathbb{R}^3)$. Since $h_1\perp h_2$, we also have $e^{-it\sqrt D}h_1\perp e^{-it\sqrt D}h_2$ and, therefore, part of the wave propagates ballistically. We notice carefully that not all of $e^{it\sqrt D}f$ is necessarily propagating: if $h_2$ is not equal to zero, then part of the wave can be localized around the origin (e.g., oscillate like in the example with eigenstate discussed above or undergo even more complicated dynamics if singular continuous spectrum is present). \bigskip

  The classes of potentials considered in this paper are strikingly sharp for the kind of results we obtain. In both \eqref{n11} and \eqref{n12}, we let $V$ decay at infinity slowly and oscillate. More precisely, this can be expressed in the following way:
 $
 V={\rm div}\, Q
 $
 where $Q$ is $C^1(\mathbb{R}^3)$ vector-field that decays at infinity. For example, one can think of
 \[
 V={\rm div}\left(\frac{(\sin x_1,0,0)}{(x^2+1)^{0.25+\delta}}\right)=\frac{\cos x_1}{(x^2+1)^{0.25+\delta}}+O(|x|^{-1.5-2\delta})\,, \delta>0\,.
 \]
In the last section of the paper, we show that these potentials are in fact ubiquitous, e.g., they appear naturally in investigation of the random models  (e.g., random decaying potentials studied by Bourgain \cite{bour1} and Rodnianski-Schlag \cite{rods})).

    To avoid some unessential technical issues (e.g., the correct definition of the operator $H$) we assume that both $Q$ and $V$ are bounded and that they decay at infinity as follows:
 \begin{equation}\label{lotto}
 \sup_{n-1<|x|<n}|V(x)|\in \ell^2(\mathbb{N}), \,  \sup_{n-1<|x|<n}|Q(x)|\in \ell^2(\mathbb{N}).
 \end{equation}
 This decay, similar to \eqref{l22}, is also $L^2$-like and that makes our results optimal (i.e., changing $\ell^2$ to $\ell^p, p>2$ in \eqref{lotto} leads to absence of a.c. spectrum in general). The oscillation of the potential is also crucial for our analysis of Green's function asymptotics. Indeed, even in the one-dimensional case this asymptotics contains the nontrivial WKB correction if the potential $V$ does not decay fast or does not have some oscillation. In \cite{denik, perelman}, the WKB correction was studied in the three-dimensional case.

  \bigskip
 
The problems considered in this paper are classical to scattering theory of PDE and the references to older results are numerous. If the potential $V$ is short-range, i.e., 
 \[
 |V(x)|\le \frac{C}{(1+|x|)^{1+\delta}},\, \delta>0,
\]
the limiting absorption principle  (see \cite{ag},\cite{yafaev1}) implies that the positive spectrum of $H=-\Delta+V$ is purely absolutely continuous.  For elliptic problems written in the form \eqref{n12}, the limiting absorption principle was studied in \cite{eidus,ilin1,ilin2,ilin}.  As far as existence of wave operators is concerned, another very effective tool, the Enss method, has been widely used to analyze the scattering problem in the case when potential is short-range.   We recommend the monograph \cite{yafaev1} as a reference that contains most of the classical results in scattering theory that are relevant to our paper. 
 In comparison to short-range case, potentials that satisfy \eqref{lotto} are too rough for the spectrum to be purely absolutely continuous. In fact, the a.c spectrum can coexist with rich singular spectrum and thus the standard methods (limiting absorption principle or Enss method mentioned above) become ineffective. The technique we use allows to overcome this obstacle.
 
 The scattering theory for the wave equation \eqref{acu-acu} was developed in \cite{eidus, ilin,yama} under the assumptions that $V$ decays at infinity fast. In this context, see also \cite{lp} for the classical treatment of the scattering problem for wave equations.
Our method to control evolution $e^{it\sqrt{D}}$ is based on the well-known  formula that expresses it as a contour integral of the resolvent (see, e.g., \cite{vainberg1} where this approach is discussed). This allows us to prove existence of wave operators and obtain the stationary representation for them.  In \cite{deni1,deni2,deni3},  the analysis of the stationary scattering problem has been used to study the existence of wave and modified wave operators in the one-dimensional case. The current paper develops this technique and puts it into the multidimensional setting.
 
The basis for our analysis is the method of a priori estimates for some Helmholtz-like equations. For Helmholtz equations, it was used in \cite{perthame1, perthame2} in a different context. In \cite{den5}, analogous a priori estimates were used to study hyperbolic pencils related to Schr\"odinger operator. The idea to control the asymptotics of $G(x,y,z)$ in the $L^2(\mathbb{S}^2)$ topology is not new,  it was used by Agmon \cite{Ag1}  for the short-range case. We, however, consider the functions $A_\infty(\sigma,y,k)$ in \eqref{noas} as elements of the $L^2(\mathbb{S}^2)$-valued Hardy space and that allows us to obtain necessary estimates on the boundary behavior. These bounds become crucial in the proof of existence of wave operators \eqref{wowo}.

We finish this introduction by making a remark  that  we  considered the three-dimensional case only to avoid unessential technicalities. We believe our approach works in any dimension after minor modifications.   It is also conceivable that all results obtained in this paper can be generalized to  $V$ that can be written in the following form:
 \[
 V=V_{\rm osc}+V_{\rm sr},
 \]
where $V_{\rm osc}$, the slowly decaying and oscillating part, is like in theorems \ref{il_1} and \ref{main-th_il} and $V_{\rm sr}$, the short-range part, satisfies
\[
 \|V\|_{\ell^1(\mathbb{Z}^+),L^\infty}<\infty\,.
\]
We do not pursue this direction here.

\bigskip

{\bf Notation}

\begin{itemize}
\item  $B_r(x)$ denotes closed ball centered at $x$ and radius $r$ and $S_r(x)$ is the corresponding  sphere, $\mathbb{S}^2\dd S_1(0)$.

\item If $A$ is a self-adjoint operator defined on the dense subset of the Hilbert space  $\cal{H}$ and $z$ does not belong to its spectrum (e.g., $z\notin\sigma(A)$), then $R_z=(A-z)^{-1}$ denotes the resolvent of $A$ at point $z$. If $R_z$ is given by the integral operator, i.e., if
\[
(R_zf)(x)=\int_{\mathbb{R}^3} G(x,y,z)f(y)dy,
\]
then we will call the integral kernel $G$ the Green's function of $A$. For example, if $A=H_0=-\Delta$ and $\cal{H}=L^2(\mathbb{R}^3)$, then   (\cite{rs2}, formula (9.30), p. 73)
\[
G^0(x,y,k^2)=\frac{e^{ik|x-y|}}{4\pi|x-y|}, \quad k\in \mathbb{C}^+.
\]

\item
Sobolev spaces  over the domain $U$ with  square integrable derivatives up to order $l$ are denoted by $\cal{H}^l(U)$. The space of compactly supported infinitely smooth functions is denoted by $C_c^\infty(U)$.

\item The symbol $\delta_y$ denotes the Dirac delta-function at point $y\in \mathbb{R}^3$.

\item The symbol $\sigma_x$ stands for surface measure.

\item If $x\in \mathbb{R}^3$ and $x\neq 0$, then
$
\widehat x\dd x/|x|
$.

\item  The symbol $P_{\mathbb{C}^+}(k,\xi)$ stands for the Poisson kernel in the upper half-plane, i.e.,
\[
P_{\mathbb{C}^+}(k,\xi)=\frac{\Im k}{\pi((\Re k-\xi)^2+\Im^2k)}.
\]
In general, if $\Omega$ is the domain in $\mathbb{C}$ with piece-wise smooth boundary $\partial\Omega$, then the Poisson kernel will be denoted by
$
P_\Omega(k,\xi), \, k\in \Omega,\,\xi\in \partial\Omega
$. Thus, for every $f\in C(\overline\Omega)$, harmonic in $\Omega$, we have
\[
f(k)=\int_{\partial\Omega}P_\Omega(k,\xi)f(\xi)d|\xi|
\]
with $d|\xi|$ being the arc-length measure.

\item Given $[a,b]$ such that $0\notin [a,b]$, we define $\Pi(a,b,h)\dd \{k\in \mathbb{C}^+, \Re k\in (a,b), \, \Im k\in (0,h)\}$. 

\item For two non-negative functions
$f_{1(2)}$, we write $f_1\lesssim f_2$ if  there is an absolute
constant $C$ such that
\[
f_1\le Cf_2
\]
for all values of the arguments of $f_{1(2)}$. We define $\gtrsim$
similarly and say that $f_1\sim f_2$ if $f_1\lesssim f_2$ and
$f_2\lesssim f_1$ simultaneously.

\item

If $(\Omega_{1(2)},\mu_{1(2)})$ are two measure spaces and  $A$ is a linear operator, bounded from $L^{p_1}(\Omega_1,\mu_1)$ to $L^{p_2}(\Omega_2,\mu_2)$, then its operator norm is denoted by $\|A\|_{p_1,p_2}$. In general, if $X_{1(2)}$ are two Banach spaces and $A$ is a linear bounded operator from $X_1$ to $X_2$, then $\|A\|_{X_1,X_2}$ will denote its operator norm.

\item For shorthand, we will use $\|f\|_p$ to indicate the $L^p(\mathbb{R}^3)$ norm of the function $f$. Similarly, $L^p$ will refer to $L^p(\mathbb{R}^3)$. 


\item The Fourier transform of function $f$ will be denoted by 
\[
\cal{F}f=\widehat f(\xi)\dd \int_{\mathbb{R}^3} f(x)e^{-2\pi i\langle x,\xi\rangle}dx
\]
and the inverse Fourier by $\widecheck f$ or $\cal{F}^{-1}f$.

\item Given self-adjoint operator $H$ with spectrum $\sigma(H)$, we
define the following set
\[
\Sigma(H)\dd \{k\in \mathbb{C}^+, k^2\notin \sigma(H)\}.
\]
We will often write $\Sigma$ dropping $H$.

\item The averaging of function $f$ over the sphere centered at $x$ with the radius $r$ is denoted by
\[
M_r(f)(x)\dd \langle f\rangle_{S_r(x)}=\frac{1}{|S_r(x)|}\int_{S_r(x)}f(\xi)d\sigma_\xi.
\]

\item Potential $V$ is called short-range if there is $\delta>0$ such that $|V|\lesssim (1+|x|)^{-1-\delta}$.

\item The symbol $C$ denotes the absolute constant which can change the value from formula to formula. If we write, e.g., $C(\alpha)$, this defines a positive function of parameter $\alpha$.
\end{itemize}

\bigskip
\section{Part 1. Schr\"odinger operator with decaying and oscillating potential}

\subsection{Formulation of main results}
Consider stationary Schr\"odinger operator $H$ given by \eqref{n11}
\[
H=-\Delta+ V, \quad x\in \mathbb{R}^3
\]
with real-valued potential $V$ that  satisfies the following properties
\begin{equation}\label{main-assump}
V={\rm div}\, Q, \quad Q\in C^1(\mathbb{R}^3),\quad  \|V\|\dd \|V\|_{\ell^2(\mathbb{Z}^+),L^\infty}+\|Q\|_{\ell^2(\mathbb{Z}^+),L^\infty}<\infty\,.
\end{equation}
We notice that both $V$ and $Q$ converge to $0$ as $|x|\to\infty$. Since $\lim_{|x|\to\infty} V(x)=0$, it is known from Weyl's Theorem (\cite{reed4}, p.117) that $\sigma_{\rm ess}(H)=[0,\infty)$.
The question what decay assumptions at infinity imply that $\sigma_{ac}(H)=[0,\infty)$ is more delicate and has been extensively studied lately, especially in one-dimensional case (e.g., \cite{den-kis}).

In the first part of the paper, we study the spatial asymptotics of the Green's function $G(x,y,z)$ when $z\notin \sigma(H)$ and introduce ``an amplitude'', which is $L^2(\mathbb{S}^2)$-valued analytic function in $z$. We study its properties and establish  the absolute continuity of the spectrum of $H$ as a corollary. \bigskip

The following quantity will play the key role. Let
\[
A(x,y,k)\dd 4\pi |x-y|e^{-ik|x-y|}G(x,y,k^2)
\]
for $k\in \Sigma$. This formula is easy to understand, in fact
\[
 A(x,y,k)=\frac{G(x,y,k^2)}{G^0(x,y,k^2)}
\]
thus the comparison is made to free Green's function.
We will take $|x|\to\infty$ while keeping $y$ fixed and study the asymptotical behavior. 
This is related to the concept of  Martin boundary in the theory of harmonic functions, potential theory, and elliptic PDE (see, e.g., \cite{murata}) in the case when $k\in i\mathbb{R}^+$ and has large absolute value.

The main results of the first part of this paper are  listed below. 

\begin{theorem} \label{t1} Let $V$ satisfy \eqref{main-assump}. For every $\Pi(a,b,h)$, we have
\[
\sup_{r>1}\frac{1}{r^2} \int_{|x-y|=r} |A(x,y,k)|^2d\sigma_x<\frac{C(a,b,h,|y|,V)}{\Im^4 k}
\]
as long as $k\in \Pi(a,b,h)$.
\end{theorem}

\begin{theorem} \label{t2}  Let $V$ satisfy \eqref{main-assump}. There is the function $A_\infty(\sigma,y,k)$, defined for every $y\in \mathbb{R}^3, k\in \Sigma$. It is $L^2(\mathbb{S}^2)$ vector-valued function in $\sigma$ and it is analytic in $k\in \Sigma$ (as an $L^2(\mathbb{S}^2)$-valued function). Moreover,
\[
\lim_{r\to\infty}\|A(y+r\sigma,y,k)- A_\infty(\sigma, y,k)\|_{L^2(\mathbb{S}^2)}=0\,.
\]
\end{theorem}
For the short-range potentials, Agmon proved analogous result in \cite{Ag1}.

\begin{theorem} \label{t3}  Let $V$ satisfy \eqref{main-assump}. $A_\infty(\sigma,y,k)$ has the following asymptotics in  sectors of $\mathbb{C}^+$:
\[
\lim_{|k|\to\infty, \arg k\in (\delta,\pi-\delta)}\|A_\infty(\sigma,y,k)-1\|_{L^2(\mathbb{S}^2)}=0
\]
for every $\delta>0$. In particular, this implies that $A_\infty$ is not identically equal to zero in $\Sigma$.

\end{theorem}

 Take any $f\in L^2(\mathbb{R}^3)$ and assume that it has compact support. Let $\sigma_f$ be its spectral measure relative to $H$.
  The proofs of theorems \ref{t2} and \ref{t3} give continuity of $A_\infty(\sigma,y,k)$ in  $y$ in $L^2(\mathbb{S}^2)$ topology. So, we can define $h_f(\sigma,k)$:
\begin{equation}\label{defh}
h_f(\sigma,k)\dd \int_{\mathbb{R}^3} A_\infty(\sigma,y,k)e^{-ik\langle \sigma,y\rangle}f(y)dy\,.
\end{equation}

\begin{theorem} \label{t4}   Let $V$ satisfy \eqref{main-assump} and $[a,b]\subset (0,\infty)$. Then
\begin{equation}\label{t4t4}
\|h_f(\sigma,k)\|^2_{L^2(\mathbb{S}^2)}\le C(a',b',a,b,V,f)\left(1+\int_{a^2}^{b^2} P_{\mathbb{C}^+}(k,\sqrt\eta)  d\sigma_f(\eta) \right)
\end{equation}
for all intervals $(a',b')\subsetneq (a,b)$ and all $k\in \Pi(a',b',1)$.
\end{theorem}

 {\bf Remark.} The last theorem implies that
 \[
 \|h_f(\sigma,k)\|^2_{L^2(\mathbb{S}^2)}\le \frac{C(a,b,f,V)}{\Im k}, \quad k\in \Pi(a,b,1)
 \]
 for every $[a,b]\subset (0,\infty)$.

\begin{theorem}\label{t5}
Under the conditions of the previous theorem, if we assume that $f$ is non-negative (and not identically equal to zero), then $h_f$ is not identically equal to zero and
\[
\int_a^b \log \sigma_f'(E)dE>C(a,b,V,f)
\]
for every $[a,b]$.  As a corollary, we have $\sigma_{ac}(H)=[0,\infty)$.
\end{theorem}

The result about absolute continuity is sharp in the following sense. 

\begin{lemma}\label{ssha}
For every $p>2$, there are potentials $V$ that can be written in the form
\[
V={\rm div }\,Q, \, \|V\|_{\ell^p(\mathbb{Z}^+),L^\infty}<\infty, \,  \|Q\|_{\ell^p(\mathbb{Z}^+),L^\infty}<\infty, \, p>2
\]
and $\sigma_{ac}(H)=\emptyset$.
\end{lemma}

The plan of the first part is as follows. We start with proving sharpness, lemma \ref{ssha}.
The next section will contain some auxiliary results. In section 4, we study  properties of linear and bilinear operators used later in the text. Section 5 contains the proofs of theorems \ref{t1}--\ref{t5}. The harmonic majorant for $A_\infty(\sigma,y,k)$ is found in the last section. \vspace{0.5cm}

\subsection{Sharpness of $\ell^2$ condition}

\begin{proof}{\it(of lemma \ref{ssha}).}
Consider 
\[
V(x)\dd  {\rm div}\, Q,\, Q\dd  \frac{q(|x|)}{|x|}(x_1,x_2,x_3)\,,
\]
where 
\begin{equation}\label{potepote}
q(r)\dd \sum_{n=2}^\infty a_n \phi(r-n!)\,,
\end{equation}
$a_n\dd n^{-\gamma}, \gamma\in (0,\frac 12)$ and $\phi$ is smooth  function (a ``bump'') supported on $[-1,1]$ which is not identically zero. Differentiation gives
\[
V(x)=q'(r)+\frac{2q(r)}{r}, \quad r\dd |x|\,.
\]
Clearly, $V$ satisfies conditions of lemma \ref{ssha}.

By the theorem 7 from  \cite{ds} and Relative Trace-class Perturbation Theorem (theorem 8.8, \cite{yafaev}), we know that  $\sigma_{ac}(-\Delta+V)=\empty\sigma_{ac}(H_1)$ where $H_1=-\Delta+V, x\in \mathbb{R}^3\backslash B_1(0)$ with Dirichlet boundary condition on $\mathbb{S}^2$. Since $V$ is radially symmetric, $H_1$ is unitarily equivalent to 
\[
-\frac{d^2}{dr^2}-\frac{B}{r^2}+V(r)
\]
defined on $L^2([1,\infty),L^2(\mathbb{S}^2))$ with Dirichlet boundary condition at $r=1$. The symbol $B$ denotes Laplace-Beltrami operator on $L^2(\mathbb{S}^2)$. Thus, in the orthogonal basis of spherical harmonics, $H$ is a direct orthogonal sum of one-dimensional operators $\{L_n\}$ (counting multiplicity)
\[
L_n\dd -\frac{d^2}{dr^2}+\frac{\lambda_n}{r^2}+q'(r)+\frac{2q(r)}{r}, \,n\in \mathbb{N}
\]
with Dirichlet boundary conditions at $r=1$, where $\{\lambda_n\}$ are eigenvalues of $B$.
In \cite{kls},  theorem~1.6, the following potential $q_1$ was considered
\[
q_1(x)=\sum_{j=1} a_j W(x-x_j)\,,
\] 
where $W$ is non-negative, supported on $[-1,1]$ and $\lim_{j\to\infty}a_j=0, \lim_{j\to\infty}x_j/x_{j+1}=0$. Then, it was proved that $\sum_{j=1}^\infty a_j^2=\infty$ implies $\sigma_{ac}(-d^2/dx^2+q_1)=\emptyset$ for every boundary condition at zero. The proof of this result, however,  extends to sign-indefinite potentials without efforts and this gives
\[
\sigma_{ac}(-d^2/dr^2+q')=\emptyset\,,
\]
where $q$ is defined in \eqref{potepote} and the Dirichlet condition at $r=1$ is assumed. For the perturbation in $L_n$, we have 
\[
\frac{\lambda_n}{r^2}+\frac{2q(r)}{r}\in L^1[1,\infty)\,,
\]
which makes it a relative trace-class perturbation that leaves the absolutely continuous spectrum intact. To summarize, we have  $\sigma_{ac}(L_n)=\emptyset$ for all $n$ and so the absolutely continuous spectra of $H_1$ and $H$ are empty.

\end{proof}

\vspace{0.5cm}

\subsection{Basic estimates for Green's function}

In this section, we will be mostly interested in the general properties of the Green's kernel for bounded potential. First, we need to make sure that this kernel exists. To do that, we start with lemma.
\begin{lemma}\label{dobro1}
If $V\in L^\infty(\mathbb{R}^3)$ and $z\notin \sigma(H)\cup [0,\infty)$, then $R_zf\in \cal{H}^2(\mathbb{R}^3)$ for every $f\in L^2(\mathbb{R}^3)$.
\end{lemma}
\begin{proof}
Before proceeding with the proof, we recall two main identities from the Perturbation Theory:
\[
R_z=R^0_z-R_zVR_z^0=R_z^0-R_z^0VR_z, \quad z\notin\sigma(H)\cup\sigma(H_0),
\]
where $V=H-H_0$, and
\[
R_z=R_{z_0}+(z-z_0)R_zR_{z_0}, \quad z,z_0\notin\sigma(H)\,.
\]
We will be using them multiple times in this paper. To prove lemma,
we write $R_zf=R_z^{0}f-R_z^0VR_zf$ and notice that $R_z$ maps $L^2(\mathbb{R}^3)$ to itself, $R_z^0$ maps $L^2(\mathbb{R}^3)$ to $\cal{H}^2(\mathbb{R}^3)$. Since $V$ is a multiplier in $L^2(\mathbb{R}^3)$, we have the required property.
\end{proof}

Since $\cal{H}^2(\mathbb{R}^3)$ is continuously embedded into $L^\infty(\mathbb{R}^3)$, Corollary 2.14 from \cite{cycon} can be applied to get representation
\[
R_zf(x)=\int_{\mathbb{R}^3} G(x,y,z)f(y)dy, \quad \sup_{x\in \mathbb{R}^3}\int_{\mathbb{R}^3}|G(x,y,z)|^2dy<\infty
\]
for all $z\notin \sigma(H)\cup [0,\infty)$. In the case when $\|V\|<\infty$, we get $[0,\infty)\subseteq \sigma(H)$, so it is sufficient to require only $z\notin \sigma(H)$.

We continue with simple and well-known symmetry result.
\begin{lemma}
 If $V\in L^\infty(\mathbb{R}^3)$, then
$
G(x,y,z)=G(y,x,z)
$
for each $z\notin \sigma(-\Delta+V)$.
\end{lemma}
\begin{proof}
The perturbation series for the resolvent
\[
G(x,y,k^2)=G^0(x,y,k^2)-\int_{\mathbb{R}^3} G^0(x,\xi_1,k^2)V(\xi_1)G^0(\xi_1,y,k^2)d\xi_1+\ldots
\]
converges absolutely if  $\Im k >L$ where $L$ is large enough.
This implies $G(x,y,z)=G(y,x,z)$ for these $k$. However, both of these functions are analytic in $\mathbb{C}\backslash \sigma(H)$ so the identity can be extended to the domain of analyticity.
\end{proof}

\begin{lemma}If $V\in L^\infty(\mathbb{R}^3)$, then $G(x,y,\overline{z})=\overline{G(x,y,z)}$ if $z\notin \sigma(H)$. Then, $A(x,y,-\overline{k})=\overline{A(x,y,k)}$ if $k\in \Sigma$.\label{opew}
\end{lemma}
\begin{proof}
Since $\left((H-z)^{-1}\right)^*=(H-\overline{z})^{-1}$, the previous lemma gives
\[
G(x,y,z)=\overline{G(x,y,\overline{z})}\,.
\]
The identity for $A$ now follows from its definition.
\end{proof}

\begin{lemma}\label{shift}
Suppose $V$ satisfies  \eqref{main-assump}. Take $y\in \mathbb{R}^3$ and consider
\[
Q_{[y]}(x)\dd Q(x-y),\quad
V_{[y]}(x)\dd {\rm div}_x Q_{[y]}(x)\,.
\]
Then,
\[
\|V_{[y]}\|_{\ell^2(\mathbb{Z}^+),L^\infty}\lesssim 1+|y|, \quad  \|Q_{[y]}\|_{\ell^2(\mathbb{Z}^+),L^\infty}\lesssim 1+|y|\,.
\]

\end{lemma}

\begin{proof}
We trivially have
\[
\sup_{n_1<|x|<n_2}|Q(x)|\le \sum_{j=n_1}^{n_2-1} \sup_{j<|x|<j+1} |Q(x)|\,.
\]
Since
\[
\sup_{n<|x|<n+1}|Q_{[y]}(x)|\leq \sup_{\max (0,n-[y]-1)<|x|<n+[y]+1} |Q(x)|\,,
\]
we get the statement of the lemma from triangle inequality in $\ell^2$. The estimates for $V$ can be obtained similarly.
\end{proof}

We will need to truncate potentials in the following ways. Given $\rho>1$, consider smooth  $\alpha_\rho(x), x\in [0,\infty)$  such that $0\leq \alpha_\rho\leq 1$ for all $x$,   $\alpha_\rho=1$ on $[0,\rho]$, $\alpha_\rho=0$ for $x>\rho+1$. Take $V$ that satisfies  \eqref{main-assump} and define
\begin{equation}\label{tru}
Q_{(\rho)}(x)\dd \alpha_{\rho}(|x|)Q(x),\quad V_{(\rho)}\dd {\rm div}\,Q_{(\rho)}\,.
\end{equation}
Similarly,
\[
Q^{(\rho)}\dd Q-Q_{(\rho)},\quad V^{(\rho)}\dd V-V_{(\rho)}\,.
\]
Notice that $\|V_{(\rho)}\|\lesssim \|V\|$ and its support is restricted to $B_{\rho+1}(0)$. For $V^{(\rho)}$ we have
\[
\|V^{(\rho)}\|\lesssim \|V\|, \quad \lim_{\rho\to\infty}\|V^{(\rho)}\|=0\,.
\]
Let us consider the corresponding operator by
$
H_{(\rho)}=-\Delta+V_{(\rho)}
$,
its resolvent  $R_{(\rho)}$, and the Green's function  $G_{(\rho)}(x,y,z)$. 
\begin{lemma}  \label{ls}  Assume that $V$ satisfies \eqref{main-assump}.  If  $f\in L^2(\mathbb{R}^3)$ and  $z\notin\sigma(H)$, then
\[
\lim_{\rho\to\infty} \|{R_{(\rho)}}_zf- R_zf\|_2=0\,.
\]
\end{lemma}
\begin{proof}Since $\lim_{\rho\to\infty}\|V^{(\rho)}\|_{2,2}=0$, we have $\sigma(H_{(\rho)})\to\sigma(H)$ in the Hausdorff sense if $\rho\to\infty$. This follows from the general perturbation theory.   For every $z\notin \sigma(H)$, we can take $\rho$ large enough and write
\[
{R_{(\rho)}}_zf=R_zf-{R_{(\rho)}}_z V^{(\rho)}R_zf\,.
\]
Since $\limsup_{\rho\to\infty}\|{R_{(\rho)}}_z\|_{2,2}<\infty$ and $\lim_{\rho\to\infty}\|V^{(\rho)}R_zf\|_2=0$, we get the statement of the lemma.
\end{proof}

Given $f\in L^2(\mathbb{R}^3)$, we can define the spectral measure $\sigma_f$ of $f$ relative to $H$. Similarly, we introduce ${\sigma_f}_{(\rho)}$.
The immediate corollary of the previous lemma is
\begin{lemma}\label{wkwk}
Assume that $V$ satisfies \eqref{main-assump}.  If $f\in L^2(\mathbb{R}^3)$, then
\[
{\sigma_f}_{(\rho)}\to \sigma_f, \,{\it as}\quad \rho\to\infty
\]
in the weak-($\ast$) sense.
\end{lemma}

\begin{proof}
Indeed, since 
\[
\langle {R_{(\rho)}}_zf,f\rangle=\int \frac{d{\sigma_f}_{(\rho)}(\lambda)}{\lambda-z}\to \langle {R}_zf,f\rangle=\int \frac{d{\sigma_f}(\lambda)}{\lambda-z}
\]
as $\rho\to\infty$ for every $z\in \mathbb{C}^+$, we get the statement of the lemma because continuous function with compact support can be approximated by its Poisson integral (imaginary part of the Cauchy integral).
\end{proof}

\begin{lemma}\label{pop}
If $V\in L^\infty(\mathbb{R}^3)$, then
\[
\sup_{x,y\in \mathbb{R}^3}\Bigl|G(x,y,z)-G^0(x,y,z)\Bigr|<C_1(z,\|V\|_\infty)
\]
for all $z\notin \sigma(H)\cup [0,\infty)$. Moreover, for every $a,b,h$, we have
\begin{equation}\label{ce}
C_1(z,\|V\|_\infty)<C(a,b,h)\left(\frac{\|V\|_\infty}{\Im k}+\frac{\|V\|^2_\infty}{\Im^2 k} \right)
\end{equation}
if $z=k^2$ and $k\in \Pi(a,b,h)$.
\end{lemma}

\begin{proof}
Write
\[
G(x,y,z)=G^0(x,y,z)-\int_{\mathbb{R}^3} G(x,\xi,z)V(\xi)G^0(\xi,y,z)d\xi\,.
\]
Since
$
\|G^0(\cdot,y,k^2)\|_2\lesssim (\Im k)^{-1/2}
$
and $\|R_z\|_{2,2}<C_2(z)$, we get
\begin{equation}\label{leret}
\|G(\cdot,y,z)-G^0(\cdot,y,z)\|_2\leq \frac{\|V\|_\infty C_2(z)}{\sqrt{\Im k}}\,.
\end{equation}
Now Cauchy-Schwarz inequality and symmetry of the kernel yield
\[
\left|\int_{\mathbb{R}^3} G(x,\xi,z)V(\xi)G^0(\xi,y,z)d\xi\right|\le \|V\|_\infty \|G^0(\cdot,y,z)\|_2 \|G(x,\cdot,z)\|_2\lesssim  \frac{\|V\|_\infty}{{\Im k}}\Bigl(1+C_2(z)\|V\|_\infty\Bigr)\,.
\]
More careful analysis of the constant  gives \eqref{ce} because
\begin{equation}\label{cac}
C_{2}(z)\le \frac{C(a,b,h)}{\Im k}
\end{equation}
if $k\in \Pi(a,b,h)$.
\end{proof}
The previous proof immediately yields the following lemma.
\begin{lemma}\label{alemma}   Assume that $V$ satisfies \eqref{main-assump}.           If $z\notin \sigma(H)$, then
\[
\lim_{\rho\to\infty}\sup_{x,y\in \mathbb{R}^3}|G_{(\rho)}(x,y,z)-G(x,y,z)|=0,
\]
\begin{equation}\label{dvaff}
\lim_{\rho\to\infty}\|G_{(\rho)}(x,y,z)-G(x,y,z)\|_{\cal{H}^2(r_1<|x-y|<r_2)}=0
\end{equation}
for all $r_{1(2)}: 0<r_1<r_2$.
\end{lemma}
\begin{proof}
Fix $z\notin \sigma(H)$. We can take $\rho$ large enough to have $z\notin \sigma(H_{(\rho)})$. The second resolvent identity gives
\[
G_{(\rho)}(x,y,z)-G(x,y,z)=\int_{\mathbb{R}^3} G_{(\rho)}(x,\xi,z) V^{(\rho)}(\xi) G(\xi,y,z)d\xi\,.
\]
Now, Cauchy-Schwarz inequality along with \eqref{leret} provide
\[
|G_{(\rho)}(x,y,z)-G(x,y,z)|\le \|V^{(\rho)}\|_\infty \|G_{(\rho)}(x,\cdot,z)\|_{2}\|G(\cdot,y,z)\|_2<C(z,\|V\|_\infty) \|V^{(\rho)}\|_\infty\,.
\]
Since $\|V^{(\rho)}\|_\infty\to 0$, we have the first statement of the lemma.

To prove \eqref{dvaff}, denote $u(x)\dd G(x,y,z)$, $u_{(\rho)}(x)\dd G_{(\rho)}(x,y,z)$ and write
\[
-\Delta u+Vu=zu,  \quad -\Delta u_{(\rho)}+V_{(\rho)}u_{(\rho)}=zu_{(\rho)}
\]
for $x: \, |x-y|>0$. Taking $\delta u=u-u_{(\rho)}$, we have
\[
-\Delta (\delta u)+V\delta u=z\delta u+(V_{(\rho)}-V)u_{(\rho)}\,.
\]
Since $\lim_{\rho\to\infty}\|\delta u\|_{L^\infty(r_1<|x-y|<r_2)}=0$ uniformly, we have $\lim_{\rho\to\infty}\|\Delta (\delta u)\|_{L^\infty(r_1<|x-y|<r_2)}=0$ uniformly. The Interior Regularity Theorem for elliptic equations  (\cite{evans}, p. 309) then implies \eqref{dvaff}.
\end{proof}

\begin{lemma}\label{poliu}
If $V\in L^\infty(\mathbb{R}^3)$ then
\begin{equation}\label{nn1}
\frac{1}{r^2}\int_{|x|=r}|G(x,0,z)|^2d\sigma_x\lesssim \frac{(1+|z|+\|V\|_\infty)^2}{r^2}\int_{r-1<|x|<r+1}|G(x,0,z)|^2dx
\end{equation}
and
\begin{equation}\label{nn2p2}
\frac{1}{r^2}\int_{|x|=r}|\partial_rG(x,0,z)|^2d\sigma_x\lesssim \frac{(1+|z|+\|V\|_\infty)^2}{r^2}\int_{r-1<|x|<r+1}|G(x,0,z)|^2dx
\end{equation}
uniformly in $r>2$ and $z\notin \sigma(H)$.
\end{lemma}
\begin{proof}
Indeed, we notice that for each ball $B_\rho(\xi)$ that does not contain $0$ we have
\[
\|\Delta G(\cdot,0,z)\|_{L^2(B_\rho(\xi))}\le (|z|+\|V\|_\infty)\|G(\cdot,0,z)\|_{L^2(B_\rho(\xi))}
\]
as follows from the equation
\begin{equation}\label{nn3}
-\Delta_x G(x,0,z)+V(x)G(x,0,z)=zG(x,0,z), \quad x\neq 0\,.
\end{equation}
Now, it is sufficient to consider balls $\{B_{0.9}(x_j), j=1,\ldots,N\}$ such that $|x_j|=r$ and $S_{r}(0)\subset \cup_j B_{0.9}(x_j)$. We can take $N\sim r^2$. In each ball $B_1(x_j)$, $G(x,0,z)$ solves an elliptic equation
\[
\Delta G(x,0,z)=(V-z)G(x,0,z)\,.
\]
Therefore, by Interior Regularity Theorem, we have
\[
\|G(x,0,z)\|_{\cal{H}^2(B_{0.9}(x_j))}\lesssim (1+|z|+\|V\|_\infty)\|G(x,0,z)\|_{L^2(B_{1}(x_j))}\,.
\]
Then, we can use the theorem about restricting the $\cal{H}^1(B_j)$ functions to  hypersurfaces (in $L^2(B_j\cap S_r(0))$ norm in our case, see \cite{evans}, p.258) to write
\[
\int_{|x|=r}|G(x,0,z)|^2d\sigma_x\le \sum_{j}\int_{S_r\cap B_{0.9}(x_j)}|G(x,0,z)|^2d\sigma_x\lesssim \sum_{j} \|G(x,0,z)\|^2_{\cal{H}^2(B_{0.9}(x_j))}\lesssim
\]
\[
 (1+|z|+\|V\|_\infty)^2      \sum_{j} \|G(x,0,z)\|^2_{L^2(B_{1}(x_j))}      \lesssim      (1+|z|+\|V\|_\infty)^2        \int_{r-1<|x|<r+1}|G(x,0,z)|^2dx\,.
\]
Since $\|\nabla G\|_{\cal{H}^1(B_{0.9}(x_j))}\lesssim \| G\|_{\cal{H}^2(B_{0.9}(x_j))}$, we can write analogous bounds for $\nabla G$. This will give \eqref{nn2p2}.
\end{proof}
As a corollary we immediately have the following lemma.
\begin{lemma}\label{lm2}
If $V\in L^\infty(\mathbb{R}^3)$ and $k\in \Pi(a,b,h)$, then

\begin{equation}\label{nn1-7}
\frac{1}{r^2}\int_{|x|=r}|G(x,0,k^2)|^2d\sigma_x\le C(a,b,h) \frac{(1+\|V\|_\infty)^2}{r^2}\left(\frac{1}{\Im k}+\frac{\|V\|_\infty^2}{(\Im k)^3}\right)
\end{equation}
and
\begin{equation}\label{nn2}
\frac{1}{r^2}\int_{|x|=r}|\partial_rG(x,0,k^2)|^2d\sigma_x\le  C(a,b,h) \frac{(1+\|V\|_\infty)^2}{r^2}\left(\frac{1}{\Im k}+\frac{\|V\|_\infty^2}{(\Im k)^3}\right)
\end{equation}
uniformly in $r>2$.

\end{lemma}
\begin{proof}The proof follows from the previous lemma, \eqref{leret}, and \eqref{cac}.
\end{proof}
\vspace{0.5cm}

\subsection{Study of auxiliary operators}

We start with simple technical observation.
\begin{lemma}\label{tra}
Let $\{a_l\}, l=0,\ldots,j-1, \,a_l\ge 0$ are given and $x\in \mathbb{R}^+$ satisfies
\[
x^j\le \sum_{l=0}^{j-1}a_lx^l\,.
\]
Then,
\begin{equation}\label{sii}
x\le j \sum_{l=0}^{j-1} a_l^{1/(j-l)}\,.
\end{equation}
\end{lemma}
\begin{proof} We have
\[
x^j\le \sum_{l=0}^{j-1}a_lx^l\le j \max_l \{a_lx^l\}, \quad {\rm so}\quad x\le j^{1/j} \max_l \{(a_l)^{1/j}x^{l/j}\}\,,
\]
which implies the lemma.
\end{proof}

 We introduce the  weight
\begin{equation}\label{weighte}
w(x)\dd \left\{
\begin{array}{cc}
1, & |x|>1\\
|x|^{-2},& |x|<1
\end{array}\right.
\end{equation}
and say that $f\in L^2_w(\mathbb{R}^3)$ if
\[
\|f\|_{2,w}\dd \left(\int_{\mathbb{R}^3} |f|^2wdx\right)^{1/2}<\infty\,.
\]

We will start with three model equations. In all of them, we assume that $k\in \mathbb{C}^+$.
\begin{equation}\label{me1}
-\Delta u_1-k^2 u_1=\frac{e^{ik|x|}}{4\pi|x|}   \Bigl(|x|^2{\rm div} \Bigl(\frac{f_1}{|x|}\Bigr)\Bigr)\,,
\end{equation}
where $f_1\in L^2(\mathbb{R}^3)$ and the both sides are considered as tempered distributions.

The second one is
\begin{equation}\label{me2}
-\Delta u_2-k^2 u_2=\frac{e^{ik|x|}}{4\pi|x|}  f_2\,,
\end{equation}
where $f_2\in L^2_w(\mathbb{R}^3)$.

The third one is
\begin{equation}\label{me7}
-\Delta u_3-k^2 u_3=\frac{e^{ik|x|}}{4\pi|x|}  \Bigl(|x|Vf_3\Bigr)
\end{equation}
and $\|V\|_{\ell^2(\mathbb{Z}^+),L^\infty}<\infty, f_3\in L^2(\mathbb{R}^3)$.

In each case, the solutions $u_{1(2,3)}$ will be understood by applying $(-\Delta-k^2)^{-1}$ to the right hand side. We can write
\[
\frac{e^{ik|x|}}{|x|}   \Bigl(|x|^2{\rm div} \Bigl(\frac{f_1}{|x|}\Bigr)\Bigr)=e^{ik|x|}{\rm div} \,f_1-e^{ik|x|}\frac{x}{|x|^2}f_1
 \]
 and
 \begin{equation}\label{plllp}
 u_1=\int_{\mathbb{R}^3} \frac{e^{ik|x-y|}}{4\pi |x-y|}e^{ik|y|}{\rm div} \,f_1dy-\int_{|y|<r_1/2} \frac{e^{ik|x-y|}}{4\pi |x-y|} e^{ik|y|}\frac{y}{|y|^2}f_1   dy-\int_{|y|>r_1/2} \frac{e^{ik|x-y|}}{4\pi |x-y|} e^{ik|y|}\frac{y}{|y|^2}f_1   dy\,.
 \end{equation}
 The first integral is understood as convergent integral if we write it as
 \begin{equation}\label{losio}
- \int_{\mathbb{R}^3} \frac{e^{ik|x-y|}}{4\pi |x-y|}(\nabla_y e^{ik|y|})f_1dy-\int_{\mathbb{R}^3}\nabla_y\left( \frac{e^{ik|x-y|}}{4\pi |x-y|}\right) e^{ik|y|}f_1dy\,.
 \end{equation}
The second integral in \eqref{plllp} converges absolutely since $f_1\in L^2(\mathbb{R}^3)$. It represents a  smooth function in $x\in \{x: r_1<|x|<r_2\}$. 
Thus, $u_1\in \cal{H}^1(r_1<|x|<r_2)$ for any $0<r_1<r_2$. The theorem about restriction to  hypersurfaces implies that $u_1(r\sigma)\in L^2(\mathbb{S}^2)$ for every $r>0$. Here, we have written $x=r\sigma, \sigma\in \mathbb{S}^2$ in spherical coordinates. The formulas \eqref{plllp} and \eqref{losio} show that if $\lim_{n\to\infty}\|f_1^{(n)}-f_1\|_2=0$ and $f_1^{(n)}\in C^\infty_c(\mathbb{R}^3)$, then
$\|u_1^{(n)}-u_1\|_{\cal{H}^1(r_1<|x|<r_2)}\to 0$ as $n\to\infty$. This observation makes it possible to always assume that $f_1\in C^\infty_c(\mathbb{R}^3)$ when obtaining the estimates for $u_1$ in $\cal{H}^1(r_1<|x|<r_2)$, the space we will be interested in later on. Then,  equation \eqref{me1} is understood in the classical sense. The same reasoning can be applied to $u_{2(3)}$.

We introduce
\[
\mu_{1(2,3)}\dd \frac{u_{1(2,3)}}{G^0(x,0,k^2)}    =(4\pi)|x|e^{-ik|x|}u_{1(2,3)}\,.
\]
Then,
\begin{equation}\label{main-eq00}
-\Delta\mu_1 -2\left(ik-\frac{1}{4\pi|x|}\right)\partial_r \mu_1=|x|^2{\rm div} \Bigl(\frac{f_1}{|x|}\Bigr)       , \quad x\neq 0\,,
\end{equation}\smallskip
\begin{equation}\label{main-eq01}
-\Delta\mu_2 -2\left(ik-\frac{1}{4\pi|x|}\right)\partial_r \mu_2=f_2    , \quad x\neq 0\,,
\end{equation}\smallskip
\begin{equation}\label{main-eq02}
-\Delta\mu_3 -2\left(ik-\frac{1}{4\pi|x|}\right)\partial_r \mu_3=|x|Vf_3     , \quad x\neq 0\,.
\end{equation}\smallskip
Having explicit expression for $(-\Delta-k^2)^{-1}$, we can write representations
\begin{equation}\label{fo1}
\mu_1=B^{(1)}_r(f_1)=r \int_{\mathbb{R}^3}\left(\frac{e^{ik(-r+|r\sigma-y|+|y|}}{4\pi|r\sigma-y| |y|}\right)  \Bigl(|y|^2{\rm div} \Bigl(\frac{f_1}{|y|}\Bigr)\Bigr) dy\,,
\end{equation}

\begin{equation}\label{fo2}
\mu_2=B^{(2)}_r(f_2)=r \int_{\mathbb{R}^3}\left(\frac{e^{ik(-r+|r\sigma-y|+|y|}}{4\pi|r\sigma-y| |y|}\right) f_2(y) dy\,,
\end{equation}

\begin{equation}\label{fo3}
\mu_3=B^{(3)}_r(V,f_3)=r \int_{\mathbb{R}^3}\left(\frac{e^{ik(-r+|r\sigma-y|+|y|}}{4\pi|r\sigma-y| |y|}\right) |y|f_3(y)V(y) dy\,,
\end{equation}
thus defining operators $B_r^{(j)}, j=1,2,3$. In the definition of $B_r^{(1)}$, we can again integrate by parts to get convergent integral or assume that $f_1\in C_c^\infty(\mathbb{R}^3)$.

We also need to define the fourth operator
\[
B^{(4)}_r(f_4)=r \int_{\mathbb{R}^3}\left(\frac{e^{ik(-r+|r\sigma-y|+|y|}}{4\pi|r\sigma-y| |y|}\right) |y|{\rm div}\,f_4 dy\,,
\]
where $f_4\in L^2_w(\mathbb{R}^3)$. Notice that
\[
B^{(1)}_r(f_1)=B^{(4)}_r(f_1)-B^{(2)}_{r}\left(f_1\cdot \frac{y}{|y|}\right)
\]
or
\begin{equation}\label{chetvert}
B^{(4)}_r(f_1)=B^{(1)}_r(f_1)+B^{(2)}_{r}\left(f_1\cdot  \frac{y}{|y|}\right)\,.
\end{equation}
 Similarly to $B_r^{(1)}$, integration by parts defines convergent integral 
\begin{equation}\label{lk1}
B^{(4)}_r(f_4) =-r\int_{\mathbb{R}^3}\nabla_y \left(\frac{e^{ik(-r+|r\sigma-y|+|y|}}{4\pi|r\sigma-y| }\right)   f_4dy
\end{equation}
and this is how we will understand $B^{(4)}_r$ for $f_4\in L^2_w(\mathbb{R}^3)$.

The following lemma will be important later in the text.
\begin{lemma}
For every $k\in \mathbb{C}^+$, we have
\begin{equation}\label{eee1}
\sup_{r>0}\|\mu_1(r,\sigma)\|_{L^2(\mathbb{S}^2)}\lesssim \left( \frac{1}{|k|^2}+\frac{1}{(\Im k)^2}+\frac{|k|}{(\Im k)^{3/2}}+\frac{|k|}{(\Im k)^2} +\frac{1}{\Im k}   \right)^{1/2}\|f_1\|_2\,,
\end{equation}
\begin{equation}\label{eee2}
\sup_{r>0}\|\mu_2(r,\sigma)\|_{L^2(\mathbb{S}^2)}\lesssim \left(\frac{1}{|k|^2}+\frac{1}{|k|^2(\Im k)^2}+\frac{1}{|k|(\Im k)^{1/2}}+\frac{1}{|k|(\Im k)^2}                 \right) ^{1/2}\|f_2\|_{L^2_w}\,,
\end{equation}
\begin{equation}\label{eee3}
\sup_{r>0}\|\mu_3(r,\sigma)\|_{L^2(\mathbb{S}^2)}\lesssim  \left(\frac{1}{|k|^2}+\frac{1}{|k|^2(\Im k)^2}+\frac{1}{|k|(\Im k)^{1/2}}+\frac{1}{|k|(\Im k)^2}                 \right)^{1/2} \|V\|_{\ell^2(\mathbb{Z}^+),L^\infty}\|f_3\|_2\,.
\end{equation}
\end{lemma}

\begin{proof}
We will give all detail for the first estimate. The others can be proved similarly. One way to obtain the estimates of this type is to go on the Fourier side in $\sigma$ in formula \eqref{fo1} and control the convergence of the resulting integral in $|y|$. However, it is more instructive to proceed differently. By the standard approximation argument,  it is enough to assume that $f_1$  is smooth and is supported on annulus $\{x:a_1<|x|<a_2, a_1>0\}$ and $\|f_1\|_2=1$.  Having made these assumptions, we immediately obtain
\begin{equation}\label{regu}
\lim_{x\to 0}\mu_1=0, \quad \nabla\mu_1\in L^\infty(B_{a_1/2}(0)), \quad \mu_1\in L^\infty(\mathbb{R}^3), \quad \lim_{|x|\to\infty} \nabla \mu_1=0\,.
\end{equation}
Consider the following five quantities: $m,m_1,M,\widehat M,A$.
\[
m(r)\dd r^{-2}\int_{|x|=r}|\mu_1|^2d\sigma_x,\quad\,
m_1(r)\dd r^{-2}\int_{|x|=r}|\partial_r\mu_1|^2d\sigma_x\,,
\]
and
\[
M\dd \sup_{r>0}\int_r^{r+1}m(\rho)d\rho, \quad \widehat M\dd\sup_{r>0}m(r), \quad A\dd\int_{\mathbb{R}^3}\frac{|\nabla\mu_1|^2}{|x|^2}dx\,.
\]
From \eqref{regu}, we get
\begin{equation}\label{ue2}
\lim_{r\to 0}m=0, \quad m_1(r)\in L^\infty(0,a_1/2)\,.
\end{equation}
Notice that 
\[
\sup_{r>0}\|\mu_1(r,\sigma)\|_{L^2(\mathbb{S}^2)}=\sqrt{\widehat M}
\]
and our goal is to estimate $\widehat M$. Consider \eqref{main-eq00}, multiply both sides by $\overline{\mu}_1/|x|^2$ and integrate over the annulus $\{x: r_1<|x|<r_2\}$ where $r_1>0$.
\begin{equation}\label{eq21}
-\int_{r_1<r<r_2} \frac{\Delta \mu_1 \bar\mu_1}{|x|^2}dx-2 \int_{r_1<r<r_2} \left(ik-\frac{1}{|x|}\right)\frac{\bar\mu_1\partial_r\mu_1}{|x|^2}dx=\int_{r_1<r<r_2} {\rm div}\Bigl(\frac{f_1}{|x|} \Bigr)\bar{\mu}_1dx\,.
\end{equation}
Let $r_1<a_1$. Then, integrating by parts in the last integral gives
\[
\int_{r_1<r<r_2} {\rm div}\Bigl(\frac{f_1}{|x|} \Bigr)\bar{\mu}_1dx=-\int_{r_1<r<r_2} \frac{f_1 \nabla\bar{\mu}_1}{|x|} dx+\int_{|x|=r_2} \frac{f_1\bar \mu_1}{|x|}n d\sigma_x
\]
and $n$ is a normal vector at $x$. Integrate by parts in the first integral in \eqref{eq21}, to get
\begin{equation}\label{perpe}
-\int_{r_1<r<r_2} \frac{\Delta \mu_1 \bar\mu_1}{|x|^2}dx=\int_{r_1<r<r_2} \frac{|\nabla\mu_1|^2}{|x|^2}dx-2\int_{r_1<r<r_2} \frac{\bar\mu_1 \partial_r\mu_1 }{|x|^3}dx-I_2+I_1\,,
\end{equation}
where
\[
I_2=r_2^{-2}\int_{r=r_2}\bar \mu_1\partial_r\mu_1 d\sigma_x, \quad   I_1=r_1^{-2}\int_{r=r_1}\bar \mu_1\partial_r \mu_1 d\sigma_x\,.
\]
Notice that the second term in the right hand side of \eqref{perpe}  will cancel the same term in the second integral in \eqref{eq21}. We get
\begin{equation}\label{eq2}
\int_{r_1<r<r_2} \frac{|\nabla \mu_1|^2 }{|x|^2}dx-2ik \int_{r_1<r<r_2}  \frac{ \bar\mu_1\partial_r\mu_1}{|x|^2}dx=I_2-I_1-\int_{r_1<r<r_2} \frac{f_1 \nabla\bar{\mu}_1}{|x|} dx+\int_{|x|=r_2} \frac{f_1\bar \mu_1}{|x|}nd\sigma_x\,.
\end{equation}
Divide this formula by $-ik$ and take the real part of both sides. Making use of the identity
\[
\int_{r_1<r<r_2}\frac{\bar \mu_1\partial_r{\mu_1}+\mu_1\partial_r\bar\mu_1}{|x|^2}dx=m(r_2)-m(r_1),
\]
we  get
\[
\frac{\Im k}{|k|^2}\int_{r_1<r<r_2} \frac{|\nabla \mu_1|^2 }{|x|^2}dx+m(r_2)\leq m(r_1)   +  \frac{|I_2|+|I_1|}{|k|}
\]
\begin{equation}\label{ede}
+\frac{1}{|k|}
\left|\int_{r_1<r<r_2} \frac{f_1 \nabla\bar{\mu}_1}{|x|} dx\right|+\frac{1}{|k|}\left|\int_{|x|=r_2} \frac{f_1\bar \mu_1}{|x|}nd\sigma_x\right|\,.
\end{equation}
This bound will play the crucial role. We start by estimating $A$.
For that purpose, we send $r_1\to 0$ and $r_2\to\infty$. Since  $\mu_1\in L^\infty(\mathbb{R}^3)$ and $\lim_{|x|\to\infty}\mu'_1(x)=0$, we get   $\lim_{r_2\to\infty}I_2=0$.  Applying Cauchy-Schwarz inequality to $I_1$ we get
\[
|I_1|\le \Bigl(m(r_1)m_1(r_1)\Bigr)^{1/2}\,.
\]
The bounds \eqref{ue2} give $\lim_{r_1\to 0} I_1=0$. For the last term in \eqref{ede}, we get
\[
\lim_{r_2\to\infty}\int_{|x|=r_2} \frac{f_1\bar \mu_1}{|x|}nd\sigma_x=0
\]
because $f_1$ is compactly supported. Dropping the nonnegative term $\lim_{r_2\to\infty} m(r_2)$ and applying  Cauchy-Schwarz inequality along with $\|f_1\|_2=1$ to
\[
\left|\int_{r_1<r<r_2} \frac{f_1 \nabla\bar{\mu}_1}{|x|} dx\right|\le \|f_1\|_2\sqrt A=\sqrt A
\]
give us
\begin{equation}\label{g1}
\frac{\Im k}{|k|^2}A\lesssim  \frac{1}{|k|}  \sqrt A, \quad A\lesssim \frac{|k|^2}{(\Im k)^2}\,.
\end{equation}
Consider  \eqref{ede} again. Drop the first term, send $r_1\to 0$, and average in $r_2$ over $(r,r+1)$. This gives
\begin{eqnarray*}
\int_{r}^{r+1}m(r_2)dr_2\lesssim \frac{1}{|k|}\int_{r}^{r+1}|I_2|dr_2+\frac{1}{|k|}\int_{r}^{r+1}\left| \int_{|x|<r_2}\frac{f_1\nabla\overline{\mu}_1}{|x|}dx  \right|dr_2+\\
\frac{1}{|k|}\int_{r}^{r+1}
\left|\int_{|x|=r_2}\frac{f_1\overline{\mu}_1}{|x|}nd\sigma_x\right|dr_2\,.
\end{eqnarray*}
We use
\[
\int_{r}^{r+1} |I_2|dr\le\int_r^{r+1}\sqrt{m(r_2)m_1(r_2)}dr_2\le\left( \int_r^{r+1}m(r_2)dr_2\int_r^{r+1}m_1(r_2)    dr_2           \right)^{1/2}\le
\]
\[
\left(M\int_0^\infty m_1(r_2)dr_2\right)^{1/2}\le \sqrt{MA}
\]
and
\[
\int_{r}^{r+1}\left|\int_{|x|=r_2} \frac{f_1\bar \mu_1}{|x|}nd\sigma_x\right|dr_2\le \left(\int_{r<|x|<r+1} |f_1|^2dx\right)^{1/2}\left(\int_r^{r+1} m(r_2)dr_2\right)^{1/2}\le \sqrt M
\]
to write
\[
\int_{r}^{r+1}m(r_2)dr_2\le \frac{\sqrt{AM}+\sqrt{A}+\sqrt{M}}{|k|}\,.
\]
Taking  supremum of both sides over $r\in (0,\infty)$ gives an estimate on $M$
\begin{equation}\label{g2}
M\le \frac{1}{|k|}\left(   \sqrt A+\sqrt M +\sqrt{AM}\right)\,.
\end{equation}
Substituting the bound \eqref{g1} gives
\[
M\lesssim \frac{1}{|k|^2}+\frac{1}{(\Im k)^2}+\frac{1}{\Im k}\,.
\]
We are left with estimating $\widehat M$. Recall that
\[
m(r)=\int_{\mathbb{S}^2}|\mu_1(r,\sigma)|^2d\sigma_x\,.
\]
Differentiation in $r$ gives
\[
|m'(r)|\le {2}\int_{\mathbb{S}^2}|\mu_1(r,\sigma)\partial_r\mu_1(r,\sigma)|d\sigma_x\lesssim \sqrt{m(r)m_1(r)}, \quad
\int_r^{r+1} |m'|dr_2\lesssim \sqrt M \sqrt A\,.
\]
Writing for every $r>0$
\[
m(r)=m(\rho)+\int_r^{\rho}m'(t)dt, \quad m(r)\le \int_r^{r+1}m(\rho)d\rho+\int_{r}^{r+1}\int_r^\rho |m'(t)|dtd\rho \lesssim M+\sqrt{AM}
\]
and taking supremum in $r$ of both sides, gives
\begin{equation}\label{wt0}
\widehat M\lesssim M+\sqrt{AM}\,.
\end{equation}
This yields \eqref{eee1}. The estimates \eqref{eee2} and \eqref{eee3} can be obtained in a similar manner. For reader's convenience, we state the  estimates for $A,M,\widehat M$.

(1). For $\mu_2$, we bound
\[
\left|\int_{r_1<|x|<r_2}\frac{f_2\overline{\mu}_2}{|x|^2}dx\right|\lesssim \|f\|_{2,w}\sqrt M=\sqrt M
\]
by Cauchy-Schwarz inequality. This gives.
\[
A\lesssim \frac{|k|}{\Im k}\sqrt M, \quad M\lesssim \frac{\sqrt{M}+\sqrt{AM}}{|k|}\,.
\]
Solving these inequalities (using, e.g., lemma \ref{tra}) gives
\[
A\lesssim \frac{1}{\Im k}\left(1+\frac{1}{\Im k}\right), \quad M\lesssim \frac{1}{|k|^2}\left(1+\frac{1}{\Im^2k}\right)
\]
and \eqref{wt0} implies \eqref{eee2}.\smallskip

 (2). For $\mu_3$, the estimates are identical to those for $\mu_2$. In fact, for
 \[
 \int_{r_1<|x|<r_2}\frac{Vf_3\overline{\mu}_3}{|x|}dx
 \]
we write
\[
\left| \int_{r_1<|x|<r_2}\frac{Vf_3\overline{\mu}_3}{|x|}\right|\le \|f_3\|_2\left(  \int_{r_1}^{r_2}(\sup_{|x|=r}|V(x)|)^2m(r)dr          \right)^{1/2}\lesssim \sqrt{M}\|V\|_{\ell^2(\mathbb{Z}^+),L^\infty}
\]
and the rest follows.
\end{proof}

{\bf Remark.} Notice that the estimates for $M$ are better than for $\widehat M$. In fact, we can summarize them as

\begin{equation}\label{qeee11}
\sup_{r>0}\int_r^{r+1}\|\mu_1(\rho,\sigma)\|^2_{L^2(\mathbb{S}^2)}d\rho\lesssim
\left(\frac{1}{|k|^2}+\frac{1}{(\Im k)^2}+\frac{1}{\Im k}   \right)\|f_1\|_2^2\,,
\end{equation}
\begin{equation}\label{qeee21}
\sup_{r>0}\int_r^{r+1}\|\mu_2(\rho,\sigma)\|^2_{L^2(\mathbb{S}^2)}d\rho
\lesssim \left(\frac{1}{|k|^2}+\frac{1}{|k|^2(\Im k)^2}             \right) \|f_2\|_{L^2_w}^2\,,
\end{equation}
\begin{equation}\label{qeee31}
\sup_{r>0}\int_r^{r+1}\|\mu_3(\rho,\sigma)\|^2_{L^2(\mathbb{S}^2)}d\rho
\lesssim \left(\frac{1}{|k|^2}+\frac{1}{|k|^2(\Im k)^2}               \right) \|V\|^2_{\ell^2,L^\infty}\|f_3\|^2_2\,.
\end{equation}
After taking account of \eqref{chetvert}, these lemmas and  remark immediately imply the following theorem.
\begin{theorem} \label{ttt1}For every $k\in\mathbb{C}^+$, the operators $B^{(j)}_r$ are linear bounded operators from the corresponding Banach spaces to $L^2(\mathbb{S}^2)$ and
\begin{equation}
\sup_{r>0}\|B^{(1)}_r(f)\|_{L^2(\mathbb{S}^2)}\le C_1(k)\|f\|_2, \quad \sup_{r>0}\left(\int_r^{r+1}\|B^{(1)}_\rho(f)\|^2_{L^2(\mathbb{S}^2)}d\rho\right)^{1/2}\le C'_1(k)\|f\|_2\,,
\end{equation}
\begin{equation}
\sup_{r>0}\|B^{(2)}_r(f)\|_{L^2(\mathbb{S}^2)}\le C_2(k)\|f\|_{2,w}, \quad \sup_{r>0}\left(\int_r^{r+1}\|B^{(1)}_\rho(f)\|^2_{L^2(\mathbb{S}^2)}d\rho\right)^{1/2}\le C'_2(k)\|f\|_{2,w}\,,
\end{equation}
\begin{equation}
\sup_{r>0}\|B^{(3)}_r(f,V)\|_{L^2(\mathbb{S}^2)}\le C_3(k)\|f\|_2\|V\|_{\ell^2,L^\infty}, \quad \sup_{r>0}\left(\int_r^{r+1}\|B^{(3)}_\rho(f)\|^2_{L^2(\mathbb{S}^2)}d\rho\right)^{1/2}\le C'_3(k)\|f\|_2\|V\|_{\ell^2,L^\infty}\,,
\end{equation}
\begin{equation}
\sup_{r>0}\|B^{(4)}_r(f)\|_{L^2(\mathbb{S}^2)}\le (C_1(k)+C_2(k))\|f\|_{2,w}, \quad \sup_{r>0}\left(\int_r^{r+1}\|B^{(4)}_\rho(f)\|^2_{L^2(\mathbb{S}^2)}d\rho\right)^{1/2}\le (C'_1(k)+C_2'(k))\|f\|_{2,w}\,.
\end{equation}
and the estimates on $C_{1(2,3)}$ can be obtained from \eqref{eee1}, \eqref{eee2}, \eqref{eee3} ($C'_{1(2,3)}$ from  \eqref{qeee11}, \eqref{qeee21}, \eqref{qeee31} by taking the square root).
\end{theorem}
Having taken  $f$  as a function with compact support in the last  theorem, we can send $r\to\infty$ in the formula for each $B^{(j)}_r$ to get the limiting operators when $r\to\infty$
\begin{eqnarray}\label{eret}
B_r^{(1)}(f)\to B_\infty^{(1)}(f)=-(4\pi)^{-1}\int_{\mathbb{R}^3} \left(\frac{\widehat y}{|y|}+ik(\widehat y-\sigma) \right)e^{ik|y|(1-\langle \sigma,\widehat y\rangle)}fdy\,,\\\nonumber
B_r^{(2)}(f)\to B_\infty^{(2)}(f)=(4\pi)^{-1}\int_{\mathbb{R}^3} \frac{e^{ik|y|(1-\langle \sigma,\widehat y\rangle)}}{|y|}fdy\,,\\\nonumber
B_r^{(3)}(f)\to B_\infty^{(3)}(f)=(4\pi)^{-1}\int_{\mathbb{R}^3} e^{ik|y|(1-\langle \sigma,\widehat y\rangle)}Vfdy\,,\\\nonumber
B_r^{(4)}(f) \to B_\infty^{(4)}(f)=-(4\pi)^{-1}ik\int_{\mathbb{R}^3} (\widehat y-\sigma) e^{ik|y|(1-\langle \sigma,\widehat y\rangle)}fdy
\end{eqnarray}
and this convergence is uniform in $\sigma\in \mathbb{S}^2$. Moreover, to estimate  the limiting operators, we can use the bounds from the remark above.  Thus, 
\begin{equation}\label{erat}
\|B^{(j)}_\infty(f)\|_2\lesssim C_j'(k)\|f\|_2, \quad j=1,\ldots,4\,,
\end{equation}
if we keep assumption that $f$ is compactly supported. \smallskip

If the condition on the support of $f$ is  dropped, the integrals in the right hand sides of \eqref{eret} do not have to converge absolutely for given $\sigma$. However, being defined on the set of smooth functions with compact support, which is dense set in $L^2({\mathbb{R}^3})$, these operators are bounded from the corresponding spaces to $L^2(\mathbb{S}^2)$ as follows from \eqref{erat}. Thus  \eqref{erat} hold for all $f\in L^2({\mathbb{R}^3})$.

We will also need  the following standard result on so-called ``strong convergence'':
\begin{lemma}\label{ttt2}
For each $j=1,\ldots,4$ and $k\in \mathbb{C}^+$, we have
\[
\lim_{r\to\infty}\|B_r^{(j)}(f)- B_\infty^{(j)}(f)\|_{L^2(\mathbb{S}^2)}=0\,,
\]
where $f$ is taken from the corresponding spaces, i.e., $f\in L^2(\mathbb{R}^3)$ for $j=1,3$ and  $f\in L^2_w(\mathbb{R}^3)$ for $j=2,4$.
\end{lemma}
\begin{proof}
We will prove lemma for $j=1$, the other cases are similar. Given any $\epsilon>0$, we take $R_2$ so large that $\|f\chi_{|x|>R_2}\|_2\leq (C_1)^{-1}\epsilon/3$. Having fixed this $R_2$, we notice that
$
B_r^{(1)}(f\chi_{|x|<R_2})
$ converges to $B_\infty^{(1)}(f\chi_{|x|<R_2})$ uniformly on $\mathbb{S}^2$ as $r\to\infty$. For the tails, we have
\[
\|B_r^{(1)}(f\chi_{|x|>R_2})\|_2 <\epsilon/3, \quad \|B_\infty^{(1)}(f\chi_{|x|>R_2})\|_2<\epsilon/3
\]
and the first bound holds uniformly in $r$. This finishes the proof of the lemma.
\end{proof}
\vspace{0.5cm}

\subsection{The proofs of  main results}

Consider $u$ that solves
\begin{equation}\label{ur1}
-\Delta u+Vu=k^2u+f\,,
\end{equation}
where $V\in L^\infty(\mathbb{R}^3)$, $k\in \Sigma$ and
\begin{itemize}
\item either $\|f\|_2\leq 1, \, {\rm supp }(f)\subset B_{R}(0)\,,$ where $R=R_f$ is not fixed and can depend on $f$.

\item or $f=\delta_0$ in which case
$u=G(x,0,k^2)$. We can let $R=0$ in that case.
\end{itemize}

To control the asymptotics of $u(x,k)$ for large $x$ we will employ the strategy used in the previous section already.

 Define
\begin{equation}\label{mumu}
\mu\dd 4\pi |x|e^{-ik|x|}u\,.
\end{equation}
We have the following integral equation for $\mu$:
\begin{equation}\label{rrr1}
\mu(x,k)=\mu_0(x,k)-\int \frac{|x|e^{ik(|x-y|+|y|-|x|)}}{|x-y||y|}V\mu(y,k)dy\,,
\end{equation}
where
\[
\mu_0(x,k)=\left\{
\begin{array}{ccc}
1,& {\rm if}& f=\delta_0,\\
4\pi |x|e^{-ik|x|}R^0_zf, &{\rm if}& f\in L^2(\mathbb{R}^3)\,.
\end{array}
\right.
\]
This $\mu$  solves 
\begin{equation}\label{main-eq1}
-\Delta\mu -2\left(ik-\frac{1}{|x|}\right)\mu_r+V\mu=0,\quad  |x|>R_f\,.
\end{equation}
As before, for each $r>R_f$, we introduce
\[
m(r)\dd r^{-2}\int_{|x|=r}|\mu|^2d\sigma_x, \quad
m_1(r)\dd r^{-2}\int_{|x|=r}|\mu_r|^2d\sigma_x\,,
\]
and
\[
M(r)\dd \sup_{r<\rho}\int_\rho^{\rho+1}m(t)dt, \quad \widehat M(r)\dd \sup_{r<\rho} m(\rho), \quad A(r)\dd \int_{|x|>r} \frac{|\nabla\mu|^2}{|x|^2}dx\,.
\]
Clearly $m$ and $m_1$ are always finite since $u\in \cal{H}^2(B_{R+\epsilon}^c)$ for any $\epsilon>0$ and $M, \widehat M, A$ might be infinite.\smallskip

In the next lemma, we will estimate $A$.

\begin{lemma}Suppose $V$ is compactly supported and $k\in \Sigma$. Then,
\begin{equation}\label{key2}
A(r)=\int_{r<|x|} \frac{|\nabla \mu|^2 }{|x|^2}dx\leq  - \int_{r<|x|} V\frac{|\mu|^2}{|x|^2}dx+\frac{C|k|^2}{\Im k}\left(m(r)+|k|^{-1}\left(m(r)m_1(r))^{1/2}\right)\right)
\end{equation}
for every $r>R_f$.
\end{lemma}

\begin{proof}

Consider \eqref{main-eq1}, multiply by $\bar{\mu}/|x|^2$ and integrate over the annulus $\{x: r_1<|x|<r_2\}$ where $r_1>R_f$.
\begin{equation}\label{eq24}
-\int_{r_1<r<r_2} \frac{\Delta \mu \bar\mu}{|x|^2}dx-2 \int_{r_1<r<r_2} \left(ik-\frac{1}{|x|}\right)\frac{\mu_r\bar\mu}{|x|^2}dx+\int_{r_1<r<r_2} V\frac{|\mu|^2}{|x|^2}dx=0\,.
\end{equation}
Arguing as before, we get

\begin{equation}\label{eq25}
\int_{r_1<r<r_2} \frac{|\nabla \mu|^2 }{|x|^2}dx-2ik \int_{r_1<r<r_2}  \frac{\mu_r \bar\mu}{|x|^2}dx+  \int_{r_1<r<r_2} V\frac{|\mu|^2}{|x|^2}dx= I_2-I_1
\end{equation}
with
\begin{equation}\label{ioio}
I_2=r_2^{-2}\int_{r=r_2}\mu_r\bar \mu d\sigma, \quad   I_1=r_1^{-2}\int_{r=r_1}\mu_r\bar \mu d\sigma
\end{equation}
and
\begin{equation}\label{key1}
\frac{\Im k}{|k|^2}\int_{r_1<r<r_2} \frac{|\nabla \mu|^2 }{|x|^2}dx+m(r_2)   +\frac{\Im k}{|k|^2}  \int_{r_1<r<r_2} V\frac{|\mu|^2}{|x|^2}dx=m(r_1)-\Re\left(   \frac{I_2-I_1}{ik}   \right)\,.
\end{equation}
 We take $r_1=r$ and use
 \begin{equation}\label{lopot}
 \lim_{r_2\to \infty} \nabla \mu=0, \quad \mu\in L^\infty(B_R^c)
 \end{equation}
 to establish that
 \[
 \lim_{r_2\to\infty}I_2=0\,.
 \]
 $\Bigl(   \Bigr.$Indeed, we have
\[
G(x,y,k^2)=G^0(x,y,k^2)-\int_{{\rm supp} (V)} G^0(x,\xi,k^2)V(\xi)G(\xi,y,k^2)d\xi\,.
\]
From \eqref{leret}, $\|V(\xi)G(\xi,y,k^2)\|_2<\infty$, and we can take $|x|\to\infty$ to establish asymptotics
\[
G(x,y,k^2)=\frac{e^{ik|x|}}{4\pi |x|}\Bigl(   e^{-ik\langle \widehat x,y\rangle}-\int_{{\rm supp} (V)}e^{-ik\langle \widehat x,\xi\rangle } V(\xi)G(\xi,y,k^2)d\xi +o(1)\Bigr)
\]
for $G$ and analogous statement for the gradient.
This gives \eqref{lopot} $\Bigl. \Bigr)$.\smallskip

Sending $r_2\to\infty$ in \eqref{key1} gives us
\begin{equation}\label{key202}
\frac{\Im k}{|k|^2}\int_{r<|x|} \frac{|\nabla \mu|^2 }{|x|^2}dx+m(\infty)   +\frac{\Im k}{|k|^2}  \int_{r<|x|} V\frac{|\mu|^2}{|x|^2}dx\leq m(r)+C|k|^{-1}\left(m(r)m_1(r)\right)^{1/2}
\end{equation}
after applying  Cauchy-Schwarz inequality to $I_1$.
\end{proof}
Assuming positivity of $V$, we can get rid of the assumption that $V$ is compactly supported as can be seen from the following Corollary.
\begin{corollary}Suppose $V\in L^\infty(\mathbb{R}^3)$ and $V\geq 0$. Then,
\begin{equation}\label{key3}
\int_{r<|x|} \frac{|\nabla \mu|^2 }{|x|^2}dx+ \int_{r<|x|} V\frac{|\mu|^2}{|x|^2}dx\leq \frac{C|k|^2}{\Im k}\left(m(r)+|k|^{-1}\left(m(r)m_1(r) \right)^{1/2}\right)
\end{equation}
for every $z=k^2\notin \sigma(H)\cup [0,\infty)$ and $r>R_f$.
\end{corollary}
\begin{proof}
Consider $V_{(\rho)}=V \alpha_\rho(|x|)$ (see the formula right before \eqref{tru} for the definition of $\alpha_\rho$) . Then, assuming that supp$(f)\subset B_R(0)$ again and
comparing two solutions $u$ and $u_\rho$
\[
-\Delta u+Vu=zu+f,\,
-\Delta u_\rho+V_{(\rho)}u_\rho=zu_\rho+f
\]
we get
\[
\lim_{\rho\to\infty}\|u-u_\rho\|_{L^\infty(K)}=0, \, \lim_{\rho\to\infty}\|\Delta u-\Delta u_\rho\|_{L^\infty(K)}=0
\]
for every $K$, a compact in $({B_R}_f)^c$. $\Bigl(\Bigr.$ The proof of that fact is easy  but we will present it. Assume $f\in L^2(\mathbb{R}^3)$. Subtracting two equations, we get
\[
u-u_\rho=\int_{\mathbb{R}^3} G_{(\rho)}(x,\xi,z)V^{(\rho)}u d\xi\,.
\]
Cauchy-Schwarz implies
\[
\|u-u_\rho\|_{L^\infty({\mathbb{R}^3)}}\le  \sup_{x}\|G_\rho(x,\cdot,z)\|_2\|V^{(\rho)} u\|_2\,.
\]
The last factor converges to zero as $\rho\to \infty$. For the first one, we can apply \eqref{leret} to get
\[
\|G_\rho(x,\cdot,z)\|_2\lesssim C(z)\|V\|_\infty
\]
independently of $\rho$. Thus,
$
\lim_{\rho\to\infty}\|u-u_\rho\|_{L^\infty({\mathbb{R}^3)}}=0\,.
$
Then,
$
\|\Delta (u-u_\rho)\|_{L^\infty(K)}=0
$
follows from the equations. The case when $f=\delta_0$ can be handled similarly.$\Bigl.\Bigr)$

Therefore, $\lim_{\rho\to\infty}\|u-u_\rho\|_{\cal{H}^2(U)}=0$ for every annulus $U$  in $B_R^c(0)$. From the definition of $\mu,\mu_\rho$, we obtain
$\|\mu-\mu_\rho\|_{\cal{H}^2(U)}\to 0$ and application of the previous lemma yields
\begin{equation}\label{key4}
\int_{r<|x|<R_1} \frac{|\nabla \mu_{(\rho)}|^2 }{|x|^2}dx+ \int_{r<|x|<R_1} V_\rho\frac{|\mu_{(\rho)}|^2}{|x|^2}dx\leq
 \frac{C|k|^2}{\Im k}\left(m_{(\rho)}(r)+|k|^{-1}\left(m_{(\rho)}(r){m_1}_{(\rho)}(r) \right)^{1/2}\right)
\end{equation}
with any $R_1>r$. Taking $\rho\to\infty$ first, we get
\begin{equation}\label{key404}
\int_{r<|x|<R_1} \frac{|\nabla \mu|^2 }{|x|^2}dx+ \int_{r<|x|<R_1} V\frac{|\mu|^2}{|x|^2}dx\leq
 \frac{C|k|^2}{\Im k}\left(m(r)+|k|^{-1}\left(m(r)m_1(r) \right)^{1/2}\right)\,.
\end{equation}
Now it is only left to take $R_1\to\infty$.

\end{proof}

\begin{lemma} \label{mre} If $V$ satisfies \eqref{main-assump}, $k\in \Sigma$, and $r>R_f$, then $A(r), M(r)<\infty$ and
\begin{eqnarray}\label{key6}
A(r)\lesssim  M(r)\|Q\|^2_{\ell^2([r],\infty),L^\infty}+\frac{M(r)\|Q\|_{\ell^2([r],\infty),L^\infty}}{\sqrt{r+1}}+m(r)\|Q\|_{\ell^2([r],\infty),L^\infty}\\
+\frac{|k|^2}{\Im k}\left(m(r)+\frac{\sqrt{m(r)m_1(r)}}{|k|} \right)\nonumber\,,
\end{eqnarray}
\begin{eqnarray}\nonumber
M(r)\lesssim\frac{\Im k}{|k|^2} \Bigl(\sqrt{A(r)M(r)}\|Q\|_{\ell^2([r],\infty),L^\infty}+
m(r)\|Q\|_{\ell^2([r],\infty),L^\infty}+\frac{M(r)\|Q\|_{\ell^2([r],\infty),L^\infty}}{\sqrt{r+1}}
\Bigr)+\\ m(r)+
\frac{\sqrt{m(r)m_1(r)}}{|k|}+\frac{\sqrt{M(r)A(r)}}{|k|}\,.\label{key88}
\end{eqnarray}
For $\widehat{M}(r)$, we have a bound
\begin{equation}\label{wt}
\widehat M(r)\lesssim M(r)+\sqrt{A(r)M(r)}\,.
\end{equation}
\end{lemma}
\begin{proof}

We first consider truncations $V_{(\widehat R)}$ defined as in \eqref{tru}. Given $k\in \Sigma(H)$, we have $k\in \Sigma(H_{(\widehat R)})$ when $\widehat R>\widehat R_0$ and $\widehat R_0$ is large enough. Our first goal is to prove the estimates \eqref{key6},\eqref{key88}, and \eqref{wt} for $V_{(\widehat R)}$ with all constants independent of $\widehat R$. Then, we will take the limit as $\widehat R\to\infty$. 

In the calculations below, from the formula \eqref{keke} to \eqref{porosyo}, all functions involved depend on $\widehat R$ and we suppress this dependence to make reading easier. We notice that
$
\|V_{(\widehat R)}\|\lesssim \|V\|
$
and this will provide the necessary independence of $\widehat R$.
We start by proving \eqref{key6}.     Consider \eqref{key2}. Integration by parts gives
\begin{equation}\label{keke}
\left|\int_{r<|x|} V\frac{|\mu|^2}{|x|^2}dx\right|\lesssim \int_{r<|x|}|Q|\left(\frac{|\mu| |\nabla \mu|}{|x|^2}+\frac{|\mu|^2}{|x|^3}\right)dx+I_3, \quad I_3=\int_{|x|=r}|Q|\frac{|\mu|^2}{r^2}d\sigma_x\,.
\end{equation}
For $I_3$,
$
 I_3\lesssim m(r) \|Q\|_{\ell^2([r],\infty),L^\infty}\,.
$
The integral can be estimated we follows
\[
        \int_{r<|x|} |Q|\left(\frac{|\mu| |\nabla \mu|}{|x|^2} +\frac{|\mu|^2}{|x|^3}   \right)  dx       \le    \left(\int_{r<|x|}\frac{|Q|^2|\mu|^2}{|x|^{2}}dx\right)^{1/2}\left(    \int_{r<|x|}\frac{|\nabla\mu|^2}{|x|^2}  dx   \right)^{1/2}
+\int_{r<|x|}\frac{|Q||\mu|^2}{|x|^{3}}dx\,.
\]
For the first integral,
\[
\int_{r<|x|}\frac{|Q|^2|\mu|^2}{|x|^{2}}dx\le M(r)\sum_{n=[r]}^\infty \sup_{n<|x|<n+1} |Q|^2=M(r)\|Q\|^2_{\ell^2([r],\infty),L^\infty}
\]
and
\[
\int_{r<|x|}\frac{|Q||\mu|^2}{|x|^{3}}dx\lesssim M(r)\sum_{n=[r]}^\infty \frac{1}{n+1}\sup_{n<|x|<n+1}|Q(x)|\lesssim M(r)(r+1)^{-1/2}\|Q\|_{\ell^2([r],\infty),L^\infty}\,.
\]
Collecting the estimates, we write
\begin{equation}\label{collectivo}
\left|\int_{r<|x|} V\frac{|\mu|^2}{|x|^2}dx\right|\lesssim m(r) \|Q\|_{\ell^2([r],\infty),L^\infty}+\sqrt{A(r)M(r)}\|Q\|_{\ell^2([r],\infty),L^\infty}+M(r)(r+1)^{-1/2}\|Q\|_{\ell^2([r],\infty),L^\infty}\,.
\end{equation}
Substituting into \eqref{key2} and solving inequality for $A$, we get \eqref{key6}.\smallskip

 Consider \eqref{key1} and take $r_1=r$. Then, we  drop the first term and average in $r_2$ from $\rho$ to $\rho+1$ assuming $\rho>r$. We use \eqref{collectivo} to get
\begin{eqnarray}\nonumber
\int_{\rho}^{\rho+1}m(r_2)dr_2\lesssim m(r)+\frac{\Im k}{|k|^2} \Bigl(
\sqrt{A(r)M(r)}\|Q\|_{\ell^2([r],\infty),L^\infty}+\frac{M(r)\|Q\|_{\ell^2([r],\infty),L^\infty}}{\sqrt{r+1}}+m(r)\|Q\|_{\ell^2([r],\infty),L^\infty}
\Bigr)+    \\\label{clop}
\frac{1}{|k|}\int_{\rho}^{\rho+1} \sqrt{m(r_2)m_1(r_2)}dr_2+\frac{\sqrt{m(r)m_1(r)}}{|k|}\,.\hspace{2cm}
\end{eqnarray}
We apply Cauchy-Schwarz estimate to get
\[
\int_{\rho}^{\rho+1} \sqrt{m(r_2)m_1(r_2)}dr_2\leq \left(\int_\rho^{\rho+1} m(r_2)dr_2\right)^{1/2}A^{1/2}(r)\leq \sqrt{A(r)M(r)}\,.
\]
Taking supremum in $\rho$ of both sides in \eqref{clop}, we get an estimate
\begin{eqnarray}\nonumber
M(r)\lesssim\frac{\Im k}{|k|^2} \Bigl(\sqrt{A(r)M(r)}\|Q\|_{\ell^2([r],\infty),L^\infty}+
m(r)\|Q\|_{\ell^2([r],\infty),L^\infty}+\frac{M(r)\|Q\|_{\ell^2([r],\infty),L^\infty}}{\sqrt{r+1}}
\Bigr)+\\ m(r)+
\frac{\sqrt{m(r)m_1(r)}}{|k|}+\frac{\sqrt{M(r)A(r)}}{|k|}\,,\label{vtor}
\end{eqnarray}
which is \eqref{key88}. The proof of \eqref{wt} is identical to \eqref{wt0}. Notice that it is the support of $V$ being compact that allows us to say that $M(r)<\infty$.  
Thus, we proved \eqref{key6},\eqref{key88}, and \eqref{wt} for truncated potential $V_{(\widehat R)}$ assuming  $k\in \Sigma(H)$. 
Let us study the first two estimates. Fixing $k\in \Sigma$, we can take $r(k,V)$ so large that  inequalities take the following simpler form for all $r\ge r(k,V)$ because of  $\lim_{r\to\infty}\|Q\|_{\ell^2([r],\infty),L^\infty}=0$. 
\begin{eqnarray}\label{perun1}
A(r)\lesssim  M(r)\|Q\|^2_{\ell^2([r],\infty),L^\infty}+\frac{M(r)\|Q\|_{\ell^2([r],\infty),L^\infty}}{\sqrt{r+1}}
+\\
\frac{|k|^2}{\Im k}\left(m(r)+|k|^{-1}\left(m(r)m_1(r) \right)^{1/2}\right)\,,\\
M(r)\lesssim m(r)+
\frac{\sqrt{m(r)m_1(r)}}{|k|}+\frac{\sqrt{M(r)A(r)}}{|k|}\,.\hspace{2cm}\label{perun2}
\end{eqnarray}
Substituting the first estimate into the second gives
\begin{equation}\label{poros}
M(r)\lesssim \left(1+\frac{1}{\Im k}\right)\left(m(r)+
\frac{\sqrt{m(r)m_1(r)}}{|k|}\right)\,,
\end{equation}
\begin{equation}\label{porosyo}
A(r)\lesssim \frac{|k|^2}{\Im k}\left(m(r)+
\frac{\sqrt{m(r)m_1(r)}}{|k|}\right)\,,
\end{equation}
if $r\ge r(k,V)$ and $r(k,V)$ is large enough. Now, we  will send $\widehat R\to\infty$. To do that, we first notice that lemma \ref{alemma} implies
\[
\lim_{\widehat R\to\infty}\|\mu_{(\widehat R)}(x,k)-\mu(x,k)\|_{\cal{H}^2(r_1<|x|<r_2)}=0
\]
for every $r_{1(2)}: R_f<r_1<r_2<\infty$. Therefore, $m_{(\widehat R)}(\rho)\to m(\rho), {m_1}_{(\widehat R)}(\rho)\to m_1(\rho)$ and
\begin{equation}\label{keert}
\int_{r_1<|x|<r_2}\frac{|\nabla \mu_{(\widehat R)}|^2}{|x|^2}dx\to \int_{r_1<|x|<r_2}\frac{|\nabla \mu|^2}{|x|^2}dx\,.
\end{equation}
Taking $r>r(k,V)$ and sending $\widehat R\to\infty$ in \eqref{poros} and \eqref{porosyo}, we obtain \eqref{poros} and \eqref{porosyo} for $V$ itself. This implies that $M(r)<\infty$ and $A(r)<\infty$  for all $r>R_f$. Then, we can send $\widehat R\to\infty$ in \eqref{key6} and \eqref{key88} and this proves the lemma.
\end{proof}

Now, we are ready to prove the main results of the first part of the paper.

\begin{proof}{\it (of Theorem \ref{t1})}
Notice that given $y$, we can consider $V_{(y)}(x)=V(x-y)$. By lemma  \ref{shift},
\[
\|V_{(y)}\|_{\ell^2(\mathbb{Z}^+),L^\infty}\lesssim 1+|y|\,.
\]
Thus, we can assume that $y=0$ without loss of generality.

Take \eqref{ur1} with $f=\delta_0$. Having fixed $\Pi(a,b,h)$, we examine the estimates \eqref{key6}, \eqref{key88}, \eqref{wt}. 
By taking $r_0(a,b,h)$ sufficiently large, we can guarantee  \eqref{perun1} and \eqref{perun2} for all $k\in \Pi(a,b,h)$. Therefore, \eqref{poros} and \eqref{porosyo} hold as well and we only need to obtain upper bounds for $m(r_0)$ and $m_1(r_0)$ uniformly over $k\in \Pi(a,b,h)$.
Recall that (check \eqref{mumu})
\[
\mu=4\pi |x| e^{-ik|x|}G(x,0,k^2)\,.
\]
Thus, the lemma \ref{lm2} implies
\begin{equation}\label{kuzma1}
m(r_0)\leq C(a,b,h,\|V\|_\infty) \frac{1}{\Im k}\left(1+\frac{\|V\|_\infty^2}{(\Im k)^2}\right)
\end{equation}
and
\begin{equation}\label{kuzma2}
m_1(r_0)\leq C(a,b,h,\|V\|_\infty) \frac{1}{\Im k}\left(1+\frac{\|V\|_\infty^2}{(\Im k)^2}\right)\,.
\end{equation}
Substitution into \eqref{poros}, \eqref{porosyo} gives
\[
M(r_0)\lesssim \frac{C(a,b,h,\|V\|)}{\Im^4k}, \,A(r_0)\lesssim \frac{C(a,b,h,\|V\|)}{\Im^4k}
\]
and these estimates can be extended to all $r>R_f+1$ because we can use the estimates from lemma \ref{lm2} for $r\in [R_f+1,r_0]$. Now,  \eqref{wt} finishes the proof.
\end{proof}

\begin{proof}{\it (of Theorem \ref{t2})} We again assume that $y=0$ without any loss of generality. Now, we can write
\begin{equation}\label{poi1}
\mu(x,k)=1-\int_{\mathbb{R}^3} \frac{|x|e^{ik(|x-y|+|y|-|x|)}}{|x-y||y|}V\mu(y,k)dy=1-I_0-I_1\,,
\end{equation}
where
\[
I_0=\int_{{|y|<1}} \frac{|x|e^{ik(|x-y|+|y|-|x|)}}{|x-y||y|}V\mu(y,k)dy,\quad   I_1=\int_{|y|>1} \frac{|x|e^{ik(|x-y|+|y|-|x|)}}{|x-y||y|}V\mu(y,k)dy\,.
\]
Clearly, we only need to consider the limiting behavior of $I_1(r\sigma)$ where $x=r\sigma, \sigma\in \mathbb{S}^2$ and $r\to\infty$ since $I_0$ has all required properties. To this end, we write
\[
V\mu=({\rm div}\, Q)\mu=|y|{\rm div}\,(Q\mu |y|^{-1})-|y|Q\frac{\nabla\mu}{|y|}+Q\mu\frac{y}{|y|^2}
\]
and,  after integration by parts in the first term,
\begin{equation}\label{rew}
I_1=I_{1,1}+I_{1,2}+I_{1,3}+I_{1,4}\,,
\end{equation}
where,
\[
I_{1,1}=-|x|\int_{|y|=1} \frac{e^{ik(|x-y|+|y|-|x|)}}{|x-y||y|}Q(y)\mu(y,k)nd\sigma_y\,,
\]
\[
I_{1,2}=-|x|\int_{|y|>1} \nabla_y\left(\frac{e^{ik(|x-y|+|y|-|x|)}}{|x-y|}\right)Q(y)\frac{\mu(y,k)}{|y|}dy=4\pi B_r^{(4)}\left(Q\frac{\mu}{|y|}\chi_{|y|>1}\right), \quad ({\rm see}\,\, \eqref{lk1})
\]
\[
I_{1,3}=-|x|\int_{|y|>1} \left(\frac{e^{ik(|x-y|+|y|-|x|)}}{|x-y||y|}\right)Q(y)\nabla_y\mu(y,k)dy=-4\pi B_r^{(3)}\left(Q,\frac{\nabla\mu}{|y|}\chi_{|y|>1}  \right), \quad ({\rm see}\,\, \eqref{fo3})
\]
\[
I_{1,4}=|x|\int_{|y|>1} \left(\frac{e^{ik(|x-y|+|y|-|x|)}}{|x-y||y|}\right)|y|Q\mu\frac{y}{|y|^3}dy=4\pi B_r^{(3)}\left(Q,\frac{\mu y}{|y|^3}  \chi_{|y|>1}\right)
\]
if we denote $r=|x|$. The term $I_{1,1}$ clearly has required asymptotics as $r\to\infty$ so we focus on the other terms.
Lemma \ref{mre} gives
\begin{equation}\label{fifi}
\frac{\nabla\mu(x,k)}{|x|}\in L^2(B_1^c(0)), \quad \sup_{r>1} \frac{1}{r^2}\int_{|x|=r}|\mu(x,k)|^2d\sigma_x<\infty
\end{equation}
for every $k\in \Sigma$. Therefore,
\[
Q\frac{\mu}{|y|}\in L^2(B_1^c(0)), \, \frac{\mu y}{|y|^3}\in  L^2(B_1^c(0))
\]
and  lemma \ref{ttt2} gives the claimed convergence.  The analyticity of the limit follows from the analyticity of $G(x,y,k^2)$ in $k\in \Sigma$.

\end{proof}

\begin{proof}{\it (of Theorem \ref{t3})}
We again assume that $y=0$ and $\arg k\in (\delta,\pi-\delta)$. The proof  will proceed in two steps. Consider \eqref{key6} and \eqref{key88} when $|k|$ is large and $\arg k\in (\delta,\pi-\delta)$. In our case, $R_f=0$.
We send $r\to 0$ so we  need bounds for $\limsup_{r\to 0}m(r)$ and $\limsup_{r\to 0}m_1(r)$. That will allows us to get estimates on $M(0)$ and $A(0)$.
Then, we will write
\begin{equation}\label{again1}
 \mu(x,k)=1-\int_{\mathbb{R}^3} \frac{|x|e^{ik(|x-y|+|y|-|x|)}}{|x-y||y|}V\mu(y,k)dy
\end{equation}
and will use theorem \ref{ttt1}. To this end, we first write
\[
\mu(x,k)=1-\int_{\mathbb{R}^3} \frac{|x|e^{ik(|x-y|-|x|)}}{|x-y|}VG(y,0,k^2)dy\,.
\]
Lemma \ref{pop} implies
\begin{equation}\label{kert}
\left|\int_{\mathbb{R}^3} \frac{e^{ik(|x-y|-|x|)}}{|x-y|}VG(y,0,k^2)dy\right|\lesssim \|V\|_\infty \left|\int_{\mathbb{R}^3} \frac{e^{-\Im k (|x-y|-|x|)}}{|x-y|}
\left(\frac{e^{-\Im k |y|}}{|y|}+C(k,\|V\|_\infty)\right)  dy\right|\,,
\end{equation}
where the last expression is bounded when $x\to 0$. Thus, $\lim_{r\to 0}m(r)=1$. For the gradient,
\[
\nabla \mu=-\frac{x}{|x|}\int_{\mathbb{R}^3} \frac{e^{ik(|x-y|+|y|-|x|)}}{4\pi|x-y||y|}V\mu(y,k)dy-|x|\nabla\left(\int_{\mathbb{R}^3} \frac{e^{ik(|x-y|+|y|-|x|)}}{4\pi|x-y||y|}V\mu(y,k)dy\right)=J_1+J_2\,.
\]
We have
\[
\limsup_{x\to 0}|J_1|\lesssim \int_{\mathbb{R}^3} \frac{e^{-2\Im k|y|}}{|y|^2}|V\mu|dy\lesssim \|V\|_\infty \int_0^\infty e^{-2r\Im k }r^{-2}\int_{S_r(0)}|\mu|dr\lesssim
 \|V\|_\infty \int_0^\infty e^{-2r\Im k }\sqrt{m(r)}dr\,.
\]
Thus,  applying Cauchy-Schwarz, we get
\[
\limsup_{x\to 0}|J_1|\lesssim \frac{\sqrt{M(0)}\|V\|_\infty}{\sqrt{\Im k}}\,.
\]
For $J_2$, simple upper bounds for the integral (similar to \eqref{kert}) and the identity
\[
\lim_{|x|\to 0}|x|\int_{|y|<1}\frac{1}{|x-y|^2|y|}dy\lesssim \lim_{|x|\to 0} |x||\log |x||=0
\]
imply
$
\lim_{x\to 0}J_2=0\,.
$
That yields
\begin{equation}\label{zapa}
\limsup_{r\to 0}m_1(r)\lesssim \frac{M(0)\|V\|^2_\infty}{{\Im k}}\,.
\end{equation}
The bounds \eqref{key6} and \eqref{key88} give
\[
A(0)\le C(\|V\|)(1+M(0)+  |k|), \quad M(0)\le C(\|V\|)(1+\sqrt{A(0)M(0)}/|k|)
\]
when $|k|\to\infty$ and $k$ is in the sector. Then, $M(0)\le C(\|V\|), A(0)\le C(\|V\|)|k|$.

The next step will be to use these two bounds to control the integral in \eqref{again1} when  $|k|$ is large.
We consider the representation \eqref{poi1}. The estimate on $M(0)$ immediately gives
\[
\lim_{|k|\to\infty, \arg k\in (\delta,\pi-\delta)}\limsup_{|x|\to\infty}\|I_0\|_{L^2(\mathbb{S}^2)}=0
\]
after applying Cauchy-Schwarz inequality.
The same is true for $I_{1,1}$ in representation \eqref{rew}. In the representations for $I_{1,2}$,  $I_{1,3}$, and  $I_{1,4}$, we use estimates \eqref{erat} and the bounds for $A(0)$ and $M(0)$ to get
\[
\limsup_{r\to\infty} I_{1,2}=\limsup_{r\to\infty}\|B_r^{(4)}\left(Q\frac{\mu}{|y|}\chi_{|y|>1}\right)\|_{L^2(\mathbb{S}^2)}\le C(\|V\|)
C_4'(k)\|V\|M^{1/2}(0)\le \frac{C(\|V\|)}{\sqrt{|k|}}\,,
\]
\[
\limsup_{r\to\infty} I_{1,3}=\limsup_{r\to\infty}\|B_r^{(3)}\left(Q,\frac{\nabla\mu}{|y|}\chi_{|y|>1}\right)\|_{L^2(\mathbb{S}^2)}\le C(\|V\|)
C_3'(k)\|V\|A^{1/2}(0)\le \frac{C(\|V\|)}{\sqrt{|k|}}\,,
\]
\[
\limsup_{r\to\infty} I_{1,4}=\limsup_{r\to\infty}\|B_r^{(3)}\left(Q,           \frac{\mu y}{|y|^3}  \chi_{|y|>1}          \right)\|_{L^2(\mathbb{S}^2)}\le C(\|V\|)
C_3'(k)\|V\|M^{1/2}(0)\le \frac{C(\|V\|)}{{|k|}}\,,
\]
since $C_3'(k)\lesssim |k|^{-1}, C_4'(k)\lesssim C_1'(k)+C_2'(k)\lesssim |k|^{-1/2}$ as follows from \eqref{qeee11}, \eqref{qeee21}, and \eqref{qeee31} upon taking the square root. That finishes the proof.
\end{proof}

Consider again the truncated potential $V_{(\widehat R)}$ defined as in $\eqref{tru}$ and the corresponding function
\[
{A_\infty}_{(\widehat R)}(\sigma,y,k)=A_\infty(\sigma,y,k,V_{\widehat R})\,.
\]
We need the following stability lemma.
\begin{lemma} \label{lolo} Consider any $K$, a compact in $\mathbb{R}^3$ and $K_1$, a compact in $\Sigma$. If $k\in \Sigma$,
\begin{equation}\label{stab1}
\lim_{\widehat R\to\infty}\|{A_\infty}_{(\widehat R)}(\sigma,y,k)- A_\infty(\sigma,y,k)\|_{L^2(\mathbb{S}^2)}=0
\end{equation}
and convergence is uniform in $y\in K$ and $k\in K_1$. The function $A_\infty(\sigma,y,k)$ is continuous in $y$ and in $k$ in $L^2(\mathbb{S}^2)$ topology.
\end{lemma}
\begin{proof}
If $\mu_{(\widehat R)}(x,y,k)=A_{(\widehat R)}(x,y,k)$, we can write
\[
{A_\infty}_{(\widehat R)}(\sigma,y,k)=1-\int_{\mathbb{R}^3} \frac{e^{ik|\xi|(1-\langle \sigma,\widehat \xi\rangle)}}{4\pi|\xi|}V_{(\widehat R)}(\xi)\mu_{(\widehat R)}(\xi,y,k) d\xi=
\]
\begin{equation}\label{get}
1-\int_{|\xi|<\rho} \frac{e^{ik|\xi|(1-\langle \sigma,\widehat \xi\rangle)}}{4\pi|\xi|}V_{(\widehat R)}(\xi)\mu_{(\widehat R)}(\xi,y,k) d\xi-
\int_{|\xi|>\rho} \frac{e^{ik|\xi|(1-\langle \sigma,\widehat \xi\rangle)}}{4\pi |\xi|}V_{(\widehat R)}(\xi)\mu_{(\widehat R)}(\xi,y,k) d\xi
\end{equation}
with any $\rho>0$. For the third term,
\begin{equation}\label{lji}
\left\|\int_{|\xi|>\rho} \frac{e^{ik|\xi|(1-\langle \sigma,\widehat \xi\rangle)}}{|\xi|}V_{(\widehat R)}(\xi)\mu_{(\widehat R)}(\xi,y,k) d\xi \right\|_{L^2(\mathbb{S}^2)}\to 0
\end{equation}
as  $\rho\to\infty$ uniformly in $\widehat R$, $y\in K$, and $k\in K_1$. This follows from the estimates on the operators $B_\infty^{(j)}$ obtained in \eqref{erat} and bounds contained in lemma \ref{mre}.
The second term in \eqref{get} converges to
\[
\int_{|\xi|<\rho} \frac{e^{ik|\xi|(1-\langle \sigma,\widehat \xi\rangle)}}{|\xi|}V(\xi)\mu(\xi,y,k) d\xi
\]
in the uniform norm in $\sigma$ for every fixed $\rho$ when $\widehat R\to\infty$. This convergence is uniform in $y\in K$ and $k\in K_1$. It is now sufficient to notice that
\[
\int_{|\xi|<\rho} \frac{e^{ik|\xi|(1-\langle \sigma,\widehat \xi\rangle)}}{|\xi|}V(\xi)\mu(\xi,y,k) d\xi\to \int_{\mathbb{R}^3} \frac{e^{ik|\xi|(1-\langle \sigma,\widehat \xi\rangle)}}{|\xi|}V(\xi)\mu(\xi,y,k) d\xi, \quad \rho\to\infty\,.
\]
This  convergence is in $L^2(\mathbb{S}^2)$ and it is  uniform in $y\in K$ and $k\in K_1$. Indeed,
\[
\lim_{\rho\to\infty}\left\|\int_{|\xi|>\rho} \frac{e^{ik|\xi|(1-\langle \sigma,\widehat \xi\rangle)}}{|\xi|}V(\xi)\mu(\xi,y,k) d\xi\right\|_{L^2(\mathbb{S}^2)}=0
\]
similarly to \eqref{lji}. Thus, we first choose $\rho$ large enough to have integrals over $|\xi|>\rho$ small (uniformly in $\widehat R$) and then, with fixed $\rho$, send $\widehat R$ to infinity. 
\end{proof}

At that moment, it is important to make the following remark. When defining $A_\infty(\sigma,y,k)$ we first restricted $A(x,y,k)$ to the sphere $S_r(y)$ {\it centered at} $y$ and then took a limit as $r\to\infty$. We introduce another quantity now
\begin{equation}\label{sm-a}
a(x,y,k)\dd \frac{G(x,y,k^2)}{G^0(x,0,k^2)}, \quad k\in \Sigma\,.
\end{equation}
The results on convergence of $a(r\sigma,y,k)$ and stability lemma, similar to lemma  \ref{lolo}, can be proved in the similar manner (taking $f=\delta_y$ in \eqref{ur1}). This yields
\begin{itemize}
\item
\[
\sup_{r>|y|+1}\frac{1}{r^2}\int_{S_r(0)}|a(r\sigma,y,k)|^2d\sigma\le \frac{C(a,b,h,V,y)}{\Im^4 k}
\]
for all $k\in \Pi(a,b,h)$. 

\item
Moreover, 
\[
\lim_{r\to\infty}\|a(r\sigma,y,k)-a_\infty(\sigma,y,k)\|_{L^2(\mathbb{S}^2)}=0
\]
and $a_\infty(\sigma,y,k)$ is $L^2(\mathbb{S}^2)$-valued vector-function  analytic in $k\in \Sigma$.
\item
If $V_{\widehat R}$ is truncated potential and ${a_\infty}_{(\widehat R)}$ is an associated function, then
\begin{equation}\label{stabstab}
\lim_{\widehat R\to\infty}\|{a_\infty}_{(\widehat R)}(\sigma,y,k)- a_\infty(\sigma,y,k)\|_{L^2(\mathbb{S}^2)}=0\,.
\end{equation}
\end{itemize}

These results imply the following lemma.
\begin{lemma}\label{polk}For every $k\in \Sigma$,
\begin{equation}\label{polk2}
a_\infty(\sigma,y,k)=e^{-ik\langle \sigma,y\rangle}A_\infty(\sigma,y,k)\,.
\end{equation}
\end{lemma}
\begin{proof}
 In the case, when the potential is compactly supported, this is straightforward so 
 \[
{a_\infty}_{(\widehat R)}(\sigma,y,k)=e^{-ik\langle \sigma,y\rangle}{A_\infty}_{(\widehat R)}(\sigma,y,k)\,.
\]
Now, we only need to send $\widehat R\to\infty$ and use stability lemmas for $A$ and $a$, i.e., \eqref{stab1} and 
\eqref{stabstab}.
 \end{proof}

\begin{proof}{\it (of Theorem \ref{t4})}.
Recalling the definition of $h_f$ and using the previous lemma, we get
\[
h_f(\sigma,k)=\int_{\mathbb{R}^3} f(y) a_\infty(\sigma,y,k)dy, \quad {h_f}_{(\widehat R)}(\sigma,k)=\int_{\mathbb{R}^3} f(y){a_\infty}_{(\widehat R)}(\sigma,y,k)dy\,.
\]
The stability lemma implies that
\begin{equation}\label{lef}
\lim_{\widehat R\to\infty}\|  {h_{f}}_{(\widehat R)}-         h_f\|_{L^2(\mathbb{S}^2)}=0\,.
\end{equation}
In (\cite{yafaev}, p. 40-42, see also \cite{den1}, formula (4.2)), it was proved that
\begin{equation}\label{getq}
\sigma_f'(k^2, H_{(\widehat R)})=C k\|{h_{f}}_{(\widehat R)}(\sigma,k)\|^2_{L^2(\mathbb{S}^2)}, \quad k>0
\end{equation}
with an explicit absolute constant $C$ whose actual value is not important for us at that moment.
Consider the following function
$
p_{(\widehat R)}(k)=\|{h_{(\widehat R)}}_f(\sigma,k)\|_{L^2(\mathbb{S}^2)}\,.
$
It follows from the absorption principle for short-range potentials \cite{yafaev} that $p_{\widehat R}$ is continuous in $k\in \overline{\Pi(a,b,h)}$. It is also subharmonic and satisfies the following estimate, uniform in $\widehat R$:
\begin{equation}\label{getr}
p_{(\widehat R)}(k)\le \frac{C(a,b,h,f,V)} {\Im^2 k}
\end{equation}
as follows from the analyticity of $a_\infty$ and the main result of theorem \ref{t1}. Now we use the following argument (see, e.g., \cite{rowan}).
Consider an isosceles triangle $T_\alpha(a,b)$ with the base $[a,b]$ and the two angles equal $\alpha$ (Figure 1).

\begin{picture}(300,270)
\put(0,100){\line(1,0){450}}
\put(10,80){\line(0,1){160}}

\put(50,100){\line(1,0){300}}
\put(50,100){\line(4,1){150}}

 \put(50,100){\line(0,1){100}}
 \put(350,100){\line(0,1){100}}

 \put(50,200){\line(1,0){300}}

\put(350,100){\line(-4,1){150}}

\put(100,100){\line(0,1){12}}
\put(300,100){\line(0,1){12}}
\put(100,112){\line(1,0){200}}
\put(190,170){$\Pi(a,b,h)$}

\qbezier(75,100)(74,103)(72,105)
\put(77,101){$\alpha$}

\put(100,90){$a'$} \put(300,90){$b'$}
\put(15,80){$O$}
\put(15,230){$Y$}
\put(435,80){$X$}

\put(150,90){$ \Pi(a',b',h')$}

\put(210,118){$T_\alpha(a,b)$}

\put(200,105){\circle*{2}}\put(203,104){$k$}

\put(81,108){\circle*{2}}\put(82,112){$\xi$}

\put(50,90){$a$} \put(350,90){$b$} \put(180,50){Figure 1}
\end{picture}

Then, since $p_{(\widehat R)}$ is subharmonic,
\[
p_{(\widehat R)}(k)\le \int_{\partial T_\alpha} {P}_{T_\alpha(a,b)}(k,\xi)   p_{(\widehat R)}(\xi)d|\xi|
\]
at any interior point $k$.
The behavior of $P_{T_\alpha}$ at the corners $\xi=a,b$ is governed by the following estimates
\[
|P_{T_\alpha}(k,\xi)|\le C(a',b',h',\alpha) \min\Bigl( |\xi-a|^{(\pi-\alpha)/\alpha}, |\xi-b|^{(\pi-\alpha)/\alpha}\Bigr)
\]
uniformly over all $k\in \Pi(a',b',h')\subset T_\alpha$.  These bounds can be obtained by conformal mapping to the disc.  Provided that $\alpha$ is small enough,  \eqref{getq} and \eqref{getr} imply inequality
\begin{equation}\label{ker}
p_{(\widehat R)}(k)\le C(a',b',a,b,f,V)\left(1+\int_a^b P_{\mathbb{C}^+}(k,\xi)  \sqrt{\sigma'_f(\xi^2,H_{\widehat R})} d\xi\right)
\end{equation}
since
\[
P_{T_\alpha}(k,\xi)<C(a',b',h') P_{\mathbb{C}^+}(k,\xi)
\]
uniformly over $\xi\in [a',b']$ and $k\in \Pi(a',b',h')$.
Using the Cauchy-Schwarz inequality and changing variables, we get
\[
\int_a^b P_{\mathbb{C}^+}(k,\xi)  \sqrt{\sigma'_f(\xi^2,H_{\widehat R})} d\xi\le \left(   \int_{a^2}^{b^2} P_{\mathbb{C}^+}(k,\sqrt\eta)  d\sigma_f(\eta,H_{\widehat R})        \right)^{1/2}\,.
\]
This gives us
\[
p_{(\widehat R)}(k)\le C(a',b',a,b,f,V)\left(1+\left(   \int_{a^2}^{b^2} P_{\mathbb{C}^+}(k,\sqrt\eta)  d\sigma_f(\eta,H_{\widehat R})        \right)^{1/2}\right)\,.
\]
Fixing $k\in \Pi(a',b',h')$ and sending $\widehat R\to\infty$, we can apply lemma \ref{wkwk} to the right hand side and \eqref{lef} to the left hand side  to get the statement of the theorem when $k\in \Pi(a',b',h')$.  However, the function $p_{(\widehat R)}$ is uniformly bounded  in the domain $k\in [a',b']\times [h',1]$ so we can easily extend the result to $\Pi(a',b',1)$.

\end{proof}

\begin{proof}{\it (of theorem \ref{t5})} If $f$ is non-negative then $\langle h_f(\cdot,id),1\rangle^2_{L^2(\mathbb{S}^2)}>0$ for $d$ large enough (by  theorem \ref{t3}). Thus, $\langle h_f,1\rangle^2 $ is not identically zero in $\Sigma$. It is analytic in every $\Pi(a,b,h)$ and \eqref{t4t4} holds. We can map $\Pi(a',b',h')$ conformally to the unit disc by $w=\phi(k), w\in \mathbb{D}, k\in \Pi(a',b',h')$. Then
$
\langle h_f(\cdot,\phi^{-1}(w)),1\rangle^2_{L^2(\mathbb{S}^2)}
$
is analytic in $\mathbb{D}$ and its absolute value has  a harmonic majorant there due to the bound
\[
|\langle h_f(\cdot,k),1\rangle_{L^2(\mathbb{S}^2)}|^2\le p^2(k)
\]
and the estimate  \eqref{t4t4}. Therefore, $\langle h_f(\cdot,\phi^{-1}(w)),1\rangle_{L^2(\mathbb{S}^2)}\in H^2(\mathbb{D})$ and it is not identically zero. It has non-tangential boundary value at a.e. point on $\mathbb{T}$ and the following logarithmic integral converges
\[
\int_{\mathbb{T}} \log| \langle h_f(\cdot,\phi^{-1}(e^{i\theta}),1\rangle_{L^2(\mathbb{S}^2)}|d\theta>-\infty\,.
\]
Mapping it back to $\Pi(a',b',h')$ and taking any interior subinterval $(a_1,b_1)\subset (a',b')$ gives an estimate
\[
\int_{a_1^2}^{b_1^2}\log\sigma_f'(\eta,H)d\eta>C(a_1,b_1,V,f)\,.
\]
This argument is quite standard in the Nevanlinna theory of analytic functions \cite{garnet}.

\end{proof}
The existence of harmonic majorant for $h_f$ implies in the standard way the existence of the strong boundary values for $h_f$ when $\Im k\to 0$. We recall how that can be achieved. Fix $(a,b)\in (0,\infty)$ and $(a', b')\subsetneq (a,b)$.  Then $\widetilde h(w)=h_f(\phi^{-1}(w))$ belongs to vector-valued Hardy class $H^2(\mathbb{D})$ if $\phi$ maps $\Pi(a',b',1)$ conformally to the unit disc $\mathbb{D}$. This follows immediately from \eqref{t4t4} because
\[
C_1+C_2 \int_{a^2}^{b^2} P_{\mathbb{C}^+}(\phi^{-1}(z),\sqrt\eta)  d\sigma_f(\eta)
\]
is its harmonic majorant  in $\mathbb{D}$. It is known (\cite{rovnyak}, p. 80, Theorem A, p. 84) that functions in Hardy space $H^2(\mathbb{D})$ with values in Hilbert space ($L^2(\mathbb{S}^2)$ in our case) have strong boundary limit, i.e., there is $\widetilde h(e^{i\theta})\in L^2(\mathbb{S}^2)$ for a.e. $\theta\in [0,2\pi)$ so that
\[
\lim_{r\to 1}\|\widetilde h(re^{i\theta})-\widetilde h(e^{i\theta})\|_{L^2(\mathbb{S}^2)}=0
\]
and $ \|h(z)-\widetilde h(e^{i\theta})\|_{L^2(\mathbb{S}^2)}\to 0$ as $z\to e^{i\theta}$ for a.e. $\theta\in [0,2\pi)$, the limit in $z$  being  non-tangential.
Notice that $\widetilde h(\sigma,z)$ can be understood as
$
\widetilde h(\sigma,z)=\sum_{j} h_j(z)s_j(\sigma)\,,
$
where $\{s_j\}$ are spherical harmonics on $\mathbb{S}^2$,
$
\|\widetilde h(\sigma,z)\|^2_{L^2(\mathbb{S})}=\sum_j |\widetilde h_j(z)|^2
$
and $\widetilde h_j(z)$ are scalar functions from  Hardy space $H^2(\mathbb{D})$.
Transplanting these results back to $\Pi(a',b',1)$ we get existence of $h_f(\alpha)$ for a.e. $\alpha\in \mathbb{R}$. Moreover,
the non-tangential limit
\[
\lim_{k\to \alpha}\| h_f(k)- h_f(\alpha)\|_{L^2(\mathbb{S}^2)}=0
\]
holds for a.e. $\alpha\in \mathbb{R}$ because $a',b'$ are arbitrary.

The lemma \ref{opew} gives the symmetry
\[
h_f(-k)=\overline{h_f(k)}
\]
and \eqref{getq} yields
\begin{equation}\label{hase}
\sigma_f'(\alpha^2, H)\geq C|\alpha| \|h(\alpha)\|^2_{L^2(\mathbb{S}^2)}
\end{equation}
for a.e. $\alpha$. \bigskip

\subsection{Harmonic majorant for $A_\infty(\sigma,y,k)$}

 The first three theorems we proved had to do with the function $A_\infty(\sigma,y,k)$ and its properties as vector-valued function analytic in $k$. However, we obtained the harmonic majorant only for $h_f$  with $f$ being compactly supported  $L^2(\mathbb{R}^3)$ function. The main obstacle  to finding a majorant for $A$ is that it was defined through the solution to  equation
\[
-\Delta u+Vu=k^2u+\delta_y
\]
and we can not make sense of $\langle u,\delta_y\rangle$ because $u$ is not regular enough if the dimension is higher than one. However, we can overcome this problem by regularization. Take, e.g., $y=0$ and consider
\begin{equation}\label{reguu}
(H^2-k^4)^{-1}=(2k^2)^{-1}((H-k^2)^{-1}-(H+k^2)^{-1})\,.
\end{equation}
Notice that in the three-dimensional case $G(x,0,k^2)-G(x,0,-k^2)$ is continuous in $x$ for all $k$ that satisfy $ \arg k\in (0,\pi/4)$ provided that $V\in L^\infty(\mathbb{R}^3)$. This is so because
\[
\frac{e^{ik|x|}}{|x|}-\frac{e^{-k|x|}}{|x|}
\]
is continuous at $x=0$ and the terms
\[
\int_{\mathbb{R}^3} \frac{e^{ik|x-y|}}{|x-y|}V(y)G(y,0,k^2)dy, \quad        \int_{\mathbb{R}^3} \frac{e^{-k|x-y|}}{|x-y|}V(y)G(y,0,-k^2)dy
\]
are both continuous at $x=0$ by lemma \ref{pop}.
We now define
\begin{equation}\label{newm}
m(k)\dd \langle (H^2-k^4)^{-1}\delta_0,\delta_0\rangle\,.
\end{equation}
Approximating $\delta_0$ with any $\delta_0$--generating sequence $\{f_n\}, f_n\in L^2(\mathbb{R}^3)$, 
we obtain
\[
\Im m(k)=\lim_{n\to\infty} \langle (H^2-k^4)^{-1}f_n,f_n\rangle>0
\]
and $m$ is analytic in the sector $0<\arg k<\pi/4$. By the Nevanlinna representation (see, e.g., \cite{rovnyak}, p. 141, Theorem B), we have for every $k\in \mathbb{C}^+$:
\begin{equation}\label{refew}
m(k^{1/4})=c_1+c_2 k+\frac{1}{\pi}\int_{\mathbb{R}} \left(\frac{1}{t-k}-\frac{t}{1+t^2}\right) d\mu(t), \quad c_1\in \mathbb{R},\quad c_2\ge 0
\end{equation}
where $\mu$ is a positive measure on $\mathbb{R}$ that satisfies
\[
\int_{\mathbb{R}} \frac{d\mu}{1+t^2}<\infty\,.
\]
We can prove the following analog of \eqref{getq}.
\begin{lemma} \label{kulik} Assume that $V\in C_c^\infty(\mathbb{R}^3)$. Then,
\begin{equation}\label{vorob1}
32\pi^2k\mu'(k^4)=\|{A_\infty}(\sigma,0,k)\|^2_{L^2(\mathbb{S}^2)}, \quad k>0\,,
\end{equation}
where $\mu$ is the measure from \eqref{refew}.
\end{lemma}
\begin{proof}
Start by taking $k$ in the sector $\arg k\in (0,\pi/4)$. Let $u=(H^2-k^4)^{-1}f$, i.e., $u$ solves
\[
(-\Delta+V-k^2)(-\Delta+V+k^2)u=(H^2-k^4)u=f\,,
\]
where $f$ is any test function, i.e., $f\in C_c^\infty(\mathbb{R}^3)$. Multiply this equation by $\overline u$ and integrate over $B_R(0)$ with  $R$ so large that supp$(f)\subset B_R(0)$.
\begin{equation}\label{los}
\int_{|x|<R}\Bigl((-\Delta+V-k^2)(-\Delta+V+k^2)u\Bigr)\overline udx=\int_{\mathbb{R}^3} f\overline udx\,.
\end{equation}
Now we send $k\to \kappa\in \mathbb{R}^+$ where  $-\kappa^2$ in not an eigenvalue of $H$ and take imaginary part of both sides. We can write
\begin{equation}\label{sobka}
u=(2k^2)^{-1}((H-k^2)^{-1}f-(H+k^2)^{-1}f)\,.
\end{equation}
The term $(H+k^2)^{-1}f$ decays exponentially in space variable. For the other term, the absorption principle and integration by parts  give
\[
\Im \int_{|x|<R}\Bigl((-\Delta+V-\kappa^2)(-\Delta+V+\kappa^2)u\Bigr)\overline udx=
\]
\[
\Im \int_{|x|<R}\Bigl((-\Delta+V)^2-\kappa^4)u\Bigr)\overline udx=\Im \int_{|x|<R}\Bigl((-\Delta+V)^2)u\Bigr)\overline udx=
\]
\[
\Im \int_{|x|<R}\Bigl((-\Delta+V)u\Bigr)\Bigl((-\Delta+V)\overline u\Bigr)dx+\Im \int_{|x|=R}\Bigl( (\Delta u)_r\overline u-  (\Delta u)\overline u_r\Bigr) d\sigma_x
\]
The first term is zero, so
\begin{equation}\label{kfc}
\Im \int_{|x|=R}\Bigl( (\Delta u)_r\overline u-  (\Delta u)\overline u_r\Bigr) d\sigma_x=
\pi \mu'_f(\kappa^4),\, \,{\rm for \,\, a.e.}\, \kappa,
\end{equation}
where $\mu_f$ is the spectral measure of $f$ relative to $H^2$.
The asymptotics of $u$ at infinity is given by
\[
u(x,\kappa)=(2\kappa^2)^{-1}\frac{e^{i\kappa |x|}}{4\pi |x|}\left(\int_{\mathbb{R}^3} a_\infty(\sigma,y,\kappa)f(y)dy+o(1)\right), \quad |x|\to\infty, \widehat x\to \sigma
\]
as follows from the formula \eqref{sobka}. Moreover, this asymptotics can be differentiated in $x$.
Notice that both  $A_\infty(\sigma,y,k)$ and $a_\infty(\sigma,y,k)$ are continuous in $k\in \overline{\Pi(a,b,h)}$. This continuity follows from the limiting absorption principle.  
Now, substitute this asymptotics into the previous formula \eqref{kfc} and send $R\to\infty$ to get identity
\[
\frac{1}{32\pi^2}\int_{\mathbb{S}^2} \left|   \int_{\mathbb{R}^3} a_\infty(\sigma,y,\kappa)f(y)dy                \right|^2d\sigma=\pi\kappa \mu_f'(\kappa^4)\,.
\]
Taking $f=f_n$ where $\{f_n\}$ is $\delta_0$-generating sequence yields
\[
\frac{1}{32\pi^2}\int_{\mathbb{S}^2}    |A_\infty(\sigma,0,\kappa)|^2d\sigma=\kappa \mu'(\kappa^4)
\]
with some absolute constant $C$. 
\end{proof}
{\bf Example.} In the free case, one has $m(k)=(1+i)/(8\pi k), A=1$.\smallskip

Having established the formula \eqref{vorob1}, we immediately get the analog of theorem \ref{t4}.

\begin{theorem} \label{t4oo}   Let $V$ satisfy \eqref{main-assump} and $[a,b]\subset (0,\infty)$, then
\begin{equation}\label{t4t4-7}
\|A_\infty(\sigma,y,k)\|^2_{L^2(\mathbb{S}^2)}\le C(a',b',a,b,V,y)\left(1+\int_{a^4}^{b^4} P_{\mathbb{C}^+}(k,\eta^{1/4})  d\mu_y(\eta) \right)
\end{equation}
for all intervals $(a',b')\subsetneq (a,b)$ and all $k\in \Pi(a',b',1)$. The positive measure $\mu_y$ is related to $\delta_y$ by  \eqref{newm} and \eqref{refew} .
\end{theorem}
\begin{proof}The proof of this result repeats the proof of theorem \ref{t4}.
\end{proof}

Now that we found the harmonic majorant for $A_\infty$, we immediately get
\begin{itemize}
\item improved estimate on the possible growth at the boundary:
\begin{equation}\label{kochka1}
\|A_\infty(\sigma,y,k)\|^2_{L^2(\mathbb{S}^2)}\le \frac{C(a,b,V,y)}{\Im k}, \quad k\in \Pi(a,b,1)\,,
\end{equation}

\item existence of the strong non-tangential limit $A_\infty(\sigma,y,\kappa)\in L^2(\mathbb{S}^2)$ for a.e.$ \kappa\in \mathbb{R}$, i.e., 
\begin{equation}\label{kochka2}
\lim_{k\to\kappa}\|A_\infty(\sigma,y,k)-A_\infty(\sigma,y,\kappa)\|_{L^2(\mathbb{S}^2)}=0
\end{equation}
for a.e. $\kappa$ and the limit is non-tangential. Secondly,
\[
\lim_{\epsilon\to 0}\int_a^b\|A_\infty(\sigma,y,\kappa+i\epsilon)
)-A_\infty(\sigma,y,\kappa)\|^2_{L^2(\mathbb{S}^2)}d\kappa=0
\]
for every $[a,b]$ not containing zero.
\end{itemize}

Analogous results hold for $a_\infty(\sigma,y,k)$.

{\bf Remark.} In \eqref{defh}, we defined $h_f(\sigma,k)$ as
\[
h_f(\sigma,k)=\int_{\mathbb{R}^3} A_\infty(\sigma,y,k)e^{-ik\langle \sigma,y\rangle}f(y)dy
\] and, by lemma \ref{polk},
\[
h_f(\sigma,k)
=\int_{\mathbb{R}^3} a_\infty(\sigma,y,k)f(y)dy, \quad k\in \Sigma\,.
\]
Then, we proved that both $h_f$ and $a_\infty$ have non-tangential boundary value in $k$.  However, we didn't prove the continuity of $A_\infty(\sigma,y,\kappa)$ or $a_\infty(\sigma,y,\kappa)$ in $y$. Instead, we can show that 
\begin{equation}\label{gui}
h_f(\sigma,\kappa)=\int_{\mathbb{R}^3} a_\infty(\sigma,y,\kappa)f(y)dy
\end{equation}
in the sense of $L^2(\mathbb{S}^2,[a,b])$ functions. Indeed, the maximal function 
\[
M_{(a)}(y,\kappa)\dd \sup_{\epsilon\in (0,1)} \|a_\infty (\sigma,y,\kappa+i\epsilon)\|_{L^2(\mathbb{S}^2)}
\]
satisfies
\[
M_{(a)}(y,\kappa)\in 
 L^2(a,b)
\]
as follows from the properties of the non-tangential maximal function of the scalar $H^2(\mathbb{D})$. Moreover,
\[
\sup_{y\in K} \int_a^b|M_{(a)}(y,\kappa)|^2d\kappa<\infty
\]
where $K$ is any compact. Therefore, by the Minkowski inequality and Dominated Convergence Theorem, we have
\[
\lim_{n\to\infty}\left\|\int_{\mathbb{R}^3} f(y)\Bigl(  a_\infty(\sigma,y,\kappa+in^{-1})-a_\infty(\sigma,y,\kappa)   \Bigr)dy\right\|_{L^2(\mathbb{S}^2),[a,b])}=0
\]
and \eqref{gui} follows.\vspace{0.5cm}

{\bf Remark.} Studying asymptotical behavior of $A(r\sigma,y,k)$ for large $r$ and $k\in \Sigma$ is an interesting problem. For example, the following question is natural: can one improve $\lim_{r\to\infty}\|A(y+r\sigma,y,k)-A_\infty(\sigma,y,k)\|_{L^2(\mathbb{S}^2)}=0$ to  $\lim_{r\to\infty}\|A(y+r\sigma,y,k)-A_\infty(\sigma,y,k)\|_{L^p(\mathbb{S}^2)}=0$ with $p>2$? In \cite{den1}, this was answered affirmatively for  $p=\infty$ in the case when 
\[
|Q|\lesssim (1+|x|)^{-\frac 12-\delta}, \quad |V|\lesssim (1+|x|)^{-\frac 12-\delta}
\]
with some $\delta>0$.

We conclude the first part with a list of  questions:

\begin{itemize}
\item[(1).]{ In \eqref{hase}, does equality hold for a.e. $k$?}
\item[(2).] {Consider the boundary value $A_\infty(\sigma,y,\kappa), \kappa\in \mathbb{R}$. Is it continuous in $y$ as a function in $L^2(\mathbb{S}^2)$? Is the zero-measure set of $\kappa$ on which $M_{(a)}(y,\kappa)=\infty$ independent of $y$? We have weak continuity of $A_\infty(\sigma,y,\kappa)$ in $y$, considered as a function in $\sigma$ and $\kappa$. This follows from continuity of $A$ in $y$ for fixed $k\in \mathbb{C}^+$.}
\item[(3).] {Does theorem \ref{t5} hold for any compactly supported $f\in L^2(\mathbb{R}^3)$?}
\end{itemize}\vspace{1cm}

\section{Part 2. Elliptic operators in the divergence form: wave equation and wave operators}

\subsection{Formulation of main result}
In this part, we will be concerned with the following operator
\[
Df=-{\rm div}\,((1+V)\nabla f), \quad x\in \mathbb{R}^3
\]
We will assume that $V$ satisfies conditions that are a little stronger than those assumed in the first part:
\begin{equation}\label{2-assu}
\|V\|_\infty<1,\, V={\rm div} \, Q, \, Q\in C^2(\mathbb{R}), \, \|V\|\dd \max_{j=0,1,2}\|D^jQ\|_{\ell^2(\mathbb{Z}^+),L^\infty}<\infty\,.
\end{equation}
We will also need the following notation
\[
 \|V\|_{[r,\infty)}\dd \max_{j=0,1,2}\|D^jQ\|_{\ell^2([r],\infty),L^\infty}\,.
\]
Conditions $\|V\|_\infty<1,\, \nabla V\in L^\infty(\mathbb{R}^3)$ allow us to define $D$ as a positive operator by Kato-Rellich Theorem \cite{cycon}. \smallskip

The plan of this part is as follows. We will first establish the asymptotics of the Green's function for $D$ by mimicking the arguments in the first part. This will require only slight modification. Then, we will consider the wave equation
\[
u_{tt}+Du=0, \quad u(x,0)=f_+, \quad u_t(x,0)=i\sqrt{D} f_+\,,
\]
where $f_+$ is assumed to belong to  the domain of $\sqrt{D}$.
Recall, that our main result is existence of wave operators.
\begin{theorem}\label{main-th}
If $V$ satisfies \eqref{2-assu}, then the following wave operators exist
\[
W^{\pm}(\sqrt{D},\sqrt{H_0})\dd {s-\lim}_{t\to\pm \infty}e^{it\sqrt{D}}e^{-it\sqrt{H_0}}
\]
and the limit is understood in the strong sense.
\end{theorem}

{\bf Remark. } This implies in the standard way that ranges of $W^{\pm}$ belong to the a.c. subspaces of $D$ \cite{rs3}. Moreover, when restricted to those ranges, $D$ is unitarily equivalent to $H_0$ and thus has a.c. spectrum of infinite multiplicity. We want to mention here that the infinite multiplicity of the a.c. spectrum for multidimensional  Schr\"odinger operator with slowly decaying potential was established in \cite{lns7}, p. 614, remark 4.  The completeness of wave operators, i.e., the statement that ranges of $W^{\pm}$ are equal to $\cal{D}_{ac}$, the a.c. subspace of $D$, is an interesting question which we do not address in this paper.\smallskip

The theorem \ref{il5} about long time behavior of solution to wave equation from the Introduction will be proved at the very end of this part of the paper after we obtain stationary representation for wave operators.

{\bf Remark.} The general problem $u_{tt}+Du=0$ with initial data $u(x,0)=g_1, u_t(x,0)=g_2$ where $g_2\in$ range$(\sqrt{D})$ can be reduced to studying $e^{it\sqrt{D}}$  since the function $u=e^{it\sqrt{D}}f_1+e^{-it\sqrt{D}}f_2$ solves the wave equation and the initial conditions are
\[
u(x,0)=f_1+f_2, \quad u_t(x,0)=i\sqrt{D}(f_1-f_2)
\]
so, given $g_{1(2)}$ the corresponding $f_{1(2)}$ can be found and the problem solved.

\bigskip

{\bf Definition. }The following subset of Schwarz class $\cal{S}(\mathbb{R}^3)$ will be used later in the text: $\cal{N}$ is the set of all functions $f\in \cal{S}(\mathbb{R}^3)$ such that $\widehat f\in C^\infty_c(\mathbb{R}^3)$  and dist$(0,{\rm supp}(\widehat f))>0$.\smallskip

Clearly, $\cal{N}$ is dense in $L^2(\mathbb{R}^3)$. This subspace will be convenient for us because it has the following property: if $f\in \cal{N}$, then $\nabla f, |\Delta|^\alpha f\in \cal{N}$ for every $\alpha\in \mathbb{R}$.

\bigskip
\subsection{Basic properties of wave equation}

For the free evolution, we can write the solution exactly. In fact (see, e.g., \cite{ss}, theorem 3.6, p. 191 or p. 211), the solution to
\[
u_{tt}=\Delta u, \quad u(x,0)=f_1, \quad u_{t}(x,0)=f_2
\]
is given by Kirchhoff's formula
\[
u(x,t)=\partial_t(tM_t(f_1))+tM_t(f_2)=\Bigl\langle  \Bigl(tf_2(y)+f_1(y)+\nabla f_1(y)(y-x) \Bigr) \Bigr\rangle_{S_t(x)}\,.
\]
If we assume that, e.g., $f_{1(2)}\in \cal{S}(\mathbb{R}^3)$, then
\[
tM_t(f_2)(x)= \frac{1}{4\pi t }\left( \int_{Pl_{|x|-t,\widehat x}} f_2(y)dy +o(1)\right), \quad t\to\infty\,,
\]
where the integral represents the Radon transform $\cal{R}f_2(|x|-t,\widehat x)$ of $f_2$ (see \cite{ss}, p.201) and $Pl_{\alpha,\gamma}=\{\xi\in \mathbb{R}^3: \langle \xi,\gamma\rangle=\alpha, \, \alpha\in \mathbb{R},\gamma\in \mathbb{S}^2\}$ denotes the plane over which the function is integrated. Similarly,
\[
\partial_t(tM_t(f_1)(x))= -\frac{1}{4\pi t }\left( \int_{Pl_{|x|-t,\widehat x}} \widehat x\cdot \nabla f_1(y)dy +o(1)\right), \quad t\to\infty\,.
\]
In particular, if $f_2=i|\Delta|^{1/2}f_1$ and $f_1=f$ where $f\in \cal{N}$ then $f_{1(2)}\in \cal{S}(\mathbb{R}^3)$ and
\begin{equation}\label{poiu}
\Bigl(e^{it\sqrt{H_0}}f\Bigr)(x)= \frac{1}{4\pi t}\left( \int_{Pl_{|x|-t,\widehat x}}    \left( -\widehat x\cdot \nabla f(y)+ i\Bigl(|\Delta|^{1/2}f\Bigr) (y)\right)dy +o(1)\right), \quad t\to\infty
\end{equation}
uniformly over $x$.  Another way to get this asymptotics is to write up the dispersion relation for wave equation and apply the method of stationary phase. 
Since $e^{it\sqrt{H_0}}$ preserves $L^2(\mathbb{R}^3)$ norm, we get
\begin{equation}\label{local1}
\lim_{R\to\infty}\limsup_{t\to\infty}\int_{||x|-t|>R}\Bigl|\Bigl(e^{it\sqrt{H_0}}f\Bigr)(x)\Bigr|^2dx=0
\end{equation}
for every $f\in L^2(\mathbb{R}^3)$.

{\bf Remark.} Consider the general problem
\[
u_{tt}+D u=F, \quad u(x,0)=f_1, \quad u_t(x,0)=f_2\,.
\]
We will need the Duhamel formula (e.g., \cite{tao}, p.67)
\begin{equation}\label{duhamel}
u=\cos(t\sqrt{D})f_1+\frac{\sin(t\sqrt{D})}{\sqrt{D}}f_2+\int_0^t \frac{\sin((t-\tau)\sqrt{D})}{\sqrt{D}}F(\tau)d\tau
\end{equation}
later on.


\subsection{Auxiliary results}

The following key lemma allows to control the long-time behavior of $e^{it\sqrt{A}}$ if the large $x$ asymptotics of the Green's function $G(x,0,k^2)$ is known for $k$ away from the spectrum.

Let $[a,b]\subset (0,\infty)$ be any positive interval. Consider the following contour
\[
\Gamma_{a,b}\dd \gamma_a\cup \gamma_b\cup \gamma_+\cup\gamma_-
\]
where $\gamma_a\dd \{k: \Re k=a, |\Im k|<1\}$,    $\gamma_b\dd \{k: \Re k=b, |\Im k|<1\}$,  $\gamma_{\pm}\dd \{k: \Im k=\pm 1, a<\Re k<b\}$.

\begin{lemma}
If $A$ is self-adjoint non-negative operator defined on the Hilbert space and  $t>0$, then integration over $\Gamma_{a,b}$ counterclockwise gives us
\begin{equation}\label{formula1}
\frac{1}{2\pi i}\int_{\Gamma_{a,b}} e^{-itk} (A-k^2)^{-1}p_n(k)dk= e^{-it\sqrt{A}}q_n(\sqrt{A})\,,
\end{equation}
where 
\begin{equation}\label{kukuru}
p_n(k)=(k-a)^{n}(k-b)^{n}, \, q_n(k)=-\chi_{(a,b)}(k)p_n(k)/(2k)
\end{equation}
and $n\in \mathbb{N}$.

\end{lemma}
\begin{proof}
The proof is immediate if the residue formula is applied in the Spectral Theorem representation. From the same Spectral Theorem, we have
\[
\|(A-k^2)^{-1}\|<C(a,b)(\Im k)^{-1}\,,
\]
if $k\in \gamma_{a(b)}$ and, since $n\ge 1$, the integral converges absolutely in the operator norm, i.e.,
\[
\int_{\Gamma_{a,b}} \Bigl\| e^{-itk} (A-k^2)^{-1}p_n(k)\Bigr\|d|k|<\infty\,.
\]
\end{proof}
\begin{lemma}\label{dobro2}
If $\|V\|_\infty<1, \|\nabla V\|_\infty<\infty$,  then
\[
\|f\|_{\cal{H}^2(\mathbb{R}^3)}\le C(\|V\|_\infty,\|\nabla V\|_\infty)  \Bigl(\|Df\|_2+\|f\|_2\Bigr)\,.
\]
\begin{proof}
Indeed,
\[
\|Df\|_2\ge \|(1+V)\Delta f\|_2-\|\nabla V\|_\infty \|\nabla f\|_2\ge (1-\|V\|_\infty)\|\Delta f\|_2-\|\nabla V\|_\infty\sqrt{\|\Delta f\|_2\|f\|_2}\ge
\]
\[
(1-\|V\|_\infty)\|\Delta f\|_2-2\|\nabla V\|_\infty(\epsilon \|\Delta f\|_2+\epsilon^{-1}\|f\|_2)\,,
\]
where $\epsilon$ is an arbitrary positive number. Taking $\epsilon$ sufficiently small, we finish the proof.
\end{proof}

\end{lemma}

\begin{lemma}   \label{lek} Assume $V\in C^1(\mathbb{R}^3)$ and $\lim_{|x|\to\infty}V=0, \lim_{|x|\to\infty}\nabla V=0$.
Suppose $\alpha(E)\in C(\mathbb{R})$ and $\lim_{|E|\to\infty}\alpha(E)=0$. Then, if the functional sequence $\{\psi_n\}$ satisfies the following conditions
\begin{itemize}
\item   $\sup_{n} \|\psi_n\|_2<\infty$,
\item $\lim_{R\to\infty}\limsup_{n\to\infty}\|\psi_n\|_{L^2(B_R(0))}=0$\quad  (``escape to infinity in $L^2(\mathbb{R}^3)$ norm''),

\end{itemize}
then
\begin{equation}\label{vtornik}
\lim_{n\to\infty}\|(\alpha(D)-\alpha(H_0))\psi_n\|_2\to 0\,.
\end{equation}
\end{lemma}

\begin{proof} We notice that
\begin{equation}\label{eryt}
R_z-R_z^0=-R_zV_1R_z^0, \quad\quad V_1=VH_0-\nabla V\cdot \nabla,\quad  z\notin \mathbb{R}.
\end{equation}
Therefore, \eqref{vtornik} holds for $\alpha(E)=(E-z)^{-1}, \, z\notin \mathbb{R}$ because 
\begin{itemize}
\item   $\sup_{n} \|R_z^0  \psi_n\|_2<\infty$,
\item $\lim_{R\to\infty}\limsup_{n\to\infty}\|R^0_z\psi_n\|_{L^2(B_R(0))}=0$\quad  

\end{itemize}
and the same holds for $\nabla R_z^0\psi_n$. The linear span of the following functions 
\[
\left\{\sum_{j=1}^N \frac{c_j}{z_j-E}, \quad z_j\notin \mathbb{R}\right\}
\]
can approximate any  given $\alpha(E)$ in the supremum norm over $\mathbb{R}$ (by convolving with Poisson kernel and discretizing the Riemannian sum). Therefore, by the Spectral Theorem, we have statement of the lemma.
\end{proof}

{\bf Remark.} We will often apply this lemma in the case when $e^{it\sqrt{H_0}}f, \, f\in L^2(\mathbb{R}^3)$ is taken as $\psi_n$ (the generalization from $n\in \mathbb{N}$ to $t\in \mathbb{R}$ is obvious).  The both properties of the sequence will be satisfied due to properties of free evolutions, i.e., lemma \ref{local1}.

In the next three lemmas, we will study the basic properties of the Green's function $G(x,y,z),\,z\notin [0,\infty)$. Its existence can be deduced similarly to lemma \ref{dobro1} from, e.g., $\|R_z\|_{L^2(\mathbb{R}^3),\cal{H}^2(\mathbb{R}^3)}<\infty, \cal{H}^2(\mathbb{R}^3)\subseteq L^\infty(\mathbb{R}^3)$ (see Corollary~2.14, \cite{cycon}). 

\begin{lemma}\label{look}
Assume $\|V\|_\infty<1, \|\nabla V\|_\infty<\infty$ and denote $\lambda_0\dd \|V\|_\infty+\|\nabla V\|_\infty$. If $k\in \Pi(a,b,1)$, then
\begin{equation}\label{lukas1}
\sup_{y\in \mathbb{R}^3}\|G(x,y,k^2)\|_{L^2_x(\mathbb{R}^3)}\le C(a,b,\lambda_0) |\Im k|^{-1}, \quad \sup_{y\in \mathbb{R}^3}\| G(x,y,k^2)\|_{\cal{H}^2_x(B_1^c(y))}\le C(a,b,\lambda_0) |\Im k|^{-1}\,.
\end{equation}
In general, for $k\in \mathbb{C}^+$, we have
\begin{equation}\label{dobro3}
\sup_{y\in \mathbb{R}^3}\|G(x,y,k^2)\|_{L^2_x(\mathbb{R}^3)}<\infty,\quad 
\sup_{y\in \mathbb{R}^3}\|G(x,y,k^2)\|_{\cal{H}^2_x(B_1^c(y))}<\infty
\end{equation}
 and 
\begin{equation}\label{lukas}
\sup_{y\in \mathbb{R}^3}\|G(y,x,k^2)\|_{L^2_x(\mathbb{R}^3)}<\infty, \quad \sup_{y\in \mathbb{R}^3}\|G(y,x,k^2)\|_{\cal{H}^2_x(B_1^c(y))}<\infty\,.
\end{equation}
\end{lemma}
\begin{proof} Let $f: \|f\|_2\leq 1$ and $u\dd  (D-k^2)^{-1}f$. Then, by the Spectral Theorem,
$
\|u\|_2\le C(a,b) |\Im k|^{-1}\,.
$
From lemma \ref{dobro2}, we get
$\|u\|_{\cal{H}^2(\mathbb{R}^3)}\le C(a,b,\lambda_0) |\Im k|^{-1}$. By duality, 
\begin{eqnarray}
\sup_{y\in \mathbb{R}^3}\|G(x,y,k^2)\|_{L^2_x(\mathbb{R}^3)}=\sup_{\|g\|_2\le 1}|\langle (D-k^2)^{-1}\delta_y,g\rangle|=\sup_{\|g\|_2\le 1}|\langle (\delta_y,(D-\bar k^2)^{-1}g\rangle|\le
\label{zhak}
\\ \|\delta_y\|_{\cal{H}^{-2}(\mathbb{R}^3)}\|(D-k^2)^{-1}g\|_{\cal{H}^2(\mathbb{R}^3)}\le C(a,b,\lambda_0)(\Im k)^{-1}\nonumber
\end{eqnarray}
as claimed. For the second inequality,  denote $v(x)\dd G(x,y,k^2)$ and consider equation
\[
-\Bigl((1+V)\Delta v+\nabla V\nabla v\Bigr)=k^2v, \quad x\neq y\,,
\]
from which we get the statement (e.g., by the Interior Regularity Theorem for elliptic equations). The proof of \eqref{dobro3} is identical and to get \eqref{lukas} we only need to notice that 
\begin{equation}\label{simol}
G(x,y,\bar{z})=\overline{G(y,x,z)}
\end{equation}
since $((D-z)^{-1})^*=(D-\bar z)^{-1}$.

\end{proof}
We now state the analog of lemma \ref{lm2} from the first part. 
\begin{lemma}\label{kere}
Assume $\|V\|_\infty<1, \|\nabla V\|_\infty<\infty$ and denote $\lambda_0\dd \|V\|_\infty+\|\nabla V\|_\infty$. If $k\in \Pi(a,b,1)$, then
\begin{eqnarray*}
\sup_{y\in \mathbb{R}^3, r>|y|+1}r^{-2}\int_{S_r(0)}|G(x,y,k^2)|^2d\sigma_x\le \frac{C(a,b,\lambda_0)}{r^2\Im^2 k}\,,\\ \sup_{y\in \mathbb{R}^3,r>|y|+1}  r^{-2}\int_{S_r(0)}|\partial_rG(x,y,k^2)|^2d\sigma_x\le
\frac{C(a,b,\lambda_0)}{r^2\Im^2 k}\,.
\end{eqnarray*}
\end{lemma}
\begin{proof} This follows from the previous lemma after estimating the traces on the hupersurfaces in a way which was used to prove lemma \ref{poliu} (we now have an estimate somewhat stronger than lemma \ref{lm2} since we didn't use \eqref{leret}).

\end{proof}

If we define $Q_{(\rho)}$ by \eqref{tru}, then the analog of lemma \ref{dvaff} holds true.

\begin{lemma} If $V$ satisfies \eqref{2-assu}, then
\begin{equation}\label{dvaff-d}
\lim_{\rho\to\infty}\|G_{(\rho)}(x,y,z)-G(x,y,z)\|_{L^2_x(\mathbb{R}^3)}=0
\end{equation}
and 
\begin{equation}\label{dvaff-dva}
\lim_{\rho\to\infty}\|G_{(\rho)}(x,y,z)-G(x,y,z)\|_{\cal{H}^2(\{x: r_1<|x-y|<r_2\})}=0
\end{equation}
for all $r_{1(2)}: 0<r_1<r_2$ and $z\notin [0,\infty)$.
\end{lemma}
\begin{proof}
To prove \eqref{dvaff-d}, we can argue as in \eqref{zhak}:
\[
\|G(x,y,k^2)-G_{(\rho)}(x,y,k^2)\|_{L^2_x(\mathbb{R}^3)}=\sup_{\|g\|_2\le 1}|\langle (R_{k^2}-{R_{(\rho)}}_{k^2})\delta_y,g\rangle|=\sup_{\|g\|_2\le 1}|\langle (\delta_y,  (R_{\bar k^2}-{R_{(\rho)}}_{\bar k^2)})g\rangle|\,.
\]
For the difference of resolvents, we can write
\[
R_z-{R_{(\rho)}}_z={R_{(\rho)}}_z(-V^{(\rho)}H_0-\nabla V^{(\rho)}\cdot \nabla)R_z
\]
with $V^{(\rho)}=V-V_{(\rho)}$. Thus,
\begin{eqnarray*}
\|(R_z-{R_{(\rho)}}_z)g\|_{\cal{H}^2(\mathbb{R}^3)}\le \|{R_{(\rho)}}_z\|_{L^2(\mathbb{R}^3),\cal{H}^2(\mathbb{R}^3)}(\|V^{(\rho)}\|_\infty \|H_0\|_{\cal{H}^2(\mathbb{R}^3),L^2(\mathbb{R}^3)}+\\\|\nabla V^{(\rho)}\|_\infty\|\nabla\|_{\cal{H}^2(\mathbb{R}^3),L^2(\mathbb{R}^3)}) \|R_z\|_{L^2(\mathbb{R}^3),\cal{H}^2(\mathbb{R}^3)}\|g\|_{L^2(\mathbb{R}^3)}
\end{eqnarray*}
and this expression converges to $0$ as $\rho\to\infty$ because $V$ satisfies \eqref{2-assu}. This gives \eqref{dvaff-d}. \eqref{dvaff-dva}  follows after comparing equations satisfied by Green's kernels and using the Interior Regularity Theorem for elliptic equations.
\end{proof}

{\bf Remark.} Property \eqref{dvaff-dva} allows us to conclude that
\begin{equation}\label{apa1}
\lim_{\rho\to\infty}\int_{S_r(0)}|G(x,y,k^2)-G_{(\rho)}(x,y,k^2)|^2d\sigma_x=0, \,\lim_{\rho\to\infty}\int_{S_r(0)}|\partial_rG(x,y,k^2)-\partial_rG_{(\rho)}(x,y,k^2)|^2d\sigma_x=0
\end{equation}
for all $r\in (|y|,\infty)$. This follows from the theorem about existence of the traces on the hypersurfaces.\smallskip

We conclude with discussion of one  technical issue. In the next section, we will  need to know the local  regularity of $\nabla_x G(x,y,k^2)$ in $x$ around $y$. 
To study this problem, notice that we can write equation
$
Du=k^2u+f, \,f\in L^2(\mathbb{R}^3)
$
in the form
\[
(-\Delta-k^2)u=\frac{f}{1+V}+\frac{\nabla V\nabla u}{1+V}-k^2\frac{Vu}{1+V}\,.
\]
Assume that $V,\nabla V, (1+V)^{-1}\in L^\infty(\mathbb{R}^3)$.
Then, 
\[
R_z=R_z^0\Bigl(\frac{1}{1+V}+\frac{\nabla V}{1+V}\nabla R_z-z\frac{V}{1+V}R_z  \Bigr), \quad z=k^2   
\]
and 
\begin{equation}\label{popul}
G(x,y,z)=\frac{G^0(x,y,z)}{1+V(y)}-z\int_{\mathbb{R}^3} G^0(x,\xi,z)\frac{V(\xi)}{1+V(\xi)}G(\xi,y,z)d\xi+
\end{equation}
\[
\int_{\mathbb{R}^3} G^0(x,\xi,z)\frac{\nabla V(\xi)}{1+V(\xi)}\nabla G(\xi,y,z)d\xi\,,
\]
where this identity is understood in the weak sense, i.e., by integrating with test function $\phi(y)\in C_c^\infty(\mathbb{R}^3)$.
Taking gradient in $x$ of both sides, we get
\begin{equation}\label{kasa}
\nabla G=f_1+f_2+ B_a(\nabla G)\,,
\end{equation}
where operator $B_a$ is defined as
\[
 B_af=\int_{\mathbb{R}^3} \nabla _xG^0(x,\xi,z)\frac{\nabla V(\xi)}{1+V(\xi)}f(\xi)d\xi
\]
and \[f_1= \frac{\nabla G^0(x,y,z)}{1+V(y)}, \quad f_2=-z \nabla \int_{\mathbb{R}^3} G^0(x,\xi,z)\frac{V(\xi)}{1+V(\xi)}G(\xi,y,z)d\xi\,.\]
For every $z\notin [0,\infty)$, we have $f_1\in L^1(\mathbb{R}^3)$ and $f_2\in L^2(\mathbb{R}^3)$ by lemma \ref{look}.
Now, consider $B_a$ and notice that it is a contraction in $L^1(\mathbb{R}^3)$ and in $L^2(\mathbb{R}^3)$ for $k=id$ where $d$ is positive and  large. Indeed, it follows from 
\[
\left\|\nabla \frac{e^{-d|x|}}{|x|}\right\|_{L^1(\mathbb{R}^3)}\lesssim d^{-1}\,.
\]
Therefore, the equation $\eqref{kasa}$ has a unique solution and $\nabla G\in L^1+L^2$. To show that this is true for all $z\notin [0,\infty)$, we write the first resolvent identity
\[
R_z=R_{z_0}+(z-z_0)R_{z_0}R_z\,,
\]
where $z_0=-d^2$. This gives
\[
G(x,y,z)=G(x,y,z_0)+(z-z_0)\int_{\mathbb{R}^3} G(x,\xi,z_0)G(\xi,y,z)d\xi\,.
\]
By  \eqref{lukas1} and $\|R_{z_0}\|_{L^2(\mathbb{R}^3),\cal{H}^2(\mathbb{R}^3)}<C(z_0)$, we get 
\[
\nabla \int_{\mathbb{R}^3} G(x,\xi,z_0)G(\xi,y,z)d\xi\in L^2(\mathbb{R}^3)\,.
\]
Finally, we have
\begin{equation}\label{local-loc}
\nabla_x G(x,y,z)\in L^1(\mathbb{R}^3)+L^2(\mathbb{R}^3)
\end{equation}
for all $y\in \mathbb{R}^3$.

\bigskip

\subsection{Asymptotics on the Green's function}

Following the notation in part 1, we define
\begin{equation}\label{opet}
A(r\sigma,y,k)=G(r\sigma +y,y,k^2)/G^0(r \sigma +y,y,k^2),\,   a(r\sigma,y,k)=G(r\sigma,y,k^2)/G^0(r\sigma,0,k^2)\,,
\end{equation}
where $r>0,\, \sigma\in \mathbb{S}^2$. For the free case,
\[
a^0(r\sigma,y,k)=G^0(r\sigma,y,k^2)/G^0(r\sigma,0,k^2)\,.
\]

\begin{theorem}Assume $V$ satisfies  \eqref{2-assu}.  If $k\in \Pi(a,b,1)$,  then
\begin{equation}\label{kkk1}
\sup_{r>|y|+1}\|a(r\sigma,y,k)\|_{L^2(\mathbb{S}^2)}\le C(a,b,|y|,V)(\Im k)^{-1.5}\,,
\end{equation}
\begin{equation}\label{leo}
\lim_{\|V\|\to 0}\sup_{r>|y|+1}\|a(r\sigma,y,k)-a^0(r\sigma,y,k)\|_{L^2(\mathbb{S}^2)}=0\,.
\end{equation}
 The convergence in $\eqref{leo}$ is uniform in $y\in K$ where $K$ is any compact in $\mathbb{R}^3$.\smallskip

For each $k\in \mathbb{C}^+$ and $y\in \mathbb{R}^3$,
\begin{equation}\label{leo1}
\lim_{r\to\infty} \|a(r\sigma,y,k)-a_\infty(\sigma,y,k)\|_{L^2(\mathbb{S}^2)}=0
\end{equation}
and this convergence is uniform in $y\in K$ and in $k\in K_1$, $K_1$ is arbitrary compact in $\mathbb{C}$. 
\end{theorem}

\begin{proof}

The proof is nearly identical to proofs of lemma \ref{mre} and theorems \ref{t1}, \ref{t2} in the first part. We write equation for $u\dd G(x,y,k^2)$ in the following form
\[
(1+V)\Delta u+\nabla V\cdot \nabla u+k^2u=0, \quad |x|>|y|\,.
\]
If $\mu\dd a(x,y,k)$, then
\[
(1+V)\Delta\mu +2(1+V)\mu_r(ik-r^{-1})-k^2V\mu +\nabla V\nabla\mu+V_r\mu(ik-r^{-1})=0\,,
\]
where $r\dd |x|>|y|$. We now proceed as in  the proofs in part 1. We multiply equation by $\bar\mu/|x|^2$ and integrate over the annulus
$r_1<|x|<r_2$ where $r_1>|y|$. Dividing by $ik$, integrating by parts, and taking the real part of both sides give us (after several cancellations)
\begin{equation}\label{key11}
\frac{\Im k}{|k|^2}\int_{r_1<r<r_2} \frac{(1+V)|\nabla \mu|^2 }{|x|^2}dx+\frac{1}{r_2^2}\int_{|x|=r_2}(1+V)|\mu|^2d\sigma_x  =
\end{equation}
\[
\frac{1}{r_1^2}\int_{|x|=r_1}(1+V)|\mu|^2d\sigma_x+\Im k\int_{r_1<|x|<r_2} \frac{V|\mu|^2}{|x|^2}dx-\frac{\Im k}{|k|^2}\int_{r_1<|x|<r_2} V_r\frac{|\mu|^2}{|x|^3}dx-I_1+I_2
\]
and
\[
I_1=\Re \left(  \frac{1}{ikr_1^2} \int_{|x|=r_1} \frac{(1+V)\mu_r\bar \mu}{|x|^2}d\sigma_x \right), \quad  I_2=\Re \left(  \frac{1}{ikr_2^2} \int_{|x|=r_2} \frac{(1+V)\mu_r\bar \mu}{|x|^2}d\sigma_x \right)\,.
\]
Comparing this inequality to \eqref{key1}, we introduce
\begin{eqnarray*}
m(r)\dd r^{-2}\int_{|x|=r}|\mu|^2d\sigma_x, \quad m_1(r)\dd r^{-2}\int_{|x|=r}|\mu_r|^2d\sigma_x, \\ A(r)\dd \int_{r<|x|}\frac{|\nabla\mu|^2}{|x|^2}dx, \quad M(r)\dd \sup_{\rho>r}\int_{\rho}^{\rho+1}m(t)dt, \quad \widehat{M}(r)\dd \sup_{\rho>r}m(\rho)\,.
\end{eqnarray*}
We estimate 
\[
\left|\int_{r_1<|x|<r_2} \frac{V|\mu|^2}{|x|^2}dx\right|
\]
as in \eqref{keke}. In view of approximation \eqref{apa1}, we can repeat arguments in the proof of lemma \ref{mre} (by approximating with $V_{(\rho)}$ first and then sending $\rho\to\infty$) which gives us  
\begin{eqnarray}\label{key611-3}
A(r_1)\le C(a,b,y)\Bigl(\frac{m(r_1)+ \sqrt{m(r_1)m_1(r_1)}}{\Im k}+\|V\|_{[r_1,\infty)}M(r_1)+\Bigr.\\
\nonumber  \Bigl. m(r) \|V\|_{(r_1,\infty)}+\sqrt{A(r)M(r)}\|V\|_{[r_1,\infty)}  +M(r)(r+1)^{-1/2}\|V\|_{[r_1,\infty)}\Bigr)
\end{eqnarray}
and
\begin{eqnarray}\label{key8811-3}
M(r_1)\le
 C(a,b,y)\Bigl(m(r_1)+ \sqrt{m(r_1)m_1(r_1)}+\Im k\|V\|_{[r_1,\infty)}M(r_1)+\Bigr.\hspace{2cm}\\
\nonumber  \Bigl. \Im k\Bigl(m(r) \|V\|_{[r_1,\infty)}+\sqrt{A(r)M(r)}\|V\|_{[r_1,\infty)}  +M(r)(r+1)^{-1/2}\|V\|_{[r_1,\infty)}\Bigr)+\sqrt{A(r_1)M(r_1)}\Bigr)\,.
\end{eqnarray}
Next we  choose $r_1$ large enough to make $\|V\|_{[r_1,\infty)}$ sufficiently small to be able to solve these equations and get
\begin{equation}\label{key8811}
M(r_1)\le C(a,b,y,V)\frac{m(r_1)+\sqrt{m(r_1)m_1(r_1)}}{\Im k}, \quad A(r_1)\le C(a,b,y,V)\frac{m(r_1)+\sqrt{m(r_1)m_1(r_1)}}{\Im k}\,.
\end{equation}
This  $r_1$ depends on $a$,$b$,$V$ and $y$ only. Then, following the proof of \eqref{wt}, we get
\begin{equation}\label{w11t}
\widehat M(r)\lesssim M(r)+\sqrt{A(r)M(r)}\lesssim \frac{ m(r)+\sqrt{m(r)m_1(r)}}{\Im k}\,.
\end{equation}
Lemma \ref{kere} then yields \eqref{kkk1}. Notice that although we obtained these bounds in any $\Pi(a,b,1)$, we can write
\[
A(r)<C(k,|y|,V),\quad \widehat M(r)<C(k,|y|,V), \quad r>|y|+1
\]
for every $k\in \mathbb{C}^+$.

\smallskip

To prove \eqref{leo} and \eqref{leo1}, we notice that \eqref{popul} provides
\[
\mu= \frac{|x|e^{ik(|x-y|-|x|)}}{|x-y|(1+V(y))}-k^2\int_{\mathbb{R}^3} \frac{|x|e^{ik(|x-\xi|+|\xi|-|x|)}}{4\pi |x-\xi||\xi|}\frac{V}{1+V}\mu d\xi+
\]
\[
\int_{\mathbb{R}^3}  \frac{|x|e^{ik(|x-\xi|+|\xi|-|x|)}}{4\pi|x-\xi||\xi|}  \frac{\nabla V}{1+V} \nabla \mu d\xi+ik\int_{\mathbb{R}^3}  \frac{|x|e^{ik(|x-\xi|+|\xi|-|x|)}}{|x-\xi||\xi|}  \frac{ V_r}{1+V}  \mu d\xi-\int_{\mathbb{R}^3}  \frac{|x|e^{ik(|x-\xi|+|\xi|-|x|)}}{|x-\xi||\xi|}  \frac{ V_r}{1+V}  \frac{\mu}{|\xi|} d\xi\,.
\]
This can be rewritten as 
\begin{equation}\label{pert14}
a=a^0+I_1+\ldots+I_5,
\end{equation}
\[
I_1\dd - \frac{|x|e^{ik(|x-y|-|x|)}}{|x-y|}\frac{V(y)}{1+V(y)}
\]
and $I_2,\ldots, I_5$ are defined respectively. \smallskip

Before we proceed with the analysis of every term, we split  integral in the definition of each $I_j, j=2,\ldots,5$ to integrals over $|\xi|<|y|+1$ and over $|\xi|>|y|+1$.  The former has necessary asymptotics since $G\in L^2(\mathbb{R}^3),$ $\nabla G\in L^1(B_{|y|+1}(0))$ by \eqref{local-loc}. Thus, we can assume that integration is done over $|\xi|>|y|+1$, the domain on which we established the bounds for $A$ and $\widehat M$.\smallskip

The first term $I_1$ obviously has the required asymptotics. 
Having estimates \eqref{key611-3} and \eqref{w11t} on $\mu$ and $\nabla\mu$, we will apply theorem \ref{ttt1} to each of $I_2,\ldots,I_5$. For example, to handle $I_2$, we write
\[
\frac{V}{1+V}\mu=|\xi| \Bigl(  \frac{\mu}{|\xi|(1+V)}{\rm div}\, Q        \Bigr)=|\xi|{\rm div}\,\left(    Q\frac{\mu}{|\xi|(1+V)}\right)  -|\xi|Q\nabla\left(  \frac{\mu}{|\xi|(1+V)}            \right)\,.
\]
Recall the definitions of operators $\{B^{(j)}\}$ from subsection 2.3. The contributions from the first and second terms can be estimated by bounds on $B^{(4)}$ and $B^{(3)}$, respectively. Then, for each $k\in \mathbb{C}^+$, we can bound, e.g.,
\[
\int_{|\xi|>|y|+1}\frac{|\mu|^2Q^2}{|\xi|^2(1+V)^2}d\xi
\]
by using \eqref{w11t} and \eqref{2-assu}. The $L^2$ norm  $\left\|\nabla\left(  \frac{\mu}{|\xi|(1+V)}            \right)\right\|_{L^2(B_{|y|+1}^c)}$ can be estimated similarly.

The terms $I_{4(5)}$ can be handled similarly considering that
\[
V_r={\rm div}\,\Bigl(   V\frac{\xi}{|\xi|}   \Bigr)-\frac{2V}{|\xi|}\,.
\]
The contribution from the first term is handled as $I_2$ and from the second one is estimated using the bound for $B^{(2)}$.

For $I_3$, we write
\[
 \frac{\nabla V}{1+V} \nabla \mu=|\xi|\cdot\frac{\nabla V}{1+V}\cdot \frac{\nabla\mu}{|\xi|}
\]
and the estimate for $B^{(3)}$ can be used along with
\[
\int_{|\xi|>|y|+1}\frac{|\nabla \mu|^2}{|\xi|^2}d\xi <C(k,y,V)
\]
as follows by \eqref{key8811} and \eqref{w11t}. Finally, we get \eqref{leo1} and clearly this convergence is uniform in $y$ and in $k$.

Consider \eqref{pert14}. Sending $\|V\|$  to zero in the estimates for $\{I_j\}$, we get 
 \eqref{leo}. 

\end{proof}
In the same way, the estimates for $A(r\sigma,y,k)$ and its asymptotics can be obtained.

 We now turn to the question about finding the harmonic majorant for $\|a_\infty\|_{L^2(\mathbb{S}^2)}$ and $\|A_\infty\|_{L^2(\mathbb{S}^2)}$. Let us focus on $A_\infty$, the analysis of $a_\infty$ is similar. We will need the following statement first.
 
For $\delta>0$, we define $w_\delta\dd (1+|x|)^{1/2+\delta}$.
 \begin{lemma} \label{potluck1} Let $\delta>0$ and assume that $V\in C_c^\infty(\mathbb{R}^3)$ and $\|V\|_\infty<1$. Then, $G(x,y,k^2)-G^0(x,y,k^2)$ can be continuously extended in $k$  to $\mathbb{R}\backslash 0$ as $L^2_{w_\delta^{-1}}(\mathbb{R}^3)$ function. The function $A_\infty(\sigma,y,k)$ can be continuously extended in $k$ to $\mathbb{R}\backslash 0$ as an $L^2(\mathbb{S}^2)$ function.
 \end{lemma}
 
 \begin{proof} The main ingredient of our proof is the limiting absorption principle (LAP) for the operator $D=-{\rm div}(1+V)\nabla$ with short range potential, studied in, e.g.,   \cite{eidus,yama}. LAP claims that $R_z$ can be continuously extended in $z$ from $\mathbb{C}\backslash [0,\infty)$ to $\mathbb{R}^+\pm i0$ as an operator from $L^2_{w_\delta}(\mathbb{R}^3)$ to $L^2_{w_\delta^{-1}}(\mathbb{R}^3)$. This result  is an extension of the LAP established by Agmon for the  Schr\"odinger operators with short range potentials (see, e.g., \cite{rt,yafaev1}).

 For $z\notin [0,\infty)$, write \eqref{popul} and integrate by parts in the third term
\begin{equation}\label{popul1}
G(x,y,z)=\frac{G^0(x,y,z)}{1+V(y)}-z\int_{\mathbb{R}^3} G^0(x,\xi,z)\frac{V(\xi)}{1+V(\xi)}G(\xi,y,z)d\xi-
\end{equation}
\[
\int_{\mathbb{R}^3} G(\xi,y,z){\rm div}_\xi\left(G^0(x,\xi,z)\frac{\nabla V(\xi)}{1+V(\xi)}\right) d\xi\,.
\]
Notice that all integrals involved converge absolutely.
Iterating this identity once (substitute the left hand side into the third term on the right hand side) gives
\begin{equation}\label{popul2}
G(x,y,z)=\frac{G^0(x,y,z)}{1+V(y)}-z\int_{\mathbb{R}^3} G^0(x,\xi,z)\frac{V(\xi)}{1+V(\xi)}G(\xi,y,z)d\xi-
\end{equation}
\[
\int_{\mathbb{R}^3} \frac{G^0(\xi_1,y,z)}{1+V(y)}{\rm div}_{\xi_1}\left(G^0(x,\xi_1,z)\frac{\nabla V(\xi_1)}{1+V(\xi_1)}\right) d\xi_1-
\]
\[
z\int_{\mathbb{R}^3} {\rm div}_{\xi_1}\left(G^0(x,\xi_1,z)\frac{\nabla V(\xi_1)}{1+V(\xi_1)}\right) \int_{\mathbb{R}^3} G^0(\xi_1,\xi_2,z)\frac{V(\xi_2)}{1+V(\xi_2)}G(\xi_2,y,z)d\xi_2          d\xi_1-
\]
\[
\int_{\mathbb{R}^3} {\rm div}_{\xi_1}\left(G^0(x,\xi_1,z)\frac{\nabla V(\xi_1)}{1+V(\xi_1)}\right)              \int_{\mathbb{R}^3}       {\rm div}_{\xi_2}\left(G^0(\xi_1,\xi_2,z)\frac{\nabla V(\xi_2)}{1+V(\xi_2)}\right)     G(\xi_2,y,z)d\xi_2     d\xi_1\,.
\]
Recall that $G(x,y,z)=\overline{G(y,x,\bar{z})}$. Changing the order of integration in the last two integrals gives
\begin{equation}\label{popul3}
\overline{G(y,x,\bar z)}=\frac{\overline{G^0(y,x,\bar z)}}{1+V(y)}- z\int_{\mathbb{R}^3} G^0(x,\xi,z)\frac{V(\xi)}{1+V(\xi)}\overline{G(y,\xi,\bar z)}d\xi-
\end{equation}
\[
\int_{\mathbb{R}^3} \frac{G^0(\xi_1,y,z)}{1+V(y)}{\rm div}_{\xi_1}\left(G^0(x,\xi_1,z)\frac{\nabla V(\xi_1)}{1+V(\xi_1)}\right) d\xi_1-
\]
\[
z\int_{\mathbb{R}^3}\frac{V(\xi_2)}{1+V(\xi_2)}\overline{G(y,\xi_2,\bar z)}d\xi_2   \int_{\mathbb{R}^3}  G^0(\xi_1,\xi_2,z)   {\rm div}_{\xi_1}\left(G^0(x,\xi_1,z)\frac{\nabla V(\xi_1)}{1+V(\xi_1)}\right)        d\xi_1-
\]
\[
          \int_{\mathbb{R}^3}          \overline{G(y,\xi_2,\bar z)}d\xi_2   \int_{\mathbb{R}^3}    {\rm div}_{\xi_2}\left(G^0(\xi_1,\xi_2,z)\frac{\nabla V(\xi_2)}{1+V(\xi_2)}\right)          {\rm div}_{\xi_1}\left(G^0(x,\xi_1,z)\frac{\nabla V(\xi_1)}{1+V(\xi_1)}\right) d\xi_1 \,.
\]
Now, we define three functions 
\[
F_1(\xi_2)\dd \frac{V(\xi_2)}{1+V(\xi_2)}\int_{\mathbb{R}^3}  G^0(\xi_1,\xi_2,k^2)  {\rm div}_{\xi_1}\left(G^0(x,\xi_1,k^2)\frac{\nabla V(\xi_1)}{1+V(\xi_1)}\right)        d\xi_1\,,
\]
\[
F_2(\xi_2)\dd \int_{\mathbb{R}^3}    {\rm div}_{\xi_2}\left(G^0(\xi_1,\xi_2,k^2)\frac{\nabla V(\xi_2)}{1+V(\xi_2)}\right)          {\rm div}_{\xi_1}\left(G^0(x,\xi_1,k^2)\frac{\nabla V(\xi_1)}{1+V(\xi_1)}\right) d\xi_1 \,,
\]
\[
F_3(\xi_2)\dd G^0(x,\xi_2,k^2)\frac{V(\xi_2)}{1+V(\xi_2)}\,.
 \]
Since we have exact formula for $G^0(\xi_1,\xi_2,k^2)$, simple estimates for the integrals show that $F_1,F_3$ can be extended in $k$ continuously to $\mathbb{R}\backslash 0$ as $L^2(\mathbb{R}^3)$ functions in $\xi_2$. The most singular term in $F_2$ can be written as
\[
q(\xi_2)\int_{\mathbb{R}^3} \nabla_{\xi_2} G^0(\xi_1,\xi_2,k^2)\nabla_{\xi_1} G^0(x,\xi_1,k^2)q(\xi_1)d\xi_1, \quad q\dd \frac{\nabla V}{1+V}\,.
\]
This expression can be continued in $k$ to $\mathbb{R}\backslash 0$ as $L^1_{\xi_2}(\mathbb{R}^3)$ function. However, since $\nabla_{\xi_1}G^0=-\nabla_{\xi_2}G^0$ and $(-\Delta_{\xi_1}-k^2)G^0(\xi_1,\eta,k^2)=\delta_\eta$, we can rewrite it in the following form
\[
-q(\xi_2)\int_{\mathbb{R}^3} \nabla_{\xi_1} G^0(\xi_1,\xi_2,k^2)\nabla_{\xi_1} G^0(x,\xi_1,k^2)q(\xi_1)d\xi_1=I_1+I_2,
\]
\[
I_1=-k^2q(\xi_2)\int_{\mathbb{R}^3}  G^0(\xi_1,\xi_2,k^2) G^0(x,\xi_1,k^2)q(\xi_1)d\xi_1-q^2(\xi_2)G^0(x,\xi_2,k^2)\,,
\]
\[
I_2=
q(\xi_2)\int_{\mathbb{R}^3} \nabla_{\xi_1} G^0(\xi_1,\xi_2,k^2) G^0(x,\xi_1,k^2)\nabla_{\xi_1}q(\xi_1)d\xi_1\,.
\]
Now, the elementary properties of the convolution and  explicit form of $G^0$ show that $F_2$ can be continued to $\mathbb{R}\backslash 0$ as $L^2_{\xi_2}(\mathbb{R}^3)$ function.
Notice that the calculation for $I_1$ can be performed on the Fourier side as well.

Next, we consider \eqref{popul3} as equality for functions in $y$ where $x$ is fixed. Notice that compactness of support of $V$ guarantees that $F_{1(2,3)}\in L^2_{w_\delta}(\mathbb{R}^3)$ and LAP for $D$ shows that $G(y,x,k^2)-G^0(y,x,k^2)$ can be continued to $\mathbb{R}\backslash 0$ as an element of $L^2_{w_\delta^{-1}}(\mathbb{R}^3)$. Then, the formula \eqref{popul1} implies that $A(r\sigma,y,k)$ can be defined for all $r>0$ and $y\in \mathbb{R}^3$ as $L^2(\mathbb{S}^2)$ valued function continuous in $k$ up to $\mathbb{R}\backslash 0$. The existence of $A_\infty(\sigma,y,k)$ and its continuity in $k$ are immediate as well.

 \end{proof}

  The formula \eqref{reguu} can be rewritten in the following form:
\[
(D^2-k^4)^{-1}=(D-k^2)(D+k^2), \quad k\in \mathbb{C}^+\,.
\]
Thus,
\[
(D^2-k^4)^{-1}\delta_0=\int_{\mathbb{R}^3} G(x,\xi,k^2)G(\xi,0,-k^2)d\xi\,.
\]
Since $G(\xi,0,-k^2)\in L^2_\xi(\mathbb{R}^3)$ and $R_{k^2}$ maps $L^2(\mathbb{R}^3)$ to $\cal{H}^2(\mathbb{R}^3)\subset C(\mathbb{R}^3)$, the function $m(k)$ defined by \eqref{newm} has  imaginary part positive in $\mathbb{C}^+$ and has representation \eqref{refew}. The following analog of lemma \ref{kulik} holds true. 

\begin{lemma} \label{kulik1} Assume that $V\in C_c^\infty(\mathbb{R}^3)$. Then,
\[
32\pi^2\kappa \mu'(\kappa^4)=\|{A_\infty}(\sigma,0,\kappa)\|^2_{L^2(\mathbb{S}^2)}, \quad \kappa> 0\,.
\]
\end{lemma}
\begin{proof}
Let $u\dd (D^2-k^4)^{-1}f=(2k^2)^{-1}(R_{k^2}-R_{-k^2})f$,
where $f$ is any test function, i.e., $f\in C_c^\infty(\mathbb{R}^3)$. Multiply  equation $(D^2-k^4)u=f$ by $\overline u$ and integrate over $B_R(0)$ with large $R$. The analog of \eqref{los} is
\[
\int_{|x|<R}\Bigl(({\rm div}(1+V)\nabla+k^2)({\rm div}(1+V)\nabla-k^2)u\Bigr)\overline udx=\int_{\mathbb{R}^3} f\overline udx\,.
\]
 Take imaginary part of both sides and send $k\to \kappa\in \mathbb{R}\backslash 0$.
After integration by parts, we have
\[
\int_{|x|=R}\Bigl(  \partial_r  \Bigl( ({\rm div}(1+V)\nabla-k^2)\Bigr) u\Bigr)(1+V)\overline u- \Bigl(    \Bigl( ({\rm div}(1+V)\nabla-k^2)\Bigr) u\Bigr)(1+V)\partial_r \overline u   dx=\int_{\mathbb{R}^3} f\overline udx\,.
\]
Lemma \ref{potluck1}  gives asymptotics
\[
u(x,k)=\frac{e^{ik|x|}}{8\pi k^2|x|}\left(\int_{\mathbb{R}^3} a_\infty(\widehat x,y,k)f(y)dy+o(1)\right),\quad |x|\to\infty\,,
\]
which holds in $k$ up to $\mathbb{R}\backslash 0$ and can be differentiated in $x$.  Substitute it into the identity to get
\[
\int_{\mathbb{S}^2}\left|\int_{\mathbb{R}^3} a_\infty(\sigma,y,\kappa)f(y)dy\right|^2d\sigma=32\pi^3\kappa \mu'_f(\kappa^4)\,.
\]
Taking $\{f_n\}\to \delta_0$ gives the statement of the lemma.
\end{proof}
This lemma provides the harmonic majorants for $\|A_\infty(\sigma,y,k)\|_{L^2(\mathbb{S}^2)}$ and $\|a_\infty(\sigma,y,k)\|_{L^2(\mathbb{S}^2)}$. 
In particular, theorem \ref{t4oo} and \eqref{kochka1}, \eqref{kochka2} hold for $A_\infty(\sigma,y,k)$. Repeating the proof  of lemma \ref{polk}, we get \eqref{polk2} and thus there exists $a_\infty(\sigma,y,\kappa)$ such that
\begin{equation}\label{contour}
\lim_{\epsilon\to 0}\int_a^b\|a_{\infty}(\sigma,y,\kappa+i\epsilon)-a_\infty(\sigma,y,\kappa)\|^2_{L^2(\mathbb{S}^2)}d\kappa=0
\end{equation}
for every $y$ and every $[a,b]$ not containing $0$.

\bigskip

\subsection{Proof of the main theorem}

\begin{proof}{\it (of theorem \ref{main-th} )}  Since $\cal{N}$ is dense in $L^2(\mathbb{R}^3)$ and $e^{it\sqrt{D}},e^{it\sqrt{H_0}}$ preserve $L^2(\mathbb{R}^3)$ norm, it is sufficient to take $f\in \cal{N}$ and prove that the limit
${\lim}_{t\to\pm \infty}e^{it\sqrt{D}}e^{-it\sqrt{H_0}}f$ exists in $L^2(\mathbb{R}^3)$.
Let $\phi\dd e^{it\sqrt{H_0}}f$. Consider
\[
\psi\dd e^{-it\sqrt D}\phi
\]
when $t\to+\infty$. The case $t\to -\infty$ is similar.

To prove existence of the limit of $\psi$ in $L^2(\mathbb{R}^3)$, it is sufficient to show that

\begin{itemize}
\item[(A)] As $t\to+\infty$,  $\psi(x,t)$ converges in $L^2(B_R(0))$ for every $R>0$.

\item[(B)]  $\{\psi(x,t)\}$ is ``tight'' in the following sense
\[
\lim_{R\to\infty}\limsup_{t\to +\infty} \|\psi(x,t)\|_{L^2(R_R^c)}=0\,.
\]
\end{itemize}
We start proving (A) by writing: 
\begin{equation}\label{formula7}
e^{-it\sqrt D}\phi =e^{-it\sqrt{D}}e^{it\sqrt{H_0}} q_n(\sqrt{H_0})q_n^{-1}(\sqrt{H_0})f=e^{-it\sqrt{D}} q_n(\sqrt{H_0})e^{it\sqrt{H_0}}q_n^{-1}(\sqrt{H_0})f\,,
\end{equation}
where $q_n$ is defined in \eqref{kukuru} and parameters $a,b$ are chosen such that  supp$(\widehat f)\subset \{\xi: 0<a<\frac{|\xi|}{2\pi}<b\}$. 
If we denote $f_1=q_n^{-1}(\sqrt{H_0})f$, then $f_1\in \cal{N}$ by the choice of $a$ and $b$. The parameter $n$ will be chosen later. 
From lemma \ref{lek}, we obtain
\begin{equation}\label{zo1}
e^{-it\sqrt{D}}\phi=e^{-it\sqrt{D}}q_n(\sqrt{H_0})e^{it\sqrt{H_0}}f_1 =e^{-it\sqrt{D}} q_n(\sqrt{D})\phi_1+\epsilon_1, \quad \lim_{t\to\infty}\|\epsilon_1\|_2= 0\,,
\end{equation}
where $\phi_1\dd e^{it\sqrt{H_0}}f_1$.
We performed this algebra to be able to write formula \eqref{formula1} for
\begin{equation}\label{sera}
e^{-it\sqrt{D}} q_n(\sqrt{D})  \phi_1=\frac{1}{2\pi i}\int_{\Gamma_{a,b}} e^{-itk}(D-k^2)^{-1}p_n(k) \phi_1 dk\,.
\end{equation}
The properties of free evolution $e^{it\sqrt{H_0}}$, i.e., preservation of $L^2(\mathbb{R}^3)$ norm and estimate \eqref{local1},  allow us to replace $\phi_1$ by a function
\begin{equation}\label{zo2}
\widetilde\phi_1=\phi_1 w_\rho(|x|-t)\,,
\end{equation}
where $w_\rho(\tau), \, \tau\in \mathbb{R}$ is smooth, even, nonnegative function that satisfies three properties:
\begin{itemize}
\item
$
w_\rho(\tau)=1, \quad |\tau|<\rho\,,
$
\item
$
w_\rho(\tau)=0, \quad |\tau|>\rho+1\,,
$
\item $0\le w_\rho\le 1, \quad \tau\in \mathbb{R}$\,.
\end{itemize}
Since the operators $q_n(\sqrt{D})$,  $ e^{it\sqrt{H_0}}$,  $ e^{-it\sqrt{D}}$ are bounded from $L^2(\mathbb{R}^3)$ to $L^2(\mathbb{R}^3)$ and their norms are uniformly bounded in $t$, the error made by that change can be made arbitrarily small by choosing  $\rho$ large and then sending $t\to\infty$.

We collect now the  properties of $\widetilde \phi_1$ that will be important later on:
\begin{itemize}
\item[(P1)] $\lim_{\rho\to\infty}\limsup_{t\to\infty}\|\phi_1-\widetilde\phi_1\|_2=0$. We will fix $\rho$ large enough and $t$--independent.
\item[(P2)] $\widetilde \phi_1$ is supported on the annulus $\{x: |x|\in [t-\rho-1,t+\rho+1]\}$.

\item[(P3)] $\widetilde \phi_1$ has asymptotics (see \eqref{poiu})
\begin{equation}\label{tripleone}
 \widetilde \phi_1((t+\tau)\sigma,t)=\omega_\rho(\tau) \frac{-\sigma \cal{R}( \nabla f_1)(\tau,\sigma)+i\cal{R}(|\Delta|^{1/2}f_1)(\tau,\sigma)}{4\pi t}+o(t^{-1}),
\end{equation}
uniform in $\tau\in \mathbb{R}$ and $\sigma\in \mathbb{S}^2$.
In particular, $\|\widetilde \phi_1\|_\infty<C(f)t^{-1}$ for all $t>1$.

\item[(P4)] $\widetilde \phi_1$ is sufficiently smooth 
\[
 \|D^j\widetilde \phi_1\|_\infty<C_jt^{-1}, \quad j\in \mathbb{N}\,.
\]
This follows from the definition of $\widetilde \phi_1$ and smoothness of $f$.
\end{itemize}

Consider the integral in \eqref{sera} with $\phi_1$ replaced by $\widetilde \phi_1$. For $\Gamma_{a,b}$ we can write $\Gamma_{a,b}=\Gamma_{a,b}^+\cup \Gamma_{a,b}^-$, where $\Gamma_{a,b}^{\pm}\in \mathbb{C}^{\pm}$. We studied $G(x,y,k^2)$ assuming that $k\in \mathbb{C}^+$, so, since $(-k)^2=k^2$, we will write
\begin{equation}\label{perd1}
(D-k^2)^{-1} \widetilde \phi_1=\int_{\mathbb{R}^3} G(x,y,(-k)^2)\widetilde \phi_1(y)dy
\end{equation}
for $k\in \Gamma_{a,b}^-$. Thus, we will need to control  $G(x,y,\beta^2)$ where $\beta\in -\Gamma_{a,b}^-\subset \mathbb{C}^+$.

We will start with writing the following estimate. For every $R>0$,
\begin{equation}\label{zorro0}
\left\|\frac{1}{2\pi i}\int_{\Gamma_{a,b}} e^{-itk} (D-k^2)^{-1}p_n   \widetilde \phi_1    dk\right\|
_{L^2(B_R(0))}\lesssim \int_{\Gamma_{a,b}} e^{t\Im k} \sup_{\|h\|_{L^2(B_R(0))}=1}|
\langle (D-k^2)^{-1}\widetilde \phi_1,h\rangle| |p_n|d|k|
\end{equation}
by duality.
Since
\begin{equation}
\|\widetilde \phi_1\|_\infty <C(f) t^{-1}, \quad {\rm supp} (\widetilde \phi_1)\subset \{x:||x|-t|<\rho+1\},
\end{equation}
and $((D-k^2)^{-1})^*=(D-\bar k^2)^{-1}$, we have
\[
|\langle \widetilde \phi_1,(D-{\overline{k}}^2)^{-1}h\rangle|\le\frac{ C(f)}{t}\int_{t-\rho-1<|x|<t+\rho+1}|u|dx\,,
\]
for every $k\in \Gamma_{a,b}^-$, where
\[
 u=(D-{\overline{k}}^2)^{-1}h=G^0(x,0,\overline{k}^2)\int_{|y|<R} a(x,y,\overline{k}^2)h(y)dy
\]
as follows from the definition of $a$. If $k\in \Gamma_{a,b}^+$, then $\overline{k}\in \Gamma_{a,b}^-$, and we need to make modification as in \eqref{perd1}.
We use estimate \eqref{kkk1} on $a(x,y,k)$ to obtain
\begin{equation}\label{zorro1}
\frac{1}{t}\int_{t-\rho<|x|<t+\rho}|u|dx<C(R,\rho,a,b)e^{-|\Im k|t}|\Im k|^{-1.5}
\end{equation}
after applying Cauchy-Schwarz and $\|h\|_2\le 1$.
This amounts to absolute convergence in $k$ and uniform boundedness of the integral
\begin{equation}\label{yura-u}
 \int_{\Gamma_{a,b}} e^{t\Im k} \sup_{\|h\|_{L^2(B_R(0))}=1}|
\langle (D-k^2)^{-1}\widetilde \phi_1,h\rangle| |p_n|d|k|\le C(\rho,R,a,b,f)\,,
\end{equation}
provided that $n=3$ which is our choice of $n$ from now on.

Now, we will show that for every $k\in \Gamma_{a,b}, k\neq a,b$, the integrand in \eqref{sera} converges in $L^2(B_R(0))$ as $t\to +\infty$.  We have
\[
e^{-itk}\Bigl((D-k^2)^{-1}\widetilde\phi_1\Bigr)(y)=e^{-ikt}\int_{\mathbb{R}^3} G(y,x,k^2)\widetilde \phi_1(x)dx=e^{-ikt}\int_{\mathbb{R}^3}  \overline{G(x,y,\overline{k}^2)}\widetilde\phi_1(x)dx
\]
by identity \eqref{simol}. We can write 
\[
e^{-ikt}\int_{\mathbb{R}^3}  \overline{G(x,y,\overline{k}^2)}\widetilde\phi_1(x)dx=e^{-ikt}\int_{\mathbb{R}^3}  \frac{e^{ik|x|} \overline{a(x,y,-\overline{k})}}{4\pi|x|}\widetilde\phi_1(x)dx
\]
if $k\in \Gamma^+_{a,b}$ and 
\[
e^{-ikt}\int_{\mathbb{R}^3}  \overline{G(x,y,\overline{k}^2)}\widetilde\phi_1(x)dx=e^{-ikt}\int_{\mathbb{R}^3}  \frac{e^{-ik|x|} \overline{a(x,y,\overline{k})}}{4\pi|x|}\widetilde\phi_1(x)dx
\]
if $k\in \Gamma^-_{a,b}$. Now, we use asymptotics of $a$ (check \eqref{leo1}) and of  $\phi_1$ (check \eqref{tripleone}) for $r=|x|\to\infty$ to conclude that 
\begin{eqnarray}\label{kusto}
e^{-ikt}\int_{\mathbb{R}^3}  \overline{G(x,y,\overline{k}^2)}\widetilde\phi_1(x)dx\to \cal{G},
\\\nonumber
\cal{G}(k,y)\dd 
\frac{1}{(4\pi)^2} \int_{\mathbb{S}^2}\overline{a_\infty(\sigma,y,-\overline k)} \int_{\mathbb{R}} 
\omega_\rho(\tau)e^{ik\tau} (-\sigma\cal{R}(\nabla f_1)(\tau,\sigma)+i\cal{R}(|\Delta|^{1/2}f_1)(\tau,\sigma))d\tau d\sigma 
\end{eqnarray}
if $t\to\infty$ and this convergence is uniform in $y\in B_R(0)$ and $k\in K_1$ where $K_1$ is any compact in $\mathbb{C}^+$. On the other hand, if $k\in \Gamma^-_{a,b}$, then
\[
e^{-ikt}\int_{\mathbb{R}^3}  \overline{G(x,y,\overline{k}^2)}\widetilde\phi_1(x)dx\to 0
\]
uniformly in $y\in B_R(0)$ and $k\in K_1\subset \mathbb{C}^-$. Together with the unform bound \eqref{yura-u}, we get
\begin{equation}\label{limki}
(e^{-it\sqrt{D}}q_n(\sqrt{D})\widetilde\phi_1)(y)\to \frac{1}{2\pi i} \int_{\Gamma^+_{a,b}} p_n(k) \cal{G}(k,y)dk, \quad t\to +\infty
\end{equation}
and this convergence is in $L^2(B_R(0))$.
Since $\rho$ can be chosen arbitrarily large, we have (A).\smallskip

{\bf Remark.} It is now instructive to discuss the importance of the cut-off $\omega_\rho$ which one might consider to be artificial. In fact, it is crucial for our proof. Indeed, the interior integral in the definition of $\cal{G}$ represents a Fourier integral of a Radon transform which is not even well-defined if $k\notin \mathbb{R}$ unless we introduce a cut-off. With $\omega_\rho$ present, we can now say that this interior integral is entire function of exponential type and integrals in $\sigma$ and $k$ can be controlled.

\bigskip

The following lemma is immediate from the proof given above and it will be used in the proof of~(B). 
\begin{lemma}\label{stability} If $f\in L^2(\mathbb{R}^3)$, then
\[
\lim_{\|V\|\to 0}\limsup_{t\to\infty}\|e^{-it\sqrt{D}}e^{it\sqrt{H_0}}f-f\|_2=0\,.
\]
\end{lemma}
\begin{proof}We will use \eqref{leo}. Notice that $\|a_\infty(\sigma,y,k)-a_\infty^0(\sigma,y,k)\|_2$ converges to zero uniformly in $k\in K_1$ and $y\in B_R(0)$ when $\|V\|\to 0$. Therefore, 
substituting ``$a_\infty=a_\infty^0+\overline{o}(1)$'' into the formula \eqref{kusto} and recalling that $e^{-it\sqrt{H_0}}e^{it\sqrt{H_0}}f=f$, we get 
\[
\lim_{\|V\|\to 0}\limsup_{t\to\infty}\|e^{-it\sqrt{D}}e^{it\sqrt{H_0}}f-f\|_{L^2(B_R(0))}=0
\]
for every $R>0$. Since $e^{-it\sqrt{D}}$ and $e^{it\sqrt{H_0}}$ both preserve the $L^2(\mathbb{R}^3)$ norm, we get the statement of the lemma.

\end{proof}

 We now recall the following notation. Given $V$, we define
\[
V^{(\rho)}=V-V_{(\rho)}
\]
and $V_{(\rho)}$ is defined in \eqref{tru}. Clearly
\[
\lim_{\rho\to\infty}\|V^{(\rho)}\|\to 0\,.
\]\bigskip

We now turn to proving (B) which is more involved. For fixed large $R$, we need to estimate the following expression
\begin{equation}\label{smo}
\limsup_{t\to+\infty}\|\chi_{|x|>R}e^{-it\sqrt{D}}\phi(t)\|_2\,.
\end{equation}
Take  $R_{1(2)}$ - two large parameters that we will specify later. At that moment, we only require that $R_1<R_2/2, R_2<R/2$.\smallskip

Before giving the formal proof, we want to explain an idea. To show that \eqref{smo} is small for large $R$, we will prove that, given large $t$, the function $e^{-i\tau\sqrt{D}}\phi(t)$ at $\tau=t-R_2$ has $L^2(\mathbb{R}^3)$ norm localized to $B_{CR_2}(0)$ when $t$ and $R_2$ are large enough and $C>1$ is an absolute constant to be specified later. Then, we argue that in  time increment  $\Delta \tau=R_2$ the function $e^{-i\tau\sqrt{D}}\phi(t)$ can not have significant part of its $L^2(\mathbb{R}^3)$ norm carried outside  $B_R(0)$ by the group $e^{-iR_2\sqrt{D}}$ if $R$ is much larger than $R_2$ and then the proof is finished because $e^{-iR_2\sqrt{D}}e^{-i(t-R_2)\sqrt{D}}\phi(t)=e^{-it\sqrt{D}}\phi(t)$, as needed. 
However, the question remains: how do we show that  $e^{-i(t-R_2)\sqrt{D}}\phi(t)$ is localized to $B_{CR_2}(0)$ with large $C$? To do that we prove that the function  $e^{-i(t-R_2)\sqrt{D}}\phi(t)$ depends very little on the value of potential $V$ in the ball $B_{R_1}(0)$ where $R_1$ is much smaller than $R_2$. This suggests that it makes sense to consider new operator $D_1$ with potential $V^{(R_1)}$ and show that $e^{-i(t-R_2)\sqrt{D_1}}\phi(t)$ has the right localization. For that purpose, we write $e^{-i(t-R_2)\sqrt{D_1}}\phi(t)=e^{itR_2\sqrt{D_1}}e^{-it\sqrt{D_1}}\phi(t)$ and notice that $e^{-it\sqrt{D_1}}\phi(t)$ is close to $f$ in $L^2(\mathbb{R}^3)$ norm if $R_1$ is large because of lemma \ref{stability}. Since $f$ is fixed, we conclude that $e^{itR_2\sqrt{D_1}}$ has small $L^2(\mathbb{R}^3)$ norm outside $B_R(0)$ if $R$ is much larger than $R_2$. Therefore, the question about localizing $e^{-i(t-R_2)\sqrt{D}}\phi(t)$ is resolved positively.

 To carry out this program, two things are clearly needed. Firstly, we need to control the ``speed of propagation'' of the function whose support is known. Indeed, that has been claimed several times in the outline given above. Notice  that although the general principle of ``finite speed of propagation'' for hyperbolic equations (\cite{evans}, p.395, theorem 8) does give some information in terms of Riemannian metric, it is not sharp enough for us. Secondly, we need to make sure that the value of potential ``far from the solution'' does not affect this solution. This will be achieved by employing the Duhamel formula \eqref{duhamel}. \smallskip

Recall, that (see \eqref{zo1}, \eqref{zo2})
\[
e^{-it\sqrt{D}}\phi=e^{-it\sqrt{D}}q_n(\sqrt{D})\phi_1+\epsilon_1=e^{-it\sqrt{D}}q_n(\sqrt{D})\widetilde \phi_1+\epsilon_1+\epsilon_2\,,
\]
where $\|\epsilon_1\|_2\to 0$ as $t\to\infty $, $\lim_{\rho\to\infty}\limsup_{t\to\infty}\|\epsilon_2\|_2=0$ and $\widetilde \phi_1$ satisfies four properties (P1)--(P4).  Thus, we only need to prove (B) for $e^{-it\sqrt{D}}q_n(\sqrt{D})\widetilde \phi_1(t)$ when $\rho$ is fixed. \smallskip

We split the proof into several steps:

\smallskip
\begin{itemize}
\item[(B.1)]  {\it Consider $\chi_{B_{{R_1}}(0)}e^{ -i\tau \sqrt{D}}q_n(\sqrt{D})\widetilde \phi_1(t)$
and prove that its $L^2(\mathbb{R}^3)$ norm is small  for all $\tau:  0<\tau<t-R_2$.\smallskip}

More precisely, we have
\begin{lemma}\label{zorro8} For every $\tau\in [0,t-R_2]$ and $n\in \mathbb{N}$, we have
\begin{equation}\label{zorro4}
\|\chi_{B_{{R_1}}(0)}e^{ -i\tau \sqrt{D}}q_n(\sqrt{D})\widetilde \phi_1(t)\|_2\le \frac{C(\rho,f,R_1,n,a,b)}{(t-\tau)^{n-0.5}}
\end{equation}
and
\begin{equation}\label{zorro5}
\|\chi_{B_{{R_1}}(0)}De^{ -i\tau \sqrt{D}}q_n(\sqrt{D})\widetilde \phi_1(t)\|_2\le \frac{C(\rho,f,R_1,n,a,b)}{(t-\tau)^{n-0.5}}\,.
\end{equation}

\end{lemma}
\begin{proof}
From the integral representation \eqref{sera} and estimates \eqref{zorro0}, \eqref{zorro1}, we get
\[
\|\chi_{B_{{R_1}}(0)}e^{ -i\tau \sqrt{D}}q_n(\sqrt{D})\widetilde \phi_1(t)\|_2\le C(R_1,\rho,f,a,b)\int_{\Gamma_{a,b}} \frac{e^{\tau \Im k}e^{-t|\Im k|} |p_n(k)|}{|\Im k|^{1.5}}d|k|\,.
\]
Recall that $p_n$ has roots at $k=a$ and $k=b$ of degree $n$, thus, the simple integration yields \eqref{zorro4}. 

The second inequality can be proved in the same way because
\[
De^{ -i\tau \sqrt{D}}q_n(\sqrt{D})\widetilde \phi_1(t)=e^{ -i\tau \sqrt{D}}q_n(\sqrt{D}) \Bigl(D\widetilde \phi_1(t)\Bigr)
\]
and $D\widetilde \phi_1(t)$ satisfies the same properties as  $\widetilde \phi_1(t)$ since $D=-(1+V)\Delta-\nabla V\nabla$ is local operator and $\widetilde \phi_1$ is smooth (property (P4)).
\end{proof}

\item[(B.2)] {\it  Use Duhamel formula to show that the influence of $V_{(R_1)}$  on  $e^{-i(t-R_2)\sqrt{D}}q_n(\sqrt{D})\widetilde \phi_1(t)$ is ``negligible'' if $R_2$ is much larger than $R_1$ and $t\to\infty$.}

\begin{lemma} Given any $R_1$ and $n\in \mathbb{N}$, we have
 \[
\lim_{R_2\to\infty}\limsup_{t\to\infty}\| e^{-i(t-R_2)\sqrt{D}} q_n(\sqrt{D})\widetilde \phi_1(t)       - e^{-i(t-R_2)\sqrt{D_1}}    q_n(\sqrt{D})\widetilde \phi_1(t)      \|_2=0\,,
 \]
where $D_1=-{\rm div}\Bigl( (1+V^{(R_1)})\nabla \Bigr)$.
\end{lemma}
\begin{proof}
Define the function
\[
d(k)=\frac{k}{k^2+1}\,.
\]
It is analytic inside each $\Gamma_{a,b}$, its restriction to $\mathbb{R}$ is continuous and decays at infinity. Its inverse $d^{-1}(k)$ is analytic away from zero. 
The  introduction of $d$ will be explained in the due course. Repeating the arguments from the previous lemma, we get
\begin{equation}\label{poi101}
\|\chi_{B_{{R_1}}(0)}e^{ -i\tau \sqrt{D}}d^{-1}(\sqrt{D})q_n(\sqrt{D})\widetilde \phi_1(t)\|_2\le \frac{C(\rho,f,R_1,n,a,b)}{(t-\tau)^{n-0.5}}\,,
\end{equation}
 \begin{equation}\label{poi2}
\|\chi_{B_{{R_1}}(0)}De^{ -i\tau \sqrt{D}}d^{-1}(\sqrt{D})q_n(\sqrt{D})\widetilde \phi_1(t)\|_2\le \frac{C(\rho,f,R_1,n,a,b)}{(t-\tau)^{n-0.5}}\,,
\end{equation}
and analogous estimates hold for $D_1$ evolution. 
\begin{equation}\label{zuzu}
\|\chi_{B_{{R_1}}(0)}e^{ -i\tau \sqrt{D_1}}d^{-1}(\sqrt{D_1})q_n(\sqrt{D_1})\widetilde \phi_1(t)\|_2\le \frac{C(\rho,f,R_1,n,a,b)}{(t-\tau)^{n-0.5}}\,,
\end{equation}
 \[
\|\chi_{B_{{R_1}}(0)}D_1e^{ -i\tau \sqrt{D_1}}d^{-1}(\sqrt{D_1})q_n(\sqrt{D_1})\widetilde \phi_1(t)\|_2\le \frac{C(\rho,f,R_1,n,a,b)}{(t-\tau)^{n-0.5}}
\]
for all $\tau\in [0,t-R_2]$.
Indeed, the formula \eqref{formula1} can be rewritten for $e^{ -i\tau \sqrt{D}}d^{-1}(\sqrt{D})q_n(\sqrt{D})$
in the same way due to analyticity of $d^{-1}$ away from zero. Then, the estimates from proof of lemma \ref{zorro8} go through.

Now, consider two functions
\[
u(x,\tau)\dd e^{-i\tau \sqrt{D}}d^{-1}(\sqrt{D})q_n(\sqrt{D})\widetilde \phi_1, \quad u_1(x,\tau)\dd e^{-i\tau \sqrt{D_1}}d^{-1}(\sqrt{D_1})q_n(\sqrt{D_1})\widetilde \phi_1\,.
\]
They solve
\[
u_{\tau\tau}=-Du=-D_1u+F, \quad F=(D_1-D)u, \quad {u_1}_{\tau\tau}=-D_1u_1
\]
and satisfy initial conditions
\[
u_0\dd u(x,0)=d^{-1}(\sqrt{D})q_n(\sqrt{D})\widetilde \phi_1, \quad u_1\dd u_\tau(x,0)=-id^{-1}(\sqrt{D})\sqrt{D}q_n(\sqrt{D})\widetilde \phi_1\,,
\]
\[
{u_1}_0\dd u_1(x,0)=d^{-1}(\sqrt{D_1)}q_n(\sqrt{D_1})\widetilde \phi_1, \quad {u_1}_1\dd {u_1}_\tau(x,0)=-id^{-1}(\sqrt{D_1})\sqrt{D_1}q_n(\sqrt{D_1})\widetilde \phi_1\,.
\]
The Duhamel formula \eqref{duhamel} written for $u$ gives
\[
u=\cos(\tau \sqrt{D_1})u_0+\frac{\sin(\tau\sqrt{D_1})}{\sqrt{D_1}}u_1+\int_0^\tau\frac{\sin((\tau-\xi)\sqrt{D_1})}{\sqrt{D_1}}F(\xi)d\xi\,.
\]
Subtracting the identity
\[
u_1=\cos(\tau \sqrt{D_1}){u_1}_0+\frac{\sin(\tau\sqrt{D_1})}{\sqrt{D_1}}{u_1}_1
\]
 from this equation gives us
\[
\delta u=\cos(\tau\sqrt{D_1})(u_0-{u_1}_0)+\frac{\sin(\tau\sqrt{D_1})}{\sqrt{D_1}}(u_1-{u_1}_1)+\int_0^\tau\frac{\sin((\tau-\xi)\sqrt{D_1})}{\sqrt{D_1}}F(\xi)d\xi\,,
\]
where $\delta u\dd u-u_1$. Apply an operator $d(\sqrt{D_1})$ to both sides to get
\begin{eqnarray}
d(\sqrt{D_1})\delta u=d(\sqrt{D_1})\cos(\tau\sqrt{D_1})(u_0-{u_1}_0)+\nonumber
\\
d(\sqrt{D_1})\frac{\sin(\tau\sqrt{D_1})}{\sqrt{D_1}}(u_1-{u_1}_1)+d(\sqrt{D_1})\int_0^\tau\frac{\sin((\tau-\xi)\sqrt{D_1})}{\sqrt{D_1}}F(\xi)d\xi\,.
\end{eqnarray}
Now, we can appreciate the role of auxiliary function $d$. Notice that
\[
d(\lambda)\cos(\tau\lambda),\quad  d(\lambda)\frac{\sin (\tau\lambda)}{ \lambda}
\]
are bounded uniformly in $\lambda$ and $\tau$. From lemma \ref{lek},
\[
\lim_{t\to\infty}\|u_0-{u_1}_0\|_2= 0, \quad \lim_{t\to\infty} \|u_1-{u_1}_1\|_2=0\,.
\]
We estimate $F$ as follows
\[
F=(H_1-H)u=-(V^{(R_1)}-V)\Delta u-\nabla (V^{(R_1)}-V)\nabla u=V_{(R_1)}\Delta u+\nabla V_{(R_1)}\nabla u\,.
\]
Notice that $V_{(R_1)}$ is supported on $B_{R_1+1}(0)$. On  the ball $B_{R_1+2}(0)$, we have estimates \eqref{poi101} and \eqref{poi2} for $u$ and $Hu=-(1+V)\Delta u-\nabla V\nabla u$ in $L^2(B_{R_1+2}(0))$. By Interior Regularity Theorem, we get analogous estimates on $\Delta u$ and $\nabla u$ in $L^2(B_{R_1+1}(0))$.  Thus, we have
\begin{eqnarray*}
\lim_{t\to\infty}\|d(\sqrt{D_1})\delta u(t-R_2)\|_2\le C(\rho,f,R_1,n,a,b)\limsup_{t\to\infty} \int_0^{t-R_2} \frac{1}{(t-\tau)^{n-0.5}}d\tau\le\\ C(\rho,f,R_1,n,a,b)R_2^{-(n-1.5)}
\end{eqnarray*}
and
\begin{equation}\label{popi}
\lim_{R_2\to\infty} \limsup_{t\to\infty}\|d(\sqrt{D_1})\delta u(t-R_2)\|_2=0
\end{equation}
because we have chosen $n= 3$.
Recall that 
\[
d(\sqrt{D_1})\delta u(t-R_2)=d(\sqrt{D_1})\Bigl(  e^{-i(t-R_2) \sqrt{D}}d^{-1}(\sqrt{D})q_n(\sqrt{D})\widetilde \phi_1-e^{-i(t-R_2) \sqrt{D_1}}d^{-1}(\sqrt{D_1})q_n(\sqrt{D_1})\widetilde \phi_1    \Bigr)\,.
\]
For the second term, 
\[
d(\sqrt{D_1})e^{-i(t-R_2) \sqrt{D_1}}d^{-1}(\sqrt{D_1})q_n(\sqrt{D_1})\widetilde \phi_1    =e^{-i(t-R_2) \sqrt{D_1}}q_n(\sqrt{D_1})\widetilde \phi_1 \,.
\]
In the first one, we can not commute $d(\sqrt{D_1})$ with $e^{-i(t-R_2)\sqrt{D}}$. However, we can apply lemma \ref{lek}. Indeed, $d(\sqrt \alpha)$ is continuous and decays at infinity. Consider \[
e^{-i(t-R_2) \sqrt{D_1}}d^{-1}(\sqrt{D_1})q_n(\sqrt{D_1})\widetilde \phi_1\,.\]
Take any two sequences $\{t^{(j)}\}$ and $\{R_2^{(j)}\}$ that converge to infinity such that \mbox{$\lim_{j\to\infty}(t^{(j)}-R_2^{(j)})=+\infty$.} Then, estimate \eqref{zuzu}, applied with $\tau=t^{(j)}-R_2^{(j)}$ and arbitrary $R_1$, shows that 
\[
e^{-i(t^{(j)}-R_2^{(j)}) \sqrt{D_1}}d^{-1}(\sqrt{D_1})q_n(\sqrt{D_1})\widetilde \phi_1(t^{(j)})
\]
satisfies conditions of lemma \ref{lek} and we can commute the operators in the limit which gives us
\[
\lim_{j\to\infty}\|e^{-i(t^{(j)}-R_2^{(j)}) \sqrt{D}}q_n(\sqrt{D})\widetilde \phi_1(t^{(j)})- e^{-i(t^{(j)}-R_2^{(j)}) \sqrt{D_1}}q_n(\sqrt{D_1})\widetilde \phi_1(t^{(j)})\|_2=0\,.
\]
Since $\{t^{(j)}\}, \{R_2^{(j)}\}$ are arbitrary, we have the statement of the lemma.  
\end{proof}
{\bf Remark.} Since we proved the lemma for approximants $q_n(\sqrt{D})\widetilde \phi_1(t)$ with arbitrary $\rho>0$, we have 
\begin{equation}\label{lisk}
\lim_{R_2\to\infty}\limsup_{t\to\infty}\| e^{-i(t-R_2)\sqrt{D}}  \phi(t)       - e^{-i(t-R_2)\sqrt{D_1}}    \phi(t)      \|_2=0\,.
 \end{equation}

\item[(B.3)] {\it We can use 
\[
\lim_{R_1\to\infty}\|V^{(R_1)}\|=0
\]
to make sure that $ e^{-it\sqrt{D_1}}\phi(t)$ is close  to $f$ in $L^2(\mathbb{R}^3)$.}

More precisely, from lemma \ref{stability}, we get

\begin{equation}\label{hera}
\lim_{R_1\to\infty}\limsup_{t\to\infty}\|e^{-it\sqrt{D_1}}e^{it\sqrt{H_0}}g-g\|_2=0\,,
\end{equation}
where, again, $D_1=-{\rm div}\,(1+V^{(R_1)})\nabla$ and $g\in L^2(\mathbb{R}^3)$.

\item[(B.4)]{\it  Now, we use the so-called intertwining property.} Fixing $R_2$ and taking $g=e^{iR_2\sqrt{H_0}}f$ in \eqref{hera}, we get

\begin{equation}\label{lisk1}
\lim_{R_1\to\infty}\limsup_{t\to\infty}\|e^{-i(t-R_2)\sqrt{D_1}}\phi(t)-e^{iR_2\sqrt{H_0}}f\|_2=0
\end{equation}
for every $R_2$.

\item[(B.5)] {\it Compare (B.2) (in particular, \eqref{lisk})   with (B.4), \eqref{lisk1}, to conclude that $e^{-i(t-R_2)\sqrt{D}}\phi(t)$ satisfies
 \begin{equation}\label{lastik}
\lim_{R_2\to\infty}\limsup_{t\to\infty}\|e^{-i(t-R_2)\sqrt{D}}\phi(t)-e^{iR_2\sqrt{H_0}}f\|_2=0\,.
\end{equation}}

\item[(B.6)] Notice that the  formula \eqref{lastik} is equivalent to 
\begin{equation}\label{heret}
\lim_{R_2\to\infty}\limsup_{t\to\infty}\|e^{-it\sqrt{D}}\phi(t)-e^{-iR_2 \sqrt{D}}\left(  e^{iR_2\sqrt{H_0}}f  \right)\|_2=0\,.
\end{equation}
Consider the second term. We can write
\[
e^{-iR_2 \sqrt{D}} e^{iR_2\sqrt{H_0}}f=e^{-iR_2 \sqrt{D}} q_n(\sqrt{D}) \widetilde \phi_1(R_2)+\epsilon_2\,,
\]
where $\|\epsilon_2\|_2\to 0$ when $R_2$ and $\rho$ go to infinity (see  \eqref{zo2}).

 We can now write  formula \eqref{formula1}
\[
e^{-iR_2 \sqrt{D}} q_n(\sqrt{D})  \widetilde \phi_1(R_2)=
\frac{1}{2\pi i}\int_{\Gamma_{a,b}} e^{-iR_2k} (D-k^2)^{-1}p_n(k)  \widetilde \phi_1(R_2)    dk\,.
\]
 The estimates on the amplitude give us
\begin{eqnarray}
\left\|\frac{1}{2\pi i}\int_{\Gamma} e^{-iR_2k} (D-k^2)^{-1}p_n(k) \widetilde \phi_1(R_2)   dk\right\|_{L^2(S_r(0))}\le C(f,\rho, R_2,a,b)\int_{\Gamma_{a,b}} e^{R_2\Im k}e^{-|\Im k|r} \frac{|p_n(k)|d|k|}{|\Im k|^{1.5}}<\nonumber\\
C(f,\rho,R_2,n,a,b)(r-R_2)^{-(n-0.5)}\nonumber\,,
\end{eqnarray}
provided that $r>R_2+1$. This finally provides an estimate
\begin{equation}\label{herwe}
\|e^{-iR_2 \sqrt{D}} q_n(\sqrt{D}) \widetilde \phi_1(R_2)\|^2_{L^2(B^c_R(0))}<C(f,R_2,\rho,n,a,b)R^{-(2n-4)}
\end{equation}
which holds for  $R>2R_2$. Thus, fixing $R_2$ and choosing $R$ large, we can make $C(f,\rho, R_2,a,b)R^{4-2n}$ as small as we wish since $n=3$ was our choice for $n$.\bigskip

Now, the claim (B) is proved. Indeed, given any $\epsilon>0$, we chose $\rho$ to make approximation error in $\widetilde \phi_1$  smaller than $\epsilon$ when $\limsup_{t\to\infty}$ is taken. Then, we choose $R_2$ large enough to make the left hand side in \eqref{heret} smaller than $\epsilon$. Finally, we choose $R$ so that the right hand side in \eqref{herwe} is smaller than $\epsilon$.
\end{itemize}
\end{proof}

{\bf Remark.} Notice that when $t\to +\infty$, the integral over $\Gamma^-_{a,b}$  does not contribute anything. If $t\to -\infty$, the roles of $\Gamma^+_{a,b}$ and $\Gamma^-_{a,b}$ change. \bigskip

\subsection{Stationary representation for wave operators and orthogonal eigenfunction decomposition}

We start with a theorem.
\begin{theorem}If $f\in \cal{N}$, then 
\begin{equation}\label{kili}
(W^-f)(y)=\frac{1}{(2\pi)^3}\int_0^\infty d\kappa |\kappa|^2 \int_{\mathbb{S}^2}\overline{a_\infty(\sigma,y,-\kappa)}\widecheck f(\kappa\sigma/(2\pi))d\sigma
\end{equation}
for every $y$.\label{pr_sv}
\end{theorem}

\begin{proof}

From \eqref{limki} and part (B) (tightness), we know that 
\[
\lim_{t\to+\infty}\left\|(e^{-it\sqrt{D}}q_n(\sqrt{D})\widetilde\phi_1)(y)- \frac{1}{2\pi i} \int_{\Gamma^+_{a,b}} p_n(k) \cal{G}(k,y)dk\right\|_{L^2(\mathbb{R}^3)}=0\,,
\]
where $\widetilde\phi_1$ approximates $\phi_1$ in $L^2(\mathbb{R}^3)$ as $\rho\to\infty$. 
The integral defines a function continuous in $y$. We first fix $y$ and $\rho$ and then  use formula for $\cal{G}$, analyticity of $\overline{a_\infty(\sigma,y,-\bar k)}$, and \eqref{contour}   to replace the contour $\Gamma_{a,b}^+$ by $\overrightarrow{[b,a]}$ and write 
\begin{equation}\label{yellow1}
 \frac{1}{2\pi i} \int_{\Gamma^+_{a,b}} p_n(k) \cal{G}(k,y)dk= \frac{i}{2\pi } \int_{[a,b]}  p_n(\kappa) \cal{G}(\kappa,y)d\kappa\,.
\end{equation}
The boundary value of $\cal{G}$ on the real line is understood in the  $L^2 [a,b]$ topology  for every $y$ and every $[a,b]$. 

In the formula for  $\cal{G}$, the interior integral is equal to
\[
 \int_{\mathbb{R}} 
\omega_\rho(\tau)e^{i\kappa \tau} (-\sigma\cal{R}(\nabla f_1)(\tau,\sigma)+i\cal{R}(|\Delta|^{1/2}f_1)(\tau,\sigma)d\tau\,.
\]
It converges to one-dimensional inverse Fourier transform of  $(-\sigma\cal{R}(\nabla f_1)(\tau,\sigma)+i\cal{R}(|\Delta|^{1/2}f_1)(\tau,\sigma)$ in $\tau$, evaluated at point $\kappa/2\pi$.
Since the function $f_1\in \cal{N}$, this convergence is uniform in $\kappa\in [a,b],\sigma \in \mathbb{S}^2$. 

Take $\rho\to\infty$ in the right hand side of \eqref{yellow1}. Using the formulas
\[
\widehat{\cal{R}}(s,\sigma)=\widehat f(s\sigma), \,\widecheck{\cal{R}}(s,\sigma)=\widecheck f(s\sigma)
\]
(see \cite{ss}, p. 204), we conclude (recalling the definitions of $p_n$ and $q_n$) that
\[
(W^-f)(y)=(2\pi)^{-3}
\int_0^\infty d\kappa |\kappa|^2 \int_{\mathbb{S}^2}\overline{a_\infty(\sigma,y,- \kappa)}\widecheck f(\kappa\sigma/(2\pi))d\sigma\,.
\]
Notice that we obtained this formula for every fixed $y$ and the right hand side of \eqref{kili} belongs to domain of $D$ which is $\cal{H}^2(\mathbb{R}^3)\subset C(\mathbb{R}^3)$.
\end{proof}
{\bf Remark.}  In the free case, $a_\infty^0=e^{-ik\langle \sigma,y\rangle}$ so
\[
\frac{1}{8\pi^3}\int_0^\infty dk |k|^2 \int_{\mathbb{S}^2}\overline{a_\infty^0(\sigma,y,-\overline k)}\widecheck f(k\sigma/(2\pi))d\sigma=
\]
\[
\frac{1}{8\pi^3}\int_0^\infty dk |k|^2 \int_{\mathbb{S}^2} e^{-ik\langle \sigma,y\rangle }\widecheck f(k\sigma/(2\pi))d\sigma=\frac{1}{8\pi^3}\int_{\mathbb{R}^3} e^{-i\langle \xi,y\rangle} \widecheck{f}(\xi/(2\pi))d\xi=f(y)
\]
by Fourier inversion formula, as expected. \smallskip

{\bf Remark. } One can get an analogous formula for $W^+$  (compare with \cite{yafaev1}, formula (6.9), p.~247)
\begin{equation}\label{voln_o}
(W^+f)(y)=\frac{1}{(2\pi)^3}
\int_0^\infty d\kappa |\kappa|^2 \int_{\mathbb{S}^2}\overline{a_\infty(\sigma,y,\kappa)}\widehat f(\kappa\sigma/(2\pi))d\sigma\,.
\end{equation}

 \smallskip

For every $y\in \mathbb{R}^3$, we define the following functions
\[
\cal{A}^-(\xi,y)\dd  \overline{a_\infty(\sigma,y,-2\pi\kappa)},\quad \cal{A}^+(\xi,y)\dd  \overline{a_\infty(\sigma,y,2\pi \kappa)},\,\quad \xi\in \mathbb{R}^3
\]
and $ \xi=|\kappa|\sigma$ is representation of $\xi$ in spherical coordinates. Clearly, 
\[
\int_{r_1<|\xi|<r_2} |\cal{A}^{\pm}(\xi,y)|^2d\xi<\infty
\]
for every $0<r_1<r_2<\infty$.

If $g\in C_c^\infty(\mathbb{R}^3)$ and its support has a positive distance from the origin, we can define the map
\[
(\cal{U}^-g)(y)=\int_{\mathbb{R}^3} \cal{A}^-(\xi,y)g(\xi)d\xi,\quad (\cal{U}^+g)(y)=\int_{\mathbb{R}^3} \cal{A}^+(\xi,y)g(\xi)d\xi\,.
\]
The following theorem gives the stationary representation for wave operators
\begin{theorem} $\cal{U}^{\pm}$ are isometries from $L^2(\mathbb{R}^3)$ to $L^2(\mathbb{R}^3)$ and $W^{\pm}=\cal{U}^{\pm}\cal{F}^{\mp 1}$\,.
\end{theorem}
\begin{proof}
This is immediate since $W^{\pm}$ are isometries and Fourier transform is unitary from $L^2(\mathbb{R}^3)$ to $L^2(\mathbb{R}^3)$. Thus, we can extend $\cal{U}^\pm$ to all of $L^2(\mathbb{R}^3)$.
\end{proof}

Now that we established that $\{\cal{A}^{\pm}\}$ form orthonormal systems,
the next natural question is: do $\{\cal{A}^{\pm}\}$, as functions in $y$, represent eigenfunctions of $D$ in some sense (check formula (6.7), p. 246, \cite{yafaev1})? 
In fact, we have (with $''\cdot''$ below indicating the variable on which the operator $D$ acts)
\begin{equation}\label{ruka}
D a_\infty(\sigma,\cdot,\kappa)=\kappa^2   a_\infty(\sigma,\cdot,\kappa)
\end{equation}
in the following weak sense. 

Take any test function $\phi\in C_c^\infty(\mathbb{R}^3), k\in \mathbb{C}^+$, and $h\in L^2(\mathbb{S}^2)$ and consider $r$ so large that the support of $\phi$ is inside $B_r(0)$. Then, we can write 
\[
\langle D \phi,\left(\frac{1}{r^2}\int_{S_r(0)} G(\cdot,x,(-\overline{k})^2)h(\widehat x)d\sigma_x\right)\rangle=\langle\phi,D\left(\frac{1}{r^2}\int_{S_r(0)} G(\cdot,x,(-\overline{k})^2)h(\widehat x)d\sigma_x\right)\rangle=
\]
\[
k^2\langle\phi,\left(\frac{1}{r^2}\int_{S_r(0)} G(\cdot,x,(-\overline{k})^2)h(\widehat x)d\sigma_x\right)\rangle\,.
\]
Notice that  $G(y,x,(-\overline{k})^2)=\overline{G(x,y,k^2)}$. Substitute this identity into the previous formula, send $r\to\infty$ and compare the main terms in asymptotics. This provides
\[
\langle D \phi,\int_{\mathbb{S}^2} \overline{a_\infty(\sigma,\cdot,k)} h(\sigma)d\sigma\rangle =k^2\langle  \phi,\int_{\mathbb{S}^2} \overline{a_\infty(\sigma,\cdot,k)} h(\sigma)d\sigma\rangle\,.
\]
We now take $\kappa$ for which the non-tangential limits of both sides exist (each one is a full measure set). Comparing the limiting values, we get \eqref{ruka} in ``weak'' sense, which can be formulated as  (the conjugation can be dropped by the choice of $h$ and $\phi$)
\begin{lemma} Take any $\phi\in C_c^\infty(\mathbb{R}^3)$ and $h\in L^2(\mathbb{S}^2)$. Then, for a.e. $\kappa\in \mathbb{R}$, we have
\[
\langle D \phi,\int_{\mathbb{S}^2} {a_\infty(\sigma,\cdot,{\kappa})} h(\sigma)d\sigma\rangle =\kappa^2\langle  \phi,\int_{\mathbb{S}^2} {a_\infty(\sigma,\cdot,{\kappa})} h(\sigma)d\sigma\rangle\,.
\]
\end{lemma}
{\bf Remark.} The obvious drawback of the given argument is that the set of ``good'' $\kappa$ for which the non-tangential limits exist, might depend on both $\phi$ and $h$. Had we been able to establish the $y$-independence of the set of ``good'' $\kappa$ in the definition of $a_\infty(\sigma,y,\kappa)$, we would have had
\[
\int_{\mathbb{S}^2} {a_\infty(\sigma,y,\kappa)}h(\sigma)d\sigma
\]
being a weak (and then $\cal{H}^2(\mathbb{R}^3)$-regular, by Interior Regularity Theorem, \cite{evans}, p.309) eigenfunction of $D$. Notice, that we do not have this issue in the case when the problem is considered on $\ell^2(\mathbb{Z}^3)$.\bigskip

Having obtained the stationary representation for wave operators, we can now prove theorem \ref{il5} from Introduction. 

\begin{proof}{\it (of theorem \ref{il5})}
First, we suppose that for given $f$, its orthogonal projection to subspace ${\rm ran}\,W^+$ is nonzero. Call it $h_1$ and write $h_2=f-h_1$. Since $h_1=W^+(W^+)^{-1}h_1$, we have
\[
e^{it\sqrt D}e^{-it\sqrt H_0}(W^+)^{-1}h_1\to h_1, \quad t\to+\infty
\]
and, since $e^{it\sqrt D}$ preserves $L^2$ norm, 
\[
\|e^{-it\sqrt H_0}(W^+)^{-1}h_1-e^{-it\sqrt D}h_1\|_2\to 0, \quad t\to+\infty
\] 
thus proving the lemma.

Therefore, it suffices to show that orthogonal projection of $f$ to ${\rm ran} \,W^+$ is nonzero. Suppose it is zero.
Then, using $\langle f, W^+g\rangle=0$ for all $g\in L^2(\mathbb{R}^3)$, formula \eqref{voln_o} and weak continuity (as function in $\sigma,\kappa$) of $a_\infty(\sigma,y,\kappa)$ in parameter $y$, we get
\begin{equation}\label{voln_1}
\int_{\mathbb{R}^3} f(y)
\int_0^\infty d\kappa |\kappa|^2 \int_{\mathbb{S}^2}{a_\infty(\sigma,y,\kappa)}\overline{\widehat g(\kappa\sigma/(2\pi))}d\sigma=0\,
\end{equation}
for every $g\in \cal{N}$. This can be rewritten as
\begin{equation}\label{p0b}
\int_0^\infty d\kappa |\kappa|^2 \int_{\mathbb{S}^2}
F(\sigma,\kappa)\overline{\widehat g(\kappa\sigma/(2\pi))}d\sigma=0,\,\quad F(\sigma,\kappa)\dd \int_{\R^3}f(y)a_\infty(\sigma,y,\kappa)dy
\end{equation}
if we change the order of integration. Now, note that 
$F(\sigma,k)$ is analytic in $k\in \mathbb{C}^+$ and positive on $i\R$ since $f$ is nonnegative and $a_\infty(\sigma,y,k)$ is positive there. Thus, $F(\sigma,\kappa)$, being boundary value, is not identically zero as function in $\sigma\in \mathbb{S}^2$ and $\kappa\in \mathbb{R}$. This, however, gives a contradiction  with \eqref{p0b} since $g$ is chosen arbitrarily in $\cal{N}$ and $\cal{N}$ is dense in $L^2(\R^3)$.
\end{proof}

We finish this part by formulating the following questions:

\begin{itemize}
\item[1.] Are the wave operators $W^{\pm}(\sqrt{D},\sqrt{H_0})$ complete?
\item[2.] Can methods developed in this paper be generalized to Schr\"odinger evolution $e^{it H}$? The free evolution for Schr\"odinger equation is very different from $e^{it\sqrt H_0}$ and proving existence of wave operators is a major challenge even in one-dimensional case.

\end{itemize}\vspace{0.5cm}

\section{Examples}

In the last section, we want to consider the large class of potentials, for which the conditions \eqref{main-assump} and \eqref{2-assu} are satisfied. 
  In many cases, if potential $V$  oscillates and decays at infinity, it can be written in the form $V={\rm div}\, Q$. \smallskip

\begin{itemize} 

\item
 Take 
$
Q(x)=q(|x|)P(x)\,,
$
where $P$ is any $C^2(\mathbb{R}^3)$ vector-field satisfying 
$
\sum_{j=0}^2 \|D^jP\|_\infty<\infty
$
and $q\in C^2(\mathbb{R}^+)$ and $q,q',q''\in L^2(\mathbb{R}^+)$. $\Bigl(\Bigr.$ For instance, take  $P$ as any trigonometric polynomial in $x$ and let $q(r)=(r^2+1)^{-\gamma},\, \gamma>1/4$. Then, 
$
V=q(|x|)\,{\rm div}\, P+V_1\,,
$
where $V_1$ is short-range.$\Bigl.\Bigr)$\smallskip

\item 

Following \cite{den1}, consider the random model. Take any $\phi$ which is infinitely smooth function supported in $B_1(0)$. Consider
\[
V_0=\sum_{j\in \mathbb{N}} a_j \phi(x-x_j)\,,
\]
where $\{x_j\}$ are points in $\mathbb{R}^3$ that satisfy $\min_{j_1\neq j_2} |x_{j_1}-x_{j_2}|\ge 2$ (e.g., one can take the elements of the lattice $2\mathbb{Z}^3$). Then, choose $\{a_j\}$ in such a way that 
\[
|V_0(x)|\lesssim (1+|x|)^{-1/2-\epsilon}, \epsilon>0.
\]

Now, consider $V$ in \eqref{n11} or \eqref{n12} given by ``randomization'' of $V_0$, i.e.,
\begin{equation}\label{rp}
V(x)=\sum_{j\in \mathbb{N}} a_j \xi_j \phi(x-x_j)\,,
\end{equation}
where $\{\xi_j\}$ are real-valued, bounded, and odd independent random variables. In \cite{den1}, it was proved that $V$ can be written in the form $V={\rm div}\, Q$ where $Q$ satisfies \eqref{2-assu} almost surely.

The idea was based on writing the formula
\[
V=\Delta\Delta^{-1} V=-{\rm div}\nabla_x\int_{\mathbb{R}^3}\frac{V(y)}{4\pi|x-y|}dy={\rm div} Q, \quad Q(x)=\int_{\mathbb{R}^3} \frac{x-y}{4\pi|x-y|^3}V(y)dy
\]
and proving that $Q$ satisfies $|Q(x)|\le C (1+|x|)^{-1/2-\epsilon_1}, \epsilon_1>0$ with probability 1.  This implies, in particular, that theorem \ref{main-th} holds true almost surely.

In \cite{den1}, it was proved that the operators $H=-\Delta+V$ with potential given by \eqref{rp} satisfies $\sigma_{ac}(H)=[0,\infty)$ almost surely. The multidimensional random models with slow decay were considered in 
\cite{fs,rods} (on $\mathbb{R}^\nu, \nu\ge 2$) and \cite{bour1} (on $\mathbb{Z}^2$) and existence of wave operators was proved. In the current paper, we go beyond establishing a.c. spectral type (the main result in \cite{den1}) by showing that the wave operators \eqref{wowo} exist. In contrast to  \cite{bour1} and \cite{rods}, we proved deterministic results and then showed that the random potential satisfies the conditions  of the theorem almost surely.

\end{itemize}

\end{document}